\documentclass{article}
\usepackage[journal=DM, lang=american]{ems-journal}

\usepackage[scr=stixtwofancy, cal=stixplain, bb=esstix]{mathalfa}
\usepackage{bm}
\usepackage{tikz-cd}
\usetikzlibrary{calc}
\tikzcdset{arrow style=math font}
\usepackage{url}
\usepackage{etoolbox}
\usepackage{hyperref}

\numberwithin{equation}{section}

\newtheorem{thm}{Theorem}[section]
\newtheorem{lemma}[thm]{Lemma}
\newtheorem{cor}[thm]{Corollary}
\newtheorem{prop}[thm]{Proposition}

\newtheorem{introthm}{Theorem}

\theoremstyle{definition}
\newtheorem{rk}[thm]{Remark}
\newtheorem{constr}[thm]{Construction}
\newtheorem{warn}[thm]{Warning}
\newtheorem{ex}[thm]{Example}

\newtheorem*{claim*}{Claim}
\AtBeginEnvironment{claim*}{\let\oldqed=\qedsymbol%
\renewcommand\qedsymbol{$\blacktriangle$\kern-1.2pt\null}}
\AtEndEnvironment{claim*}{\let\qedsymbol=\oldqed}

\newtheorem{defi}[thm]{Definition}

\newcommand{\nerve}{\textup{N}}
\newcommand{\h}{\textup{h}}
\newcommand{\cat}[1]{\textbf{\textup{#1}}}
\newcommand{\op}{{\textup{op}}}
\newcommand{\blank}{{\textup{--}}}
\newcommand{\pr}{{\textup{pr}}}
\newcommand{\Hom}{{\textup{Hom}}}
\newcommand{\id}{{\textup{id}}}
\newcommand{\ev}{{\textup{ev}}}
\newcommand{\im}{\mathop{\textup{im}}}
\newcommand{\incl}{{\textup{incl}}}

\newcommand{\Fun}{\textup{Fun}}
\newcommand{\Inj}{\textup{Inj}}
\newcommand{\supp}{\mathop{\textup{supp}}\nolimits}
\newcommand{\Maps}{\mathord{\textup{maps}}}
\newcommand{\Ob}{\mathop{\textup{Ob}}}
\newcommand{\myh}{\mathord{\textup{`$\mskip-.1\thinmuskip h\mskip-.4\thinmuskip$'}}}

\newcommand{\forget}{\mathop{\textup{forget}}\nolimits}
\newcommand{\Alg}{\textup{Alg}}
\newcommand{\grpdfy}{\textup{grpdfy}}
\newcommand{\discr}{\textup{discr}}
\newcommand{\core}{\mathop{\textup{core}}}
\newcommand{\const}{\textup{const}}
\newcommand{\CMon}{\textup{CMon}}
\newcommand{\ppo}{\mathbin{\raise.25pt\hbox{\scalebox{.67}{$\square$}}}}
\newcommand{\Ex}{{\textup{Ex}}}
\newcommand{\hocolim}{\mathop{\textup{hocolim}}\nolimits}
\newcommand{\NERVE}{{\cat{N}}}
\newcommand{\diag}{{\textup{diag}}}
\newcommand{\Sing}{{\textup{Sing}}}

\newcommand{\Mor}{\mathop{\textup{Mor}}}
\newcommand{\src}{\mathop{\textup{src}}}

\let\setminus=\smallsetminus
\let\phi=\varphi
\let\del=\partial

\def\twocell[#1]{\arrow[#1, dash, phantom, "\Rightarrow"{scale=1.125, yshift=-.4pt, description, allow upside down, sloped, inner sep=0pt}]}

\begin{document}
\title{Genuine vs.~na\"ive symmetric monoidal $G$-categories}

\emsauthor{1}{
	\givenname{Tobias}
	\surname{Lenz}
    \mrid{1292331}}{Tobias Lenz}
\Emsaffil{1}{
	\department1{Mathematisches Institut}
	\organisation1{Rheinische Friedrich-Wilhelms-Universität Bonn}
	\rorid1{041nas322}
	\address1{Endenicher Allee 60}
	\zip1{53115}
	\city1{Bonn}
	\country1{Germany}

	\department2{Mathematisch Instituut}
	\organisation2{Universiteit Utrecht}
	\rorid2{04pp8hn57}
	\address2{Budapestlaan 6}
	\zip2{3584$\,$CD}
	\city2{Utrecht}
	\country2{The Netherlands (\textit{current address})}
	\affemail2{t.lenz@uu.nl}}

\classification[55P48, 19D23, 55U35]{55P91}

\begin{abstract}
	We prove that through the eyes of equivariant \emph{weak} equivalences the genuine symmetric monoidal $G$-categories of Guillou and May [Algebr. Geom. Topol. 17 (2017), no. 6, 3259--3339] are equivalent to just ordinary symmetric monoidal categories with $G$-action. Along the way, we give an operadic model of global infinite loop spaces and provide an equivalence between the equivariant category theory of genuine symmetric monoidal $G$-categories and the $G$-parsummable categories studied by Schwede [J. Topol. 15 (2022), no. 3, 1325--1454] and the author [New York
	J. Math. 29 (2023), 635--686].
\end{abstract}

\makeatletter
\renewcommand*\ps@titlepage{%
	\let\@oddfoot\@empty
	\let\@evenfoot\@empty
	\def\@oddhead{\ems@titlehead{\cr}{\cr}}
	\let\@evenhead\@oddhead
}
\makeatother

\maketitle
\tableofcontents

\section*{Introduction}
\emph{Commutative monoids up to coherent homotopy} play an important role in algebraic topology, not least because of their close connection to stable homotopy theory: every coherently commutative monoid can be delooped to a connective spectrum, and this construction is an important tool to obtain stable homotopy types from data of a more algebraic nature, in particular featuring prominently in the modern construction of the algebraic $K$-theory of rings \cite{may-permutative}. Over the years, several equivalent approaches to the subject have been studied, in particular May's operadic approach in terms of \emph{$E_\infty$-algebras} \cite{may-loop-spaces}, Segal's theory of \emph{special $\Gamma$-spaces} \cite{segal-gamma}, and various `ultra-commutative' models \cite{sagave-schlichtkrull}.

Equivarianty, i.e.~with respect to the action of a finite group $G$, the theory becomes more subtle: namely, if one wants such \emph{$G$-equivariantly coherently commutative monoids} to deloop to \emph{genuine} $G$-spectra in the sense of equivariant stable homotopy theory, one has to encode additional algebraic structure in the form of additive norms, intuitively corresponding to certain `twisted' sums. In the operadic approach, this leads to the notion of \emph{genuine $G$-$E_\infty$-operads} and their algebras, the \emph{genuine $G$-$E_\infty$-algebras}. However, in many practical contexts the latter are much harder to construct than their non-equivariant counterparts: in particular, if we equip an ordinary $E_\infty$-algebra with a $G$-action, even a trivial one, this does not yield a genuine $G$-$E_\infty$-algebra in any natural way, but only a so-called \emph{na\"ive} one. Similarly, when one wants to consider equivariant generalizations of algebraic $K$-theory, one faces the obstacle that typically the inputs one wants to consider only come to us as na\"ive symmetric monoidal $G$-categories, which are (pseudo)algebras over a certain $E_\infty$-operad $E\Sigma_*$ with trival $G$-action, as opposed to the desired \emph{genuine symmetric monoidal $G$-categories}.

To circumvent this issue, Guillou and May \cite{guillou-may} introduced a general procedure (following Shimakawa) that builds a genuine symmetric monoidal $G$-category from a na\"ive one via some sort of Borel construction. However, as remarked in \cite{guillou-may} and reiterated in \cite{merling}, these are actually the only examples they were able to construct, which naturally leads to the question whether other genuine symmetric monoidal $G$-categories exist. While Guillou and May originally expected this to be the case, we prove in this paper:

\begin{introthm}[see Theorem~\ref{thm:main}]\label{introthm:main}
The Guillou-May-Shimakawa construction induces an equivalence between the quasi-category of na\"ive symmetric monoidal $G$-categories (with respect to a certain explicit notion of weak equivalence) and the quasi-category of genuine symmetric monoidal $G$-categories (with respect to the $G$-equivariant \emph{weak} equivalences, i.e.~functors inducing weak homotopy equivalences on nerves of fixed points).
\end{introthm}

In fact, Guillou and May more generally work with (genuine and na\"ive) symmetric monoidal $G$-categories internal to the category of topological spaces, and we also prove the analogue of Theorem~\ref{introthm:main} in this context, see Theorem~\ref{thm:main-topological}. Note however that both of these crucially rely on using the $G$-equivariant \emph{weak} equivalences---for the $G$-equivariant equivalences (functors inducing \emph{equivalences of categories} on fixed points) there are both trivial and non-trivial examples of genuine symmetric monoidal $G$-categories not arising via their construction, see Remarks~\ref{rk:counterexamples} and~\ref{rk:counterexample-interesting}. However, at least for applications to equivariant algebraic $K$-theory the $G$-weak equivalences are fine enough; in particular, Theorem~\ref{introthm:main} allows us to derive a version of the main result of \cite{g-global} for genuine symmetric monoidal $G$-categories, which generalizes a non-equivariant result due to Thomason \cite{thomason}:

\begin{introthm}[see Theorem~\ref{thm:guillou-may-k-theory}]\label{introthm:gm-thomason}
The Guillou-May construction of equivariant algebraic $K$-theory exhibits the quasi-category of connective genuine $G$-spectra as a quasi-localiza\-tion of the category of genuine symmetric monoidal $G$-categories.
\end{introthm}

We also prove a `non-group completed' version of this in the spirit of Mandell's non-equivariant result \cite{mandell}, see Theorem~\ref{thm:g-equiv-mandell}.

\subsection*{Global and \texorpdfstring{$\bm G$}{G}-global \texorpdfstring{$\bm{E_\infty}$}{E-infinity}-algebras}
Our proof of Theorem~\ref{introthm:main} uses the language of \emph{$G$-global homotopy theory} in the sense of \cite{g-global}. Intuitively speaking, $1$-global (typically referred to simply as \emph{global}) homotopy theory \cite{schwede-book} studies equivariant phenomena that exist universally across all suitable groups, while $G$-global homotopy theory generalizes this to the presence of an additional twist in the form of a group action and can be viewed as a synthesis of ordinary global and $G$-equivariant homotopy theory.

The main portion of the present paper is then devoted to developing a $G$-global version of the theory of $E_\infty$-algebras and connecting this to the $\Gamma$-space- and ultra-commutative approaches studied in \cite{g-global}, which we hope to be interesting in its own right, and which already gives new results in the global case $G=1$. In particular, we introduce \emph{$G$-global $E_\infty$-operads} and we prove the following comparison, refining an equivariant result due to May, Merling, and Osorno \cite{may-merling-osorno}:

\begin{introthm}[see Theorem~\ref{thm:genuine-e-infty-vs-gamma}]\label{introthm:genuine-e-infty-vs-gamma}
There exists an equivalence between the quasi-category of special $G$-global $\Gamma$-spaces \cite{g-global} and the quasi-category of $\mathcal O$-algebras for any $G$-global $E_\infty$-operad $\mathcal O$.
\end{introthm}

In fact, we develop the whole theory both for algebras in simplicial sets as well as algebras in categories. As an upshot, this then allows us to give a new model for the \emph{category theory} of genuine symmetric monoidal $G$-categories---which unlike the respective homotopy theories differs from the one for na\"ive symmetric monoidal $G$-categories---in terms of the \emph{$G$-parsummable categories} studied in \cite{schwede-k-theory,sym-mon-global}. These represent a rather different approach to `coherent commutativity,' similar to the `ultra-commutative' philosophy of \cite{sagave-schlichtkrull,schwede-book}; somewhat loosely speaking, we can think of them as $G$-categories equipped with a \emph{strictly} equivariant, unital, associative, and commutative, but only \emph{partially} defined operation.

\begin{introthm}[see Theorem~\ref{thm:genuine-sym-vs-parsum}]\label{introthm:genuine-sym-mon-vs-parsum-cat}
There exists an equivalence between the quasi-category of genuine symmetric monoidal $G$-categories and the quasi-category of $G$-parsummable categories, both formed with respect to the $G$-equivariant equivalences of categories.
\end{introthm}

On the pointset level, $G$-parsummable categories actually possess more structure than just genuine symmetric monoidal $G$-categories in that unlike the latter they already model the corresponding \emph{$G$-global} theory, see Theorem~\ref{thm:taming-I-algebras}. Nevertheless, they are arguably easier to construct, as we illustrate by an example related to the algebraic $K$-theory of groups rings (see Section~\ref{sec:k-theory-group-rings}).

\subsection*{Related work} Recently, Barrero \cite{barrero} studied a notion of global $E_\infty$-algebras in Schwede's orthogonal space model of global homotopy theory \cite{schwede-book}. Our approach presented here differs from his treatment in basically two ways:

Firstly, we work in a different model based on the `universal finite group.' While this requires us to restrict to finite groups (as opposed to all compact Lie groups), this model is necessary to develop the corresponding theory for categories, and moreover this is what allows us to construct genuine $G$-global $E_\infty$-algebras of a more combinatorial or algebraic nature.

Secondly, and more importantly, we consider all operads and algebras with respect to the Cartesian product as opposed to the so-called box product used by Barrero. While this comes at the cost of various `$\Sigma$-freeness' conditions on the operads in question, this is more in line with the usual (non-)equivariant approach, and it is moreover in some sense one step further removed from the ultra-commutative model of global coherent commutativity originally studied by Schwede in \cite{schwede-book}: for example, the analogous non-equivariant comparison between $E_\infty$-algebras in the usual sense and a certain notion of ultra-commutative monoids due to Sagave and Schlichtkrull \cite{sagave-schlichtkrull} proceeds through $E_\infty$-algebras with respect to the box product.

Both of these changes actually faciliate the comparison with the equivariant theory developed by Guillou and May and in particular are important for the proofs of Theorems~\ref{introthm:main} and~\ref{introthm:gm-thomason}.

\subsection*{Strategy and outline}
While it seems natural to approach Theorems~\ref{introthm:main} and~\ref{introthm:gm-thomason} via model categorical techniques, one soon faces the issue that the Thomason style model structure on $\cat{$\bm G$-Cat}$ from \cite{gcat} is not well-behaved monoidally; in fact, already for $G=1$ it is not known to me whether it transfers to categories of operadic algebras, or whether these even carry any model structure at all whose weak equivalences are the $G$-weak equivalences.

On the other hand, as we will see below the \emph{$G$-equivariant equivalences} interact as nicely with the Cartesian product as one could wish, which in particular allows us to construct transferred model structures for operadic algebras---sadly, however, Theorem~\ref{introthm:main} is simply no longer true with respect to this finer notion of equivalence.

We will solve this dilemma by pushing the comparison with respect to the $G$-equivariant equivalences of categories as far as we can (exploiting all the model categorical techniques available in this setting) and only switching to the $G$-equivariant \emph{weak} equivalences in the very end. As a consequence of this approach, the present paper can be roughly divided into two parts (separated by a short interlude in the form of Section~\ref{sec:k-theory-group-rings}): the first part is mostly model categorical in nature and establishes the general theory of operadic algebras with respect to the $G$-equivariant equivalences of categories, as well as the corresponding simplicial theory; here we in particular prove Theorems~\ref{introthm:genuine-e-infty-vs-gamma} and~\ref{introthm:genuine-sym-mon-vs-parsum-cat}. The second part is then devoted to the proofs of Theorems~\ref{introthm:main} and~\ref{introthm:gm-thomason}, where we have to work with bare categories with weak equivalences as well as the quasi-categories they represent. In more detail:

In Section~\ref{sec:reminder} we give a brief reminder on unstable $G$-global homotopy theory, before developing the basic theory of $G$-global operads and their algebras. Section~\ref{sec:equivariant-cat} is then devoted to the analogous theory in the world of categories.

In Section~\ref{sec:g-glob-vs-g-equiv} we compare the $G$-global and $G$-equivariant approaches. Afterwards, we relate the operadic models to the models we studied in \cite{g-global} in Section~\ref{sec:parsummability}, in particular proving Theorems~\ref{introthm:genuine-e-infty-vs-gamma} and \ref{introthm:genuine-sym-mon-vs-parsum-cat}. As an application of these results, we construct a genuine symmetric monoidal $G$-category whose equivariant algebraic $K$-theory captures the usual algebraic $K$-theory of group rings in Section~\ref{sec:k-theory-group-rings}.

Section~\ref{sec:cat-vs-simpl} explains the relation between the categorical and simplicial models. Afterwards, we prove Theorems~\ref{introthm:main} and~\ref{introthm:gm-thomason} in Section~\ref{sec:main-thm} building on the model categorical comparisons established in earlier sections as well as our work in \cite{sym-mon-global, g-global}. Finally, Section~\ref{sec:internal} generalizes our main results to categories internal to spaces.

\section{\texorpdfstring{$G$}{G}-global homotopy theory}\label{sec:reminder}
\subsection{A reminder on \texorpdfstring{$\bm G$}{G}-global spaces} In this section we will set up the theory of \emph{$G$-global operads} and the corresponding algebras. We begin by giving a very brief reminder on (unstable) \emph{$G$-global homotopy theory} in the sense of \cite[Chapter~1]{g-global}. The approach we will take for this is based on a certain simplicial monoid $E\mathcal M$ that we call the \emph{universal finite group}.

\begin{defi}
We set $\omega\mathrel{:=}\{0,1,\dots\}$, and we write $\mathcal M=\Inj(\omega,\omega)$ for the monoid of self-injections of $\omega$ (with monoid structure given by composition).
\end{defi}

Using the indiscrete category functor $E\colon\cat{Set}\to\cat{Cat}$ (i.e.~the right adjoint to the functor $\Ob\colon\cat{Cat}\to\cat{Set}$) we get a category $E\mathcal M$, and as $E$ preserves products, this inherits a monoid structure. We will also write $E\mathcal M$ for the simplicial monoid $\nerve(E\mathcal M)$.

\begin{defi}
A finite subgroup $H\subset\mathcal M=(E\mathcal M)_0$ is \emph{universal} if the induced $H$-action on $\omega$ makes the latter into a \emph{complete $H$-set universe}, i.e.~every countable $H$-set embeds equivariantly into $\omega$.
\end{defi}

One can show that any finite group $H$ admits an injective homomorphism $H\to\mathcal M$ with universal image, and that any two such homomorphisms are conjugate \cite[Lemma~1.2.8]{g-global}, i.e.~any abstract finite group is isomorphic to a universal subgroup of $E\mathcal M$ in an essentially unique way.

Moreover, the universal subgroups are closed under subgroups and conjugation \cite[Corollary~1.2.7]{g-global}, i.e.~they form a so-called \emph{family of subgroups}.

\begin{thm}\label{thm:g-global-model-sset}
There is a unique model structure on the category $\cat{$\bm{E\mathcal M}$-$\bm G$-SSet}$ of simplicial sets with $(E\mathcal M\times G)$-action in which a map is a weak equivalence or fibration if and only if $f^\phi$ is a weak homotopy equivalence or Kan fibration, respectively, for every universal subgroup $H\subset\mathcal M$ and every homomorphism $\phi\colon H\to G$; here we write $(\blank)^\phi$ for the fixed points with respect to the \emph{graph subgroup} $\Gamma_{H,\phi}\mathrel{:=}\{(h,\phi(h)):h\in H\}$.

We call this the \emph{$G$-global model structure} and its weak equivalences the \emph{$G$-global weak equivalences}. This model structure is proper, simplicial, and combinatorial with generating cofibrations
\begin{equation*}
\{ (E\mathcal M\times_\phi G)\times (\del\Delta^n\hookrightarrow\Delta^n) : \text{$H\subset\mathcal M$ universal, $\phi\colon H\to G$, $n\ge0$}\}
\end{equation*}
and generating acyclic cofibrations
\begin{equation*}
\{ (E\mathcal M\times_\phi G)\times (\Lambda^n_k\hookrightarrow\Delta^n) : \text{$H\subset\mathcal M$ universal, $\phi\colon H\to G$, $0\le k\le n$}\}
\end{equation*}
where $E\mathcal M\times_\phi G\mathrel{:=} (E\mathcal M\times G)/\Gamma_{H,\phi}$. Moreover, filtered colimits and finite products in it are homotopical.
\begin{proof}
See \cite[Corollary~1.2.33 and Lemma~1.1.3]{g-global}.
\end{proof}
\end{thm}

\begin{rk}
Schwede originally studied unstable \emph{global homotopy theory} using a model in terms of orthogonal spaces \cite[Chapter~1]{schwede-book}. While his approach contains equivariant information for all compact Lie groups, the more combinatorial models we employ here (and in particular the use of categories later) force us to restrict to finite groups. Other than that, $1$-global homotopy theory in our sense recovers usual global homotopy theory: for $G=1$ the above is equivalent to Schwede's orthogonal spaces localized with respect to a certain natural notion of `$\mathcal F\textit{in}$-global weak equivalences,' see \cite[Section~1.5]{g-global}.
\end{rk}

In fact, the above is an instance of a more general construction of model structures for actions of \emph{simplicial monoids}, i.e.~(strict) monoids in the category $\cat{SSet}$ or, equivalently, simplicial objects in the category of monoids:

\begin{thm}\label{thm:equiv-model-structure-sset}
Let $M$ be a simplicial monoid and let $\mathcal F$ be a collection of finite subgroups of the ordinary monoid $M_0$. Then there exists a unique model structure on the category $\cat{$\bm M$-SSet}$ of simplicial sets with $M$-action in which a map $f$ is a weak equivalence or fibration if and only if $f^H$ is a weak homotopy equivalence or Kan fibration, respectively, for every $H\in\mathcal F$. We call this the \emph{$\mathcal F$-model structure} and its weak equivalences the \emph{$\mathcal F$-weak equivalences}. It is simplicial, proper, and combinatorial with generating cofibrations
\begin{equation*}
\{ M/H\times(\del\Delta^n\hookrightarrow\Delta^n) : H\in\mathcal F,n\ge0\}
\end{equation*}
and generating acyclic cofibrations
\begin{equation*}
\{ M/H\times(\Lambda^n_k\hookrightarrow\Delta^n) : H\in\mathcal F,0\le k\le n\}.
\end{equation*}
Moreover, filtered colimits and finite products in this model structure are homotopical.
\begin{proof}
See \cite[Proposition~1.1.2 and Lemma~1.1.3]{g-global}.
\end{proof}
\end{thm}

\begin{rk}\label{rk:g-universal-sset}
The $G$-global weak equivalences and fibrations in fact only depend on the action of the discrete group $\core(\mathcal M)\times G$ (where $\core$ denotes the maximal subgroup), i.e.~the $G$-global model structure is transferred from an analogously defined model structure on $\cat{$\bm{\core(\mathcal M)}$-$\bm G$-SSet}$. This model structure will become relevant again later in Section~\ref{sec:g-glob-vs-g-equiv} and we will refer to it as the \emph{$G$-universal model structure} and its weak equivalences as the \emph{$G$-universal weak equivalences}. Beware however that this does \emph{not} model $G$-global homotopy theory.
\end{rk}

\begin{warn}\label{warn:not-monoidal}
If $M=G$ is a discrete group and $\mathcal F$ is a family of subgroups, then one easily checks that the above is a \emph{monoidal model category} with respect to the Cartesian product, i.e.~the cofibrations and acyclic cofibrations satisfy the pushout product axiom and for any cofibrant replacement $Q\to *$ of the terminal object and any (cofibrant) $X$ the projection $Q\times X\to X$ is a weak equivalence. However, for general (simplicial) monoids $M$, $\cat{$\bm M$-SSet}$ will typically \emph{not} be monoidal, even in the non-equivariant (or `projective') case where $\mathcal F$ only consists of the trivial group: namely, the $M$-simplicial set $M\times M$ is usually not cofibrant.

As a concrete example, let $M=\{0,1\}$ with monoid operation $a*b\mathrel{:=}\max\{a,b\}$. All maps in the above set $I$ of generating cofibrations except $\varnothing\to M$ are isomorphisms in degree $0$, so any $I$-cell complex $X$ splits in degree $0$ as $X_0\cong\coprod_{k\in K} M$. By Quillen's Retract Argument it then in particular follows that for any cofibrant $X'$ the set $X_0'$ of $0$-simplices admits an equivariant embedding into a set of the form $\coprod_{k\in K'}M$. However, every $(x_1,x_2)\in M\times M$ satisfies $1.(x_1,x_2)=(\max\{1,x_1\},\max\{1,x_2\})=(1,1)$, so any equivariant map $M\times M\to\coprod_{k\in K'}M$ has to factor through one of the coproduct summands and hence cannot be injective for cardinality reasons.
\end{warn}

Let us now turn to the cofibrations of the above model structures. In the case that $M=G$ is a discrete group, there is a classical characterization of the cofibrations of the $\mathcal F$-model structure, see e.g.~\cite[Proposition~2.16]{cellular}:

\begin{lemma}\label{lemma:charact-cof-sset}
Let $G$ be a discrete group and let $\mathcal F$ be a family of (finite) subgroups of $G$. Then a map $f$ is a cofibration in the $\mathcal F$-model structure on $\cat{$\bm G$-SSet}$ if and only if $f$ is an injective cofibration (i.e.~a cofibration of underlying simplicial sets) and moreover any simplex not in the image of $f$ has isotropy contained in $\mathcal F$.\qed
\end{lemma}

In particular, if $G$ is finite and $\mathcal F=\mathcal A\ell\ell$ is the collection of all subgroups, then the cofibrations are precisely the injective cofibrations, while for general $\mathcal F$ (and in particular for general simplicial monoids) there are far fewer cofibrations. However, we can still combine the injective cofibrations and the $\mathcal F$-weak equivalences into an \emph{injective} (or `mixed') model structure that will be useful at several points below:

\begin{thm}
Let $M$ be a simplicial monoid and let $\mathcal F$ be any collection of finite subgroups of $M_0$. Then there exists a unique model structure on $\cat{$\bm M$-SSet}$ whose weak equivalences are the $\mathcal F$-weak equivalences and whose cofibrations are the injective cofibrations. We call this the \emph{injective $\mathcal F$-model structure}. It is combinatorial, simplicial, and proper.

In particular, $\cat{$\bm{E\mathcal M}$-$\bm G$-SSet}$ admits a proper, simplicial, and combinatorial model structure whose weak equivalences are the $G$-global weak equivalences and whose cofibrations are the underlying cofibrations of simplicial sets. We call this the \emph{injective $G$-global model structure}.
\begin{proof}
See \cite[Proposition~1.1.15]{g-global}.
\end{proof}
\end{thm}

\subsection{Functoriality} In the study of operadic actions below, we will at several places need to know how the above model structures relate to each other as the monoid varies. In order to formulate these results, we first have to recall the following notion:

\begin{defi}
Let $M,N$ be simplicial monoids. We write $\mathcal G_{M,N}$ for the collection of all graph subgroups of $M_0\times N_0$, i.e.~all subgroups of the form $\Gamma_{H,\phi}=\{(h,\phi(h)):h\in H\}$ with $H\subset M_0$ and $\phi\colon H\to N_0$. More generally, if $\mathcal F$ is a collection of subgroups of $M_0$, then we write $\mathcal G_{\mathcal F,N}$ for the collection of all graph subgroups $\Gamma_{H,\phi}$ with $H\in\mathcal F$.
\end{defi}

Throughout let $M$ be a simplicial monoid and let $\mathcal F$ be a family of finite subgroups of $M_0$. All of the following results are well-known at least for groups:

\begin{lemma}\label{lemma:alpha-shriek-sset}
Let $\alpha\colon H\to G$ be any group homomorphism. Then
\begin{equation*}
\alpha_!\colon \cat{$\bm{(M\times H)}$-SSet}_{\mathcal G_{\mathcal F,H}}\rightleftarrows\cat{$\bm{(M\times G)}$-SSet}_{\mathcal G_{\mathcal F,G}} :\!\alpha^*=(M\times\alpha)^*
\end{equation*}
is a Quillen adjunction with homotopical right adjoint.
\begin{proof}
See \cite[Lemma~1.1.16 and Example~1.1.21]{g-global}.
\end{proof}
\end{lemma}

In particular, specializing to $M=E\mathcal M$ and $\mathcal F$ the family of universal subgroups, we get a Quillen adjunction
\begin{equation*}
\alpha_!\colon\cat{$\bm{E\mathcal M}$-$\bm H$-SSet}_\text{$H$-global}\rightleftarrows\cat{$\bm{E\mathcal M}$-$\bm G$-SSet}_\text{$G$-global} :\!\alpha^*.
\end{equation*}

\begin{lemma}\label{lemma:free-quotients-sset}
Let $G$ be any discrete group and assume $f\colon X\to Y$ is a $\mathcal G_{\mathcal F,G}$-weak equivalence in $\cat{$\bm{(M\times G)}$-SSet}$ such that $G$ acts freely on both $X$ and $Y$. Then $f/G\colon X/G\to Y/G$ is an $\mathcal F$-weak equivalence.
\begin{proof}
See \cite[Proposition~1.1.22]{g-global}.
\end{proof}
\end{lemma}

\begin{lemma}\label{lemma:alpha-star-sset}
Let $\alpha\colon H\to G$ be an \emph{injective} homomorphism of discrete groups. Then
\begin{equation*}
\alpha^*\colon\cat{$\bm{(M\times G)}$-SSet}_{\mathcal G_{\mathcal F,G}}\rightleftarrows\cat{$\bm{(M\times H)}$-SSet}_{\mathcal G_{\mathcal F,H}} :\!\alpha_*
\end{equation*}
is a Quillen adjunction; moreover, if $\im(\alpha)$ has finite index in $G$, then $\alpha_*$ is fully homotopical.
\begin{proof}
See \cite[Proposition~1.1.19 and Example~1.1.21]{g-global}.
\end{proof}
\end{lemma}

\begin{cor}\label{cor:twisted-products-sset}
Let $G$ be any discrete group, let $n\ge 0$, and let $f\colon X\to Y$ be a $\mathcal G_{\mathcal F,G}$-weak equivalence in $\cat{$\bm{(M\times G)}$-SSet}$. Then $f^{\times n}\colon X^{\times n}\to Y^{\times n}$ is a $\mathcal G_{\mathcal F,G\times\Sigma_n}$-weak equivalence in $\cat{$\bm{(M\times G\times\Sigma_n)}$-SSet}$ with respect to the $\Sigma_n$-action permuting the factors.
\begin{proof}
This is a formal consequence of the previous results, also cf.~\cite[proof of Corollary~1.4.71]{g-global}:

Replacing $M$ by $M\times G$ and $\mathcal F$ by $\mathcal G_{\mathcal F,G}$, it suffices to prove the corresponding statement for $\cat{$\bm M$-SSet}$ (i.e.~where $G=1$). This is trivial for $n=0$; if $n\ge1$, we denote by  $\Sigma_n^1\subset\Sigma_n$ the subgroup of permutations fixing $1$, and we write $i\colon \Sigma_n^1\hookrightarrow\Sigma_n$ for the inclusion and $p\colon \Sigma_n^1\to 1$ for the unique homomorphism. It is then straight-forward to check that $(\blank)^{\times n}\colon\cat{$\bm M$-SSet}\to\cat{$\bm{(M\times\Sigma_n)}$-SSet}$ factors up to isomorphism as the composite
\begin{equation*}
\cat{$\bm M$-SSet}\xrightarrow{p^*}\cat{$\bm{(M\times\Sigma_n^1)}$-SSet}\xrightarrow{i_*}\cat{$\bm{(M\times\Sigma_n)}$-SSet},
\end{equation*}
so that the claim follows from Lemmas~\ref{lemma:alpha-shriek-sset} and~\ref{lemma:alpha-star-sset}.
\end{proof}
\end{cor}

\subsection{Equivariant simplicial operads} Next, we will study operads in $\cat{$\bm M$-SSet}$ with respect to the Cartesian symmetric monoidal structures (i.e.~operads with an $M$-action) as well as their algebras. While the general theory we develop here works for all simplicial monoids, we will be particularly interested in the cases where $M=E\mathcal M\times G$ or $M=G$ for a (finite) group $G$, where we will refer to the corresponding operads as \emph{$G$-global operads} or \emph{$G$-equivariant operads}, respectively.

I actually expect the results in this subsection to be known to experts, at least for discrete groups; however, as I am not aware of a place where these results appear in the literature, I have decided to give full proofs. This will at the same time allow us to already see several of the arguments we will employ in the categorical setting later without some of the technical baggage necessary there.

\begin{constr}
Let $\mathcal O$ be an operad in $\cat{$\bm M$-SSet}$. The category $\Alg_{\mathcal O}(\cat{$\bm M$-SSet})$ of $\mathcal O$-algebras in $\cat{$\bm M$-SSet}$ comes with a forgetful functor $\forget\colon\Alg_{\mathcal O}(\cat{$\bm M$-SSet})\to\cat{$\bm M$-SSet}$ and it is naturally enriched, tensored, and cotensored over $\cat{SSet}$ so that this forgetful functor strictly preserves cotensors.

The forgetful functor has a left adjoint, which we denote by $\cat P$; explicitly, this is given by $\cat{P}X=\coprod_{n\ge 0} \mathcal O(n)\times_{\Sigma_n} X^{\times n}$ with the evident functoriality in $X$ and with $\mathcal O$-algebra structure induced by operad structure maps of $\mathcal O$. The unit of the adjunction is given by the composition $X\to\mathcal O(1)\times X\hookrightarrow\forget\cat{P}X$ where the first map is induced by the inclusion of the identity element of $\mathcal O(1)$, while the second one is the inclusion of the summand indexed by $1$.

By \cite[Proposition~3.7.10]{cathtpy} there is then a unique way to make $\cat{P}$ into a simplicially enriched functor such that $\cat{P}\dashv\forget$ is a simplicially enriched adjunction, and with respect to this enrichment $\cat{P}$ preserves tensors.
\end{constr}

\begin{thm}
Let $\mathcal O$ be any operad in $\cat{$\bm M$-SSet}$. Then $\Alg_{\mathcal O}(\cat{$\bm M$-SSet})$  carries a unique model structure in which a map is a weak equivalence or fibration if and only if its image under the forgetful functor $\forget\colon\Alg_{\mathcal O}(\cat{$\bm M$-SSet})\to\cat{$\bm M$-SSet}$ is a weak equivalence or fibration, respectively, in the $\mathcal F$-model structure.

We call this model structure the \emph{$\mathcal F$-model structure} again. It is combinatorial, simplicial, and right proper. Moreover, filtered colimits in it are homotopical.

Finally, we have a Quillen adjunction
\begin{equation}\label{eq:free-forget-sset-QA}
\cat{P}\colon\cat{$\bm M$-SSet}\rightleftarrows\Alg_{\mathcal O}(\cat{$\bm M$-SSet}):\!\forget.
\end{equation}
\end{thm}
For $M=G$ a discrete group, the corresponding statement for algebras in orthogonal $G$-spectra (and suitable operads $\mathcal O$) can be found as \cite[Proposition~A.1]{blumberg-hill}.
\begin{proof}
While we cannot directly apply the results of \cite{berger-moerdijk} since $\cat{$\bm M$-SSet}$ is typically not a monoidal model category (see Warning~\ref{warn:not-monoidal}), a similar strategy works in our case:

Namely, $\Alg_{\mathcal O}(\cat{$\bm M$-SSet})$ is locally presentable, so it will be enough by Quillen's Path Object Argument \cite[2.6]{berger-moerdijk} to show that $\Alg_{\mathcal O}(\cat{$\bm M$-SSet})$ admits a fibrant replacement functor (i.e.~an endofunctor $P$ together with a natural transformation $\iota\colon\id\Rightarrow P$ such that $\forget PX$ is fibrant and $\forget \iota_X$ is a weak equivalence for every $X$) as well as functorial path objects for fibrant objects (i.e.~for every fibrant $X$ a factorization $X\to X^I\to X\times X$ of the diagonal into a weak equivalence followed by a fibration that is functorial in maps of fibrant objects).

For the first statement we fix a fibrant replacement functor $\id\Rightarrow P$ on $\cat{SSet}$ such that $P$ preserves finite limits, for example Kan's $\Ex^\infty$-functor with the natural transformation $e\colon\id\Rightarrow\Ex^\infty$ \cite{kan-ex} or the unit of the geometric realization-singular set adjunction $\cat{SSet}\rightleftarrows\cat{Top}$. As $P$ preserves products, it lifts to a functor $\cat{$\bm M$-SSet}\to\cat{$\bm{P(M)}$-SSet}$ and then to an endofunctor of $\cat{$\bm M$-SSet}$ by restricting the action along $M\to P(M)$. As $P$ preserves finite limits (hence in particular fixed points for finite groups) this lift is then a fibrant replacement functor for the $\mathcal F$-model structure on $\cat{$\bm M$-SSet}$. Using again that $P$ preserves products, this then lifts to the desired fibrant replacement on $\Alg_{\mathcal O}(\cat{$\bm M$-SSet})$.

For the second statement, we simply observe that the standard path object in $\cat{SSet}$
\begin{equation*}
X\xrightarrow{\const}\Maps(\Delta^1,X)\xrightarrow{(\ev_0,\ev_1)} X\times X
\end{equation*}
preserves all (finite) limits, so arguing precisely as before this lifts to provide functorial path objects for fibrant objects in $\cat{$\bm M$-SSet}$ and $\Alg_{\mathcal O}(\cat{$\bm M$-SSet})$.

This completes the proof that the $\mathcal F$-model structure exists and is cofibrantly generated, hence combinatorial. The remaining statements follow easily from the corresponding statements for $\cat{$\bm M$-SSet}$ and the fact that $(\ref{eq:free-forget-sset-QA})$ is a simplicial adjunction.
\end{proof}

\begin{ex}\label{ex:g-e-infty-sset}
Let $G$ be a finite group. A \emph{na\"ive $G$-$E_\infty$-operad} is an operad $\mathcal O$ in $\cat{$\bm G$-SSet}$ such that
\begin{equation*}
\mathcal O(n)^H\simeq\begin{cases}
* & H\subset G\times1\\
\varnothing & \text{otherwise}
\end{cases}
\end{equation*}
for all $H\subset G\times\Sigma_n$; here we have turned the right $\Sigma_n$-action on $\mathcal O(n)$ coming from the operad structure into a left action as usual. In particular, any $E_\infty$-operad in the usual non-equivariant sense becomes a na\"ive $G$-$E_\infty$-operad when equipped with the trivial $G$-action. The corresponding algebras are called \emph{na\"ive $G$-$E_\infty$-algebras}.

As the name suggests (and alluded to in the introduction), na\"ive $G$-$E_\infty$-algebras are usually not the objects one wants to study in equivariant homotopy theory (unless $G=1$). For example, from the point of view of \emph{equivariant infinite loop spaces} \cite{guillou-may}, the `grouplike' na\"ive $G$-$E_\infty$-algebras only come with deloopings against the ordinary spheres $S^1,S^2,\dots$ as opposed to deloopings against all representation spheres. Instead, we are interested in \emph{genuine $G$-$E_\infty$-algebras}, which are algebras over so-called \emph{genuine $G$-$E_\infty$-operads}. Here a {genuine $G$-$E_\infty$-operad} is an operad $\mathcal P$ in $\cat{$\bm G$-SSet}$ such that
\begin{equation*}
\mathcal P(n)^H\simeq\begin{cases}
* & \text{if $H\in\mathcal G_{G,\Sigma_n}$}\\
\varnothing & \text{otherwise}
\end{cases}
\end{equation*}
for all $H\subset G\times\Sigma_n$, also see \cite[Definition~2.1]{guillou-may}.
\end{ex}

If $f\colon\mathcal O\to\mathcal P$ is any map of operads, then it is clear from the definitions that the restriction $f^*\colon\Alg_{\mathcal P}(\cat{$\bm M$-SSet})\to\Alg_{\mathcal O}(\cat{$\bm M$-SSet})$ preserves weak equivalences, fibrations, limits, and filtered colimits as these are all created in $\cat{$\bm M$-SSet}$. Appealing to the Special Adjoint Functor Theorem, we therefore get a Quillen adjunction $f_!\dashv f^*$. We are now interested in the question when this Quillen adjunction is a Quillen equivalence, for which we first have to talk about a suitable notion of weak equivalence of operads. As Example~\ref{ex:g-e-infty-sset} suggests, this will be finer than just the levelwise weak equivalences and take the $\Sigma_n$-actions into account:

\begin{defi}
A map $f\colon\mathcal O\to\mathcal P$ of operads in $\cat{$\bm M$-SSet}$ is called an \emph{$\mathcal F$-weak equivalence} if $\mathcal O(n)\to\mathcal P(n)$ is a $\mathcal G_{\mathcal F,\Sigma_n}$-weak equivalence for every $n\ge0$.
\end{defi}

However, already non-equivariantly categories of algebras are typically not invariant under weak equivalences between general operads, or, put differently, strict pointset level algebras only turn out to be the correct thing to study for suitably nice operads. This leads to the notion of \emph{$\Sigma$-cofibrancy} \cite[Remark~3.4]{berger-moerdijk}, demanding that each $\mathcal O(n)$ be cofibrant in the projective model structure on $\Sigma_n$-objects. However, the operads of interest in the equivariant setting are typically \emph{not} $\Sigma$-cofibrant---for example no genuine $G$-$E_\infty$-operad is. Instead we will use the following condition:

\begin{defi}\label{defi:sigma-free-sset}
An operad $\mathcal O$ in $\cat{$\bm M$-SSet}$ is called \emph{$\Sigma$-free} if $\Sigma_n$ acts freely on $\mathcal O(n)$ for every $n\ge0$.
\end{defi}

\begin{ex}
Let $G$ be a finite group again. An operad $\mathcal O$ in $\cat{$\bm G$-SSet}$ is a genuine $G$-$E_\infty$-operad in the sense of Example~\ref{ex:g-e-infty-sset} if and only if it is $\Sigma$-free and the unique map $\mathcal O\to*$ to the terminal operad is a $G$-weak equivalence (i.e.~an $\mathcal F$-weak equivalence for $\mathcal F=\mathcal A\ell\ell$ the collection of all subgroups of $G$).
\end{ex}

We can now introduce one of the central notions of this paper:

\begin{defi}
A $G$-global operad $\mathcal O$ is called a \emph{$G$-global $E_\infty$-operad} if it is $\Sigma$-free and the unique map $\mathcal O\to *$ is a $G$-global weak equivalence, i.e.~for every $n\ge0$ the map $\mathcal O(n)\to *$ is a $(G\times\Sigma_n)$-global weak equivalence.
\end{defi}

Below we will show that $\mathcal F$-weak equivalences of $\Sigma$-free operads induce Quillen equivalences between their model categories of algebras. A standard way to prove this in purely model categorical language would proceed via a `cell induction' argument to reduce the claim to free algebras. However, pushouts in categories of algebras are quite complicated, and while the appendix of \cite{berger-moerdijk-corrected} provides an explicit description, proving the comparison along these lines would become quite involved computationally. Instead, we will use $\infty$-categorical language to recast this reduction argument into much simpler form using mon\-adicity:

\begin{prop}\label{prop:free-monadic-sset}
Let $\mathcal O$ be an operad in $\cat{$\bm M$-SSet}$. Then the functor
\begin{equation*}
\forget^\infty\colon\allowbreak\Alg_{\mathcal O}(\cat{$\bm M$-SSet})^\infty_{\mathcal F}\to\cat{$\bm M$-SSet}_{\mathcal F}^\infty
\end{equation*}
induced by the forgetful functor on associated quasi-categories is conservative and preserves $\Delta^\op$-shaped homotopy colimits. In particular, the adjunction $\cat{L}\cat{P}\dashv{\forget}^\infty$ induced by the Quillen adjunction $(\ref{eq:free-forget-sset-QA})$ is monadic.
\end{prop}

The proof will rely on the following standard observation, also cf.~\cite[Theorem~12.2]{may-loop-spaces} or \cite[Proposition~2.1.7]{schwede-book} for similar results:

\begin{prop}\label{prop:forget-geometric-realization}
Let $\mathscr C$ be a cocomplete category with finite products that is in addition enriched, tensored, and cotensored over $\cat{SSet}$. Assume further that the geometric realization functor $\Fun(\Delta^\op,\mathscr C)\to\mathscr C$ (given by the coend of the tensoring) preserves finite products. Then the forgetful functor $\Alg_{\mathcal O}(\mathscr C)\to\mathscr C$ preserves geometric realizations for every operad $\mathcal O$ in $\mathscr C$.
\begin{proof}
We write $\nerve\colon\mathscr C\to\Fun(\Delta^\op,\mathscr C)$ for the right adjoint of $|\blank|$ given by $(\nerve X)_n=X^{\Delta^n}$ with the evident functoriality in each variable, and we write $\mathbb N$ for the right adjoint of geometric realization in $\Alg_{\mathcal O}(\mathscr C)$ defined analogously. The claim then amounts to saying that the canonical mate of the left hand square in
\begin{equation*}
\begin{tikzcd}[cramped]
{\Fun(\Delta^\op,\Alg_{\mathcal O}(\mathscr C))}\arrow[d,"{\Fun(\Delta^\op\!,\forget)}"']\twocell[dr, "\scriptstyle\id"{yshift=-7pt},yshift=3.5pt] &\arrow[l, "\mathbb N"'] \Alg_{\mathcal O}(\mathscr C)\arrow[d,"\forget"]\\
\Fun(\Delta^\op,\mathscr C)&\arrow[l,"\nerve"]\mathscr C
\end{tikzcd}\qquad
\begin{tikzcd}[cramped]
{\Alg_{\mathcal O}(\Fun(\Delta^\op,\mathscr C))}\arrow[d,"\forget"']\twocell[dr, "\scriptstyle\id"{yshift=-7pt},yshift=3.5pt] &[1.5em]\arrow[l, "\Alg_{\mathcal O}(\nerve)"'] \Alg_{\mathcal O}(\mathscr C)\arrow[d,"\forget"]\\
\Fun(\Delta^\op,\mathscr C)&\arrow[l,"\nerve"]\mathscr C
\end{tikzcd}
\end{equation*}
is an isomorphism (here we used that the forgetful functor strictly preserves cotensors by construction). Using the canonical isomorphism $\Fun(\Delta^\op,\Alg_{\mathcal O}(\mathscr C))\cong\Alg_{\mathcal O}(\Fun(\Delta^\op,\mathscr C))$ and the compatibility of mates with pasting, it is then enough to show this for the right hand square.

For this we then observe that $|\blank|$ lifts to a functor $\Alg_{\mathcal O}(\Fun(\Delta^\op,\mathscr C))\to\Alg_{\mathcal O}(\mathscr C)$ as it preserves products, and so do the unit and counit of the adjunction $|\blank|\dashv\nerve$, inducing an adjunction $\Alg_{\mathcal O}(|\blank|)\dashv\Alg_{\mathcal O}(\nerve)$. However, with respect to these choices the canonical mate of the right hand square is even the identity (by the triangle identity for adjunctions), which completes the proof of the proposition.
\end{proof}
\end{prop}

\begin{proof}[Proof of Proposition~\ref{prop:free-monadic-sset}]
It suffices to prove the first statement; the second one will then follow by the $\infty$-categorical Barr-Beck Theorem \cite[Theorem~4.7.3.5]{ha}.

Conservativity is clear since weak equivalences in $\cat{$\bm M$-SSet}$ and $\Alg_{\mathcal O}(\cat{$\bm M$-SSet})$ are saturated and since $\forget$ creates weak equivalences. For the remaining part, we observe that $\Delta^\op$-shaped homotopy colimits on both sides can be computed as geometric realizations of a Reedy cofibrant replacement (of a chosen strictification) by \cite[Corollary~A.2.9.30]{htt}. On the other hand, geometric realization in $\cat{$\bm M$-SSet}$ is just given by taking diagonals (see e.g.~\cite[Proposition~B.1]{bousfield-friedlander}), hence fully homotopical by \cite[Lemma~1.2.57]{g-global}. Thus, Proposition~\ref{prop:forget-geometric-realization} shows that also geometric realization of $\mathcal O$-algebras is fully homotopical, i.e.~$\Delta^\op$-shaped homotopy colimits can be computed by ordinary geometric realization in either category. The claim then follows from another application of Proposition~\ref{prop:forget-geometric-realization}.
\end{proof}

We can now finally prove:

\begin{thm}\label{thm:change-of-operad-sset}
Let $f\colon\mathcal O\to\mathcal P$ be an $\mathcal F$-weak equivalence of $\Sigma$-free operads in $\cat{$\bm M$-SSet}$. Then the Quillen adjunction
\begin{equation}\label{eq:change-of-operad-sset}
f_!\colon\Alg_{\mathcal O}(\cat{$\bm M$-SSet})_{\mathcal F}\rightleftarrows\Alg_{\mathcal P}(\cat{$\bm M$-SSet})_{\mathcal F} :\!f^*
\end{equation}
is a Quillen equivalence.
\end{thm}
For $M=G$ a finite group, an alternative proof using the bar construction is mentioned in \cite[discussion after Remark~4.11]{guillou-may}. The corresponding statement for (suitable) algebras in genuine $G$-spectra also appears without proof as \cite[Theorem~A.3]{blumberg-hill}.
\begin{proof}
It is clear that $(\ref{eq:change-of-operad-sset})$ is a a Quillen adjunction with homotopical right adjoint, so it suffices that $(f^*)^\infty$ is an equivalence of quasi-categories. For this we consider the diagram
\begin{equation*}
\begin{tikzcd}
\Alg_{\mathcal P}(\cat{$\bm M$-SSet})^\infty_{\mathcal F}\arrow[r, "(f^*)^\infty"]\arrow[d, "\forget^\infty"'] & \Alg_{\mathcal O}(\cat{$\bm M$-SSet})_{\mathcal F}^\infty\arrow[d, "\forget^\infty"]\\
\cat{$\bm M$-SSet}_{\mathcal F}^\infty\twocell[ur]\arrow[r, equal] & \cat{$\bm M$-SSet}_{\mathcal F}^\infty
\end{tikzcd}
\end{equation*}
which commutes up to the natural isomorphism induced by the identity transformation $\forget\Rightarrow \forget\circ f^*$. As both vertical functors are monadic (Proposition~\ref{prop:free-monadic-sset}), it suffices by \cite[Corollary 4.7.3.16]{ha} that the canonical mate $\cat{LP}_{\mathcal O}\Rightarrow(f^*)^\infty\circ\cat{LP}_{\mathcal P}$ of the above transformation is an equivalence. Unravelling definitions, this amounts to saying that for each (cofibrant) $X\in\cat{$\bm M$-SSet}$ the map
\begin{equation*}
\coprod_{n\ge 0}\mathcal O(n)\times_{\Sigma_n} X^n\to\coprod_{n\ge 0} \mathcal P(n)\times_{\Sigma_n} X^n
\end{equation*}
induced by $f$ is an $\mathcal F$-weak equivalence. As $\Sigma_n$ acts freely on both $\mathcal O(n)$ and $\mathcal P(n)$, this is in turn immediate from Lemma~\ref{lemma:free-quotients-sset}.
\end{proof}

\begin{rk}
A standard trick, see e.g.~\cite[discussion after Remark~4.11]{guillou-may}, shows that any two genuine $G$-$E_\infty$-operads $\mathcal O,\mathcal P$ are weakly equivalent through $\Sigma$-free operads: namely, it suffices to consider the zig-zag $\mathcal O\gets\mathcal O\times\mathcal P\to\mathcal P$ given by the projections. Thus, the previous theorem implies that the model categories of $\mathcal O$- and $\mathcal P$-algebras are Quillen equivalent.

Similarly one shows that any two $G$-global $E_\infty$-operads have Quillen equivalent categories of algebras. In Section~\ref{sec:parsummability} we will compare these further to the models of `$G$-globally coherently commutative monoids' previously studied in \cite[Chapter~2]{g-global}.
\end{rk}

\section{\texorpdfstring{$G$}{G}-global category theory}\label{sec:equivariant-cat}
In this section we want to develop the analogue of the above theory for categories, and in particular we will introduce $G$-global model structures on $\cat{$\bm{E\mathcal M}$-$\bm G$-Cat}$ and on categories of algebras over operads in it.

For this we first recall that the category $\cat{Cat}$ of small categories carries a \emph{canonical model structure} \cite{rezk-cat} whose weak equivalences are the equivalences of categories, whose cofibrations are those functors that are injective on objects, and whose fibrations are the \emph{isofibrations}, i.e.~those functors with the right lifting property against the inclusion $*\to E\{0,1\}$ of either object. This model structure is proper, Cartesian, and combinatorial. Moreover, the functor $\grpdfy\circ\h\colon \cat{SSet}\to\cat{Cat}$ sending a simplicial set to its fundamental groupoid is left Quillen (with right adjoint given by taking the nerve of the maximal subgroupoid) and it preserves finite products. Thus, any $\cat{Cat}$-enriched model category becomes a simplicial model category by transporting the enrichment, tensoring, and cotensoring along this adjunction; in particular, $\cat{Cat}$ itself becomes a simplicial model category.

One of the key objects of study in this paper is the following equivariant generalization of the canonical model structure:

\begin{thm}\label{thm:equiv-model-structure}
Let $M$ be a categorical monoid (i.e.~a strict monoidal category) and let $\mathcal F$ be a collection of finite subgroups of $\Ob(M)$. Then there is a unique model structure on the category $\cat{$\bm M$-Cat}$ of small categories with strict $M$-action in which a map $f$ is a weak equivalence or fibration if and only if $f^H$ is an equivalence of categories or isofibration, respectively, for each $H\in\mathcal F$. We call this the \emph{$\mathcal F$-model structure} and its weak equivalences the \emph{$\mathcal F$-equivalences}. It is right proper, $\cat{Cat}$-enriched (hence simplicial), and combinatorial with generating cofibrations
\begin{equation*}
\{ M/H\times i : H\in\mathcal F,i\in I_{\cat{Cat}}\}
\end{equation*}
and generating acyclic cofibrations
\begin{equation*}
\{ M/H\times j : H\in\mathcal F,j\in J_{\cat{Cat}}\}
\end{equation*}
for arbitrary choices of generating (acyclic) cofibrations $I_{\cat{Cat}},J_{\cat{Cat}}$ of $\cat{Cat}$.

Finally, the $\mathcal F$-equivalences are stable under filtered colimits and arbitrary products.
\begin{proof}
Let us first show that the model structure exists, which we will do by transferring along the functor $\big((\blank)^H)_{H\in\mathcal F}\colon \cat{$\bm M$-Cat}\to\prod_{H\in\mathcal F}\cat{Cat}$ (with $\cat{Cat}$-enriched left adjoint given by sending $(C_H)_{H\in\mathcal F}$ to $\coprod_{H\in\mathcal F}M/H\times C_H$).

As $\cat{$\bm M$-Cat}$ is locally presentable and every object in it is fibrant, it will be enough by Quillen's Path Object Argument to construct functorial path objects. Just as in the simplicial setting, these can be obtained from the usual path objects
\begin{equation*}
C\xrightarrow{\const}\Fun(E\{0,1\},C)\xrightarrow{(\ev_0,\ev_1)} C\times C
\end{equation*}
in $\cat{Cat}$ by pulling through the $M$-actions; here we used that the canonical model structure is Cartesian. This completes the proof of the existence of the model structure and shows that it is combinatorial with the above generating (acyclic) cofibrations.

The model structure is right proper and $\cat{Cat}$-enriched because it is transferred from a right proper $\cat{Cat}$-enriched model structure along a $\cat{Cat}$-enriched adjunction. Finally, filtered colimits and small limits in $\cat{Cat}$ can be computed pointwise, so filtered colimits commute with finite limits in $\cat{Cat}$ as they do so in $\cat{Set}$. We conclude that filtered colimits in the $\mathcal F$-model structure on $\cat{$\bm M$-Cat}$ are homotopical as they are so in $\cat{Cat}$, and likewise for products.
\end{proof}
\end{thm}

Specializing to $M=E\mathcal M\times G$ for a discrete group $G$ we get:

\begin{cor}\label{cor:g-global-model-cat}
Let $G$ be any discrete group. Then there is a unique model structure on $\cat{$\bm{E\mathcal M}$-$\bm G$-Cat}$ in which a map $f$ is a weak equivalence or fibration if and only if $f^\phi$ is an equivalence of categories or isofibration, respectively, for every universal $H\subset\mathcal M$ and every homomorphism $\phi\colon H\to G$.

We call this the \emph{$G$-global model structure} and its weak equivalences the \emph{$G$-global equivalences}. It is right proper, $\cat{Cat}$-enriched (hence simplicial), and combinatorial with generating cofibrations
\begin{equation*}
\{E\mathcal M\times_\phi G\times i : \text{$H\subset\mathcal M$ universal, $\phi\colon H\to G$, $i\in I_\cat{Cat}$}\}
\end{equation*}
and generating acyclic cofibrations
\begin{equation*}
\{E\mathcal M\times_\phi G\times j : \text{$H\subset\mathcal M$ universal, $\phi\colon H\to G$, $j\in J_\cat{Cat}$}\}
\end{equation*}
for any sets $I_\cat{Cat}, J_\cat{Cat}$ of generating (acyclic) cofibrations of $\cat{Cat}$. Moreover, the $G$-global equivalences are stable under filtered colimits and arbitrary products.\qed
\end{cor}

\begin{rk}\label{rk:g-universal-cat}
Again, the above weak equivalences and fibrations on $\cat{$\bm{E\mathcal M}$-$\bm G$-Cat}$ only depend on the action of the discrete group $\core(\mathcal M)\times G$, i.e.~the $G$-global model structure is transferred from an analogously defined \emph{$G$-universal model structure} on $\cat{$\bm{\core(\mathcal M)}$-$\bm G$-Cat}$ (which however does not model $G$-global category theory).
\end{rk}

\begin{rk}
We can also look at $\cat{$\bm{E\mathcal M}$-$\bm G$-Cat}$ through the eyes of the $G$-global \emph{weak} equivalences, i.e.~those functors $f$ such that $\nerve(f)$ is a $G$-global weak equivalence in the sense of Theorem~\ref{thm:g-global-model-sset}, or equivalently such that $f^\phi$ is a weak homotopy equivalence for every $\phi\colon H\to G$. As we prove in \cite[Corollary~3.30]{cat-g-global}, this then yields another model of unstable $G$-global homotopy theory with the nerve descending to an equivalence of the associated quasi-categories.
\end{rk}

Next, we want to show that for suitable $\mathcal F$ the $\mathcal F$-model structure on $\cat{$\bm M$-Cat}$ is also left proper, and that $\cat{$\bm M$-Cat}$ moreover admits an \emph{injective $\mathcal F$-model structure}. For this, we begin by giving an easy description of the above cofibrations in the case that $M=G$ is a discrete group, analogous to the usual characterization for cofibrations of $G$-simplicial sets recalled in Lemma~\ref{lemma:charact-cof-sset}:

\begin{lemma}\label{lemma:charact-cof}
Assume $M=G$ is a discrete group and that $1\in\mathcal F$. Then a map $i\colon A\to B$ is a cofibration in the $\mathcal F$-model structure if and only if it is an injective cofibration (i.e.~injective on objects) and every object not contained in the image of $i$ has isotropy in $\mathcal F$.

In particular, if $G$ is finite and $\mathcal F=\mathcal A\ell\ell$, then every object is cofibrant.
\begin{proof}
We first observe that all generating cofibrations are injective cofibrations and satisfy the above isotropy condition. As the class of all such functors is closed under retracts, pushouts, and transfinite compositions (using that $\Ob$ preserves colimits), we conclude that every $\mathcal F$-cofibration satisfies the above properties.

To complete the proof it therefore suffices to show that any functor $i\colon A\to B$ that is injective on objects and satisfies the above isotropy condition has the left lifting property against each $\mathcal F$-acyclic fibration $p\colon C\to D$. For this we consider any lifting problem
\begin{equation*}
\begin{tikzcd}
A\arrow[r,"\alpha"]\arrow[d, "i"'] & C\arrow[d, two heads, "p", "\sim"']\\
B\arrow[r, "\beta"']\arrow[ur, dashed, "\lambda"] & D.
\end{tikzcd}
\end{equation*}
We may assume without loss of generality that $\Ob(A)$ is a subset of $\Ob(B)$ and that $i$ is given on objects by the inclusion. To define $\lambda$ on objects, we first pick orbit representatives $(b_i)_{i\in I}$ for the $G$-action on $\Ob(B)$. We now set $\lambda(g.b_i)=\alpha(g.b_i)$ whenever $b_i\in A$; otherwise, we write $H$ for the stabilizer of $b_i$ and observe that $H\in\mathcal F$ by assumption, so that $p^H\colon C^H\to D^H$ is an acyclic fibration in $\cat{Cat}$, hence in particular surjective on objects. Thus, we can pick $c_i\in C^H$ with $p(c_i)=\beta(b_i)$, and we set $\lambda(g.b_i)=g.c_i$ for all $g\in G$. We omit the straightforward verification that this is well-defined and $G$-equivariant, that it extends $\alpha$ on objects, and that $p\lambda(b)=\beta(b)$ for all $b\in B$.

Next, we define $\lambda$ on morphisms as follows: given any $f\colon b\to b'$ in $B$, we have $p\lambda(b)=\beta(b)$ and $p\lambda(b')=\beta(b')$ by the above, so there is a unique map $g\colon\lambda(b)\to\lambda(b')$ with $p(g)=\beta(f)$ as $p$ is fully faithful. We then set $\lambda(f)\mathrel{:=} g$. We omit the straightforward verification that this is well-defined and $G$-equivariant. By construction, $p\lambda=\beta$, and moreover $\lambda i=\alpha$ on objects. To prove that $\lambda i$ and $\alpha$ also agree on morphisms, it is enough by full faithfulness to prove this after postcomposition with $p$, in which case this follows from the equalities $p\lambda i=\beta i=p\alpha$.
\end{proof}
\end{lemma}

\begin{cor}
For $M=G$ a finite discrete group and $\mathcal F=\mathcal A\ell\ell$, the $\mathcal F$-equivalences are precisely the equivalences in the $2$-category of $G$-categories, $G$-equivariant functors, and $G$-equivariant natural transformations.
\end{cor}

We will refer to these simply as `$G$-equivalences' below.

\begin{proof}
By the previous lemma, all objects of $\cat{$\bm G$-Cat}$ are cofibrant-fibrant, so the Abstract Whitehead Theorem shows that a map is an $\mathcal A\ell\ell$-equivalence if and only if it is a homotopy equivalence. Picking the cylinder objects
\begin{equation*}
C\amalg C\xrightarrow{(\incl_0,\incl_1)} C\times E\{0,1\}\xrightarrow{\pr} C
\end{equation*}
coming from the $\cat{Cat}$-enrichment, the claim follows immediately.
\end{proof}

With the above characterization of the cofibrations at hand we can now prove:

\begin{prop}\label{prop:inj-po}
Let $M$ be any categorical monoid and let $\mathcal F$ be a \emph{family} of finite subgroups of $\Ob(M)$. Then the $\mathcal F$-equivalences are stable under pushout along injective cofibrations. In particular the $\mathcal F$-model structure is left proper.
\begin{proof}
As $\mathcal F$ is closed under passing to subgroups, a map in $\cat{$\bm M$-Cat}$ is an $\mathcal F$-equivalence if and only if it is an $H$-equivalence for each $H\in\mathcal F$. As moreover pushouts in $\cat{$\bm M$-Cat}$ and $\cat{$\bm H$-Cat}$ are both created in $\cat{Cat}$, we are therefore reduced to the case that $M=H$ is a finite discrete group and $\mathcal F=\mathcal A\ell\ell$ is the family of all subgroups.

But in this situation the cofibrations of the $\mathcal F$-model structure are precisely the injective cofibrations. In particular, every object is cofibrant and hence $\cat{$\bm H$-Cat}$ is left proper, which immediately implies the claim.
\end{proof}
\end{prop}

\begin{thm}\label{thm:injective-equivariant-model}
Let $M$ be any categorical monoid and let $\mathcal F$ be a family of finite subgroups of $\Ob(M)$. Then there exists a unique model structure on $\cat{$\bm M$-Cat}$ whose weak equivalences are the $\mathcal F$-equivalences and whose cofibrations are the injective cofibrations. We call this the \emph{injective $\mathcal F$-model structure}. It is combinatorial, proper, $\cat{Cat}$-enriched (hence simplicial), and Cartesian.
\begin{proof}
To prove that the model structure exists and that it is combinatorial and proper, it suffices by \cite[Corollary~A.2.18]{g-global} and the existence of the \emph{non-equivariant} injective model structure on $\cat{$\bm M$-Cat}$ that pushouts of $\mathcal F$-equivalences along injective cofibrations are $\mathcal F$-equivalences, which is precisely the content of the previous proposition.

Next, let us show that the model structure is Cartesian. It is clear that the unit is cofibrant, so that it only remains to verify the Pushout Product Axiom. For cofibrations this follows directly from the fact that $\cat{Cat}$ is Cartesian, so it only remains to show the induced map in
\begin{equation*}
\begin{tikzcd}
A\times C\arrow[d, "A\times j"']\arrow[r, "i\times C"] & B\times C\arrow[d,"k"]\arrow[ddr, "B\times j", bend left=10pt]\\
A\times D\arrow[r]\arrow[rrd, "i\times D"', bend right=10pt] &[-1em] P\arrow[dr, dashed]\arrow[ul,phantom,"\ulcorner" very near start]\\[-1em]
&& B\times D
\end{tikzcd}
\end{equation*}
is an $\mathcal F$-equivalence for any cofibration $i\colon A\to B$ and any acyclic cofibration $j\colon C\to D$. For this we observe that $A\times j$ and $B\times j$ are acylic cofibrations, and so is $k$ as a pushout of an acyclic cofibration. The claim now follows by $2$-out-of-$3$.

The proof that this model structure is $\cat{Cat}$-enriched is analogous.
\end{proof}
\end{thm}

\begin{cor}\label{cor:injective-g-global-model-cat}
Let $G$ be any finite group. There is a unique model structure on $\cat{$\bm{E\mathcal M}$-$\bm G$-Cat}$ with cofibrations the injective cofibrations and weak equivalences the $G$-global equivalences. We call this the \emph{injective $G$-global model structure}. It is proper, combinatorial, Cartesian, and $\cat{Cat}$-enriched (hence simplicial).\qed
\end{cor}

\begin{warn}\label{warning:no-elmendorf}
In the simplicial world, \emph{Elmendorf's Theorem} \cite{elmendorf} provides an alternative description of the equivariant homotopy theory of $G$-simplicial sets in terms of fixed point systems. More precisely, we define for any discrete group $G$ and any collection $\mathcal F$ of subgroups of $G$ the \emph{orbit category} $\cat{O}_{\mathcal F}$ as the full subcategory of $\cat{$\bm G$-Set}$ spanned by the transitive $G$-sets of the form $G/H$ with $H\in\mathcal F$. We then obtain a functor $\Phi\colon\cat{$\bm G$-SSet}\to\Fun(\cat{O}_{\mathcal F}^\op,\cat{SSet})$ given by $\Phi(X)(G/H)=\Maps^G(G/H,X)\cong X^H$ with the evident functoriality in both variables, and Elmendorf's Theorem says that this is an equivalence of homotopy theories with respect to the $\mathcal F$-weak equivalences on the source and the levelwise weak equivalences on the target. This result was refined by Bergner (for $\mathcal F=\mathcal A\ell\ell$) who constructed a model structure on $\cat{$\bm G$-SSet}$ with weak equivalences those maps $f$ such that $f^H$ is a Joyal equivalence for every subgroup $H\subset G$, and proved that $\Phi$ is the right half of a Quillen equivalence to $\Fun(\cat{O}_{\mathcal A\ell\ell}^\op,\cat{SSet}_{\text{Joyal}})$, see \cite[Theorem~3.3]{bergner-equivariant}.

We caution the reader that the corresponding statement is \emph{not} true in our situation, even if $M=G$ is a finite discrete group and $\mathcal F=\mathcal A\ell\ell$. Indeed, if $\Phi$ were to induce an equivalence of associated quasi-categories, it would have to preserve homotopy pushouts, and hence so would any of the fixed point functors $(\blank)^H\cong \ev_{G/H}\circ\Phi$ for $H\in\mathcal F$. By Theorem~\ref{thm:injective-equivariant-model} above this would imply that pushouts along injective cofibrations commute with fixed points up to equivalence. We will show that this is not the case, already for $G=\mathbb Z/2$ and $\mathcal F=\mathcal A\ell\ell$; note that this does not contradict Bergner's result mentioned above as the nerve does not preserve (homotopy) pushouts.

For this, we consider the `fork' $x\rightrightarrows y\to z$ where the arrow $y\to z$ coequalizes the two arrows $x\rightrightarrows y$. Comparing corepresented functors then shows that the pushout of
\begin{equation}\label{eq:fork-pushout}
(x\rightrightarrows y\to z) \hookleftarrow (x\rightrightarrows y) \to *
\end{equation}
is simply the category $y\to z$ (with the evident maps). We now make this square into a diagram in $\cat{$\bm{\mathbb Z/2}$-Cat}$ by letting the non-trivial element of $\mathbb Z/2$ exchange the two arrows $x\rightrightarrows y$ and act trivial otherwise; it follows formally that this is then a pushout square again.

However, the $\mathbb Z/2$-fixed points of $(\ref{eq:fork-pushout})$ are given by
\begin{equation}\label{eq:fork-fixed-points-pushout}
(\begin{tikzcd}[cramped,column sep=small]x\arrow[rr, bend right=25pt, looseness=1,shift right=2pt] & y\arrow[r]& z\end{tikzcd}) \hookleftarrow (x\quad y) \to *
\end{equation}
and the pushout of this is the category $y\rightrightarrows z$.

Note that the same example shows that the injective $\mathcal F$-model structure does not exist for general collections of subgroups $\mathcal F$ (as opposed to families). Namely, if the injective model structure for $G=\mathbb Z/2$ and $\mathcal F=\{\mathbb Z/2\}$ existed, then it would necessarily be left proper as all $G$-categories are injectively cofibrant. However, the inclusion of $(\ref{eq:fork-fixed-points-pushout})$ into $(\ref{eq:fork-pushout})$ is a levelwise $\mathcal F$-equivalence, while the induced map on pushouts isn't.
\end{warn}

\begin{rk}\label{rk:injectively-fibrant}
Let $M=G$ be a finite discrete group and let $\mathcal F$ be a family of finite subgroups of $G$; we now want to construct some fibrant objects of the injective $\mathcal F$-model structure on $\cat{$\bm G$-Cat}$ in more concrete terms.

To this end, we let $Q$ be a cofibrant replacement of the terminal category $*$ in the usual $\mathcal F$-model structure on $\cat{$\bm G$-Cat}$. Then $Q^H=\varnothing$ for any $H\notin\mathcal F$ by Lemma~\ref{lemma:charact-cof}, so $Q\times\blank$ sends $\mathcal F$-equivalences to $G$-equivalences. On the other hand, it clearly preserves injective cofibrations, so that we get a Quillen adjunction
\begin{equation*}
Q\times\blank\colon\cat{$\bm G$-Cat}_{\textup{injective $\mathcal F$}}\rightleftarrows\cat{$\bm G$-Cat}_{\mathcal A\ell\ell} :\Fun(Q,\blank);
\end{equation*}
in particular $\Fun(Q,C)$ is injectively fibrant for any $G$-category $C$. While we will not do this here, one can in fact show (using the general theory of Bousfield localizations of model categories) that a $G$-category $C$ is injectively fibrant if and only if the map $C\to\Fun(Q,C)$ induced by $Q\to*$ is an $\mathcal A\ell\ell$-equivalence, also cf.~\cite[discussion after Proposition~1.12]{hausmann-equivariant} for the corresponding (classical) statement for simplicial sets.

As one particular instance of this, let us consider the case that $\mathcal F=\mathcal T\!\textit{riv}$ consists only of the trivial group. Then we have a standard choice of $Q$ as the indiscrete category $EG$ with $G$-action induced by the left regular action of $G$; in particular, $\Fun(EG,C)$ is injectively fibrant for any $G$-category $C$. $G$-categories of this form play a central role in Merling's treatment of equivariant algebraic $K$-theory \cite{merling}, and several of the key properties established by her actually become formal consequences of injective fibrancy: in particular, Ken Brown's Lemma shows that $\Fun(EG,\blank)$ takes underlying equivalences to $G$-equivalences \cite[Proposition~2.16]{merling}, while the Abstract Whitehead Theorem shows that any underlying equivalence between categories of the form $\Fun(EG,C)$ is already a $G$-equivalence, which immediately implies Lemma~2.8 of {op.~cit.}
\end{rk}

\begin{rk}
We can also use the above to answer a question raised by Merling as \cite[Question~3.5]{merling}: namely, she shows in Proposition~3.3 of {op.~cit.} how $\Fun(EG,\blank)$ lifts to a functor from the category of $G$-objects and \emph{pseudoequivariant functors} (i.e.~pseudonatural transformations of functors $BG\to\cat{Cat}_2$ into the $2$-category of categories) to just $G$-objects and \emph{strictly} equivariant functors, and she asks whether all equivariant functors $\Fun(EG,C)\to\Fun(EG,D)$ arise this way.

This is indeed the case: first we observe that the non-equivariant equivalences in $\cat{$\bm G$-Cat}$ are the just the usual level weak equivalences, so \cite[Theorem~7.9.8 and Remark~7.9.7]{cisinski-book} together with \cite[1.4.3]{dk-modern} implies that the canonical map $\cat{$\bm G$-Cat}\to\Fun(\nerve(BG),\nerve_\Delta(\cat{Cat}))$ is a quasi-localization at the non-equivariant equivalences, where we view $\cat{Cat}$ as simplicially enriched in the same way as before. This map can be identified up to equivalence with the nerve of the inclusion of $\cat{$\bm G$-Cat}$ into the strict $2$-category $\Fun_2(BG,\cat{Cat}_2)$ of strict $G$-objects, pseudoequivariant functors, and pseudoequivariant natural isomorphisms (i.e.~invertible modifications); alternatively, one can use a standard argument due to Dwyer and Kan recalled in \cite[Proposition~A.1.10]{g-global} to directly prove that the latter map is a quasi-localization. In any case we in particular see that we have an ordinary $1$-categorical localization $\cat{$\bm G$-Cat}\to\h\Fun_2(BG,\cat{Cat}_2)$.

On the other hand, $\cat{$\bm G$-Cat}$ is a model category in which every object is cofibrant, so the natural map of the $1$-category of injectively fibrant $G$-objects into the category $\h\Fun_{2,\text{strict}}(BG,\cat{Cat}_{2})_\text{fibrant}$ of injectively fibrant (strict) $G$-objects and isomorphism classes of strictly $G$-equivariant maps is also a localization. By direct inspection, Merling's construction is compatible with invertible modifications, and accordingly it descends to a map $\h\Fun_2(BG,\cat{Cat}_2)\to\h\Fun_{2,\text{strict}}(BG,\cat{Cat}_{2})_\text{fibrant}$, yielding a commutative diagram
\begin{equation*}
\begin{tikzcd}
\cat{$\bm G$-Cat}\arrow[d]\arrow[r,"{\Fun(EG,\blank)}"] &[2.5em] \cat{$\bm G$-Cat}_\text{fibrant}\arrow[d]\\
\h\Fun_2(BG,\cat{Cat}_2)\arrow[r,"{\Fun(EG,\blank)}"'] &\h\Fun_{2,\text{strict}}(BG,\cat{Cat}_{2})_\text{fibrant}\rlap.
\end{tikzcd}
\end{equation*}
Thus, we can view her construction as the derived functor of $\Fun(EG,\blank)\colon\cat{$\bm G$-Cat}\to\cat{$\bm G$-Cat}_\text{fibrant}$. As the latter is homotopy inverse to the inclusion we see that Merling's construction is an equivalence, hence in particular fully faithful as desired.

Finally, we remark that a slightly more elaborate version of the above argument actually shows that her construction extends to induce for every $C,D\in\cat{$\bm G$-Cat}$ an equivalence between the category of pseudoequivariant functors $C\to D$ and pseudoequivariant transformations between such to the category of strictly equivariant functors and strictly natural transformations between $\Fun(EG,C)$ and $\Fun(EG,D)$.
\end{rk}

\subsection{Functoriality} We now discuss some functoriality properties of the model structures from Theorem~\ref{thm:equiv-model-structure} analogous to the situation for simplicial sets. Throughout, we let $M$ be a categorical monoid and $\mathcal F$ a family of finite subgroups of $\Ob(M)$.

\begin{lemma}\label{lemma:alpha-shriek}
Let $\alpha\colon H\to G$ be any group homomorphism. Then
\begin{equation*}
\alpha_!\colon \cat{$\bm{(M\times H)}$-Cat}_{\mathcal G_{\mathcal F,H}}\rightleftarrows\cat{$\bm{(M\times G)}$-Cat}_{\mathcal G_{\mathcal F,G}} :\!\alpha^*
\end{equation*}
is a Quillen adjunction with homotopical right adjoint.
\begin{proof}
One immediately checks from the definitions that $\alpha^*$ preserves weak equivalences as well as fibrations.
\end{proof}
\end{lemma}

\begin{lemma}\label{lemma:free-quotients}
Let $G$ be any discrete group and assume $f\colon C\to D$ is a $\mathcal G_{\mathcal F,G}$-equivalence in $\cat{$\bm{(M\times G)}$-Cat}$ such that $G$ acts freely on both $C$ and $D$. Then $f/G\colon C/G\to D/G$ is an $\mathcal F$-equivalence.
\begin{proof}
As before we reduce to the case that $M=H$ is a finite discrete group and $\mathcal F=\mathcal A\ell\ell$. As $G$ acts freely on $C$, the isotropy of any $c\in C$ intersects $G$ trivially, i.e.~it is an element of the family $\mathcal G_{H,G}$, and likewise for $D$. Thus, Lemma~\ref{lemma:charact-cof} shows that both $C$ and $D$ are cofibrant in the $\mathcal G_{H,G}$-model structure, and the claim follows by applying the previous lemma to the unique homomorphism $G\to 1$.
\end{proof}
\end{lemma}

\begin{lemma}\label{lemma:alpha-star}
Let $\alpha\colon H\to G$ be an \emph{injective} homomorphism of discrete groups. Then
\begin{equation*}
\alpha^*\colon\cat{$\bm{(M\times G)}$-Cat}_{\mathcal G_{\mathcal F,G}}\rightleftarrows\cat{$\bm{(M\times H)}$-Cat}_{\mathcal G_{\mathcal F,H}} :\!\alpha_*
\end{equation*}
is a Quillen adjunction and $\alpha_*$ is fully homotopical.
\begin{proof}
It suffices to show that $\alpha_*$ preserves fibrations and weak equivalences, for which we may again reduce to the case that $M=K$ is a finite discrete group and $\mathcal F=\mathcal A\ell\ell$. As restriction along an \emph{injective} homomorphism preserves freeness, the same argument as in the previous lemma then shows that $\alpha^*$ preserves cofibrations; moreover, it is clearly homotopical, hence left Quillen.
Thus, $\alpha_*$ is right Quillen and hence homotopical by Ken Brown's Lemma.
\end{proof}
\end{lemma}

Arguing as in the proof of Corollary~\ref{cor:twisted-products-sset} we get:

\begin{cor}\label{cor:twisted-products}
Let $G$ be any discrete group, let $n\ge 0$, and let $f\colon C\to D$ be a $\mathcal G_{\mathcal F,G}$-equivalence in $\cat{$\bm{(M\times G)}$-Cat}$. Then the map $f^{\times n}\colon C^{\times n}\to D^{\times n}$ is a $\mathcal G_{\mathcal F,G\times\Sigma_n}$-equivalence in $\cat{$\bm{(M\times G\times\Sigma_n)}$-Cat}$ with respect to the $\Sigma_n$-action permuting the factors.\qed
\end{cor}

\subsection{Equivariant categorical operads}
We now want to study operads in $\cat{$\bm G$-Cat}$ and $\cat{$\bm{E\mathcal M}$-$\bm G$-Cat}$ for any (finite) group $G$, along with their algebras. Again, the basic theory works in greater generality, so we fix a categorical monoid $M$ together with a family $\mathcal F$ of finite subgroups of $\Ob(M)$.

\begin{constr}
Analogously to the simplicial situation, the forgetful functor
\begin{equation*}
    \forget\colon\Alg_{\mathcal O}(\cat{$\bm M$-Cat})\to\cat{$\bm M$-Cat}
\end{equation*}
has a left adjoint $\cat P$ given by the formula $\cat{P}C=\coprod_{n\ge 0} \mathcal O(n)\times_{\Sigma_n} C^{\times n}$ with the evident functoriality in $C$ and the $\mathcal O$-algebra structure induced by operad structure maps of $\mathcal O$.

Again, $\Alg_{\mathcal O}(\cat{$\bm M$-Cat})$ is enriched, tensored, and cotensored over $\cat{Cat}$ (with cotensors created in $\cat{$\bm M$-Cat}$), and the adjunction $\cat{P}\dashv\forget$ is then naturally a $\cat{Cat}$-enriched adjunction.
\end{constr}

\begin{thm}
Let $\mathcal O$ be any operad in $\cat{$\bm M$-Cat}$. Then there exists a unique model structure on $\Alg_{\mathcal O}(\cat{$\bm M$-Cat})$ in which a map is a weak equivalence or fibration if and only if it is so in the $\mathcal F$-model structure on $\cat{$\bm M$-Cat}$. We call this the \emph{$\mathcal F$-model structure} again. It is combinatorial, right proper, $\cat{Cat}$-enriched (hence simplicial), and filtered colimits in it are homotopical. Moreover, the free-forgetful adjunction
\begin{equation}
\cat{P}\colon\cat{$\bm M$-Cat}_{\mathcal F}\rightleftarrows\Alg_{\mathcal O}(\cat{$\bm M$-Cat}) :\!\forget
\end{equation}
is a Quillen adjunction with homotopical right adjoint.
\begin{proof}
It is clear that $\Alg_{\mathcal O}(\cat{$\bm M$-Cat})$ is locally presentable, so to prove that the transferred model structure along $\cat{P}\dashv\forget$ exists and is combinatorial, right proper, and $\cat{Cat}$-enriched, it is again enough to construct functorial path objects.

For this we observe that $\Fun(E\{0,1\},C)$ inherits a natural $\mathcal O$-algebra structure from $C$ as $\Fun(E\{0,1\},\blank)$ is $\cat{Cat}$-enriched and product-preserving. With respect to this algebra structure, the maps
\begin{equation*}
C\xrightarrow{\const}\Fun(E\{0,1\},C)\xrightarrow{(\ev_0,\ev_1)} C\times C
\end{equation*}
are then $\mathcal O$-algebra maps, so this provides the desired path object by the proof of Theorem~\ref{thm:equiv-model-structure}.
\end{proof}
\end{thm}

\begin{ex}
Let $G$ be a finite group. Analogously to the simplicial situation, an operad $\mathcal O$ in $\cat{$\bm G$-Cat}$ will be called a \emph{na\"ive $G$-$E_\infty$-operad} if each $\mathcal O(n)$ is $G$-equivariantly equivalent to the terminal category and $\mathcal O(n)^H=\varnothing$ for any $H\subset G\times\Sigma_n$ not contained in $G$. Such an operad is in particular $\Sigma$-free (i.e.~$\Sigma_n$ acts freely on $\mathcal O(n)$ for any $n\ge0$).

Again the prototypical example is an $E_\infty$-operad $\mathcal Q$ in $\cat{Cat}$ that we equip with the trivial $G$-action. A $\mathcal Q$-algebra in $\cat{$\bm G$-Cat}$ is then the same data as a $G$-object in $\Alg_{\mathcal Q}(\cat{Cat})$. A particularly important example of $\mathcal Q$ for us is the \emph{(categorical) Barratt-Eccles operad} $E\Sigma_*$ whose $n$-ary operators are given by the category $E\Sigma_n$ with the evident right $\Sigma_n$-action. There is then a \emph{unique} way to make $E\Sigma_*$ into an operad; we refer the reader to \cite[Proposition~4.2 and Lemma~4.4]{may-permutative} for details. Any permutative category gives rise to an algebra over $E\Sigma_*$ by \cite[Lemmas 4.3--4.5]{may-permutative}, and as observed e.g.~in \cite[Proposition~4.2]{guillou-may} without proof this yields an isomorphism between the category $\cat{PermCat}$ of permutative categories and \emph{strict} symmetric monoidal functors and $\Alg_{E\Sigma_*}(\cat{Cat})$.
As an upshot of all of this, we can identify $\Alg_{E\Sigma_*}(\cat{$\bm G$-Cat})$ with the category $\cat{$\bm G$-PermCat}$ of $G$-objects in $\cat{PermCat}$. Following \cite[4.1]{guillou-may}, we will refer to the objects of $\cat{$\bm G$-PermCat}$ as \emph{na\"ive permutative $G$-categories}.
\end{ex}

\begin{ex}\label{ex:genuine-G-E-infty}
In analogy to the terminology in the simplicial setting, a $\Sigma$-free operad $\mathcal O$ in $\cat{$\bm G$-Cat}$ will be called a \emph{genuine $G$-$E_\infty$-operad} if the unique map $\mathcal O\to *$ is a $G$-equivalence, i.e.~each $\mathcal O(n)$ is $\mathcal G_{G,\Sigma_n}$-equivalent to the terminal category.

If $\mathcal P$ is a $\Sigma$-free operad such that each $\mathcal P(n)$ is \emph{non-equivariantly} equivalent to the terminal category, then we can build a genuine $G$-$E_\infty$-operad $\mathcal P^{EG}$ from this as follows, also see~\cite[Definition~4.4]{guillou-may}: we set $\mathcal P^{EG}(n)=\Fun(EG,\mathcal P(n))$ with the induced $\Sigma_n$-action and $G$ acting via the diagonal of its action on $\mathcal P(n)$ and the left action induced by its right regular action on $EG$; the structure maps of $\mathcal P^{EG}$ are induced from the ones of $\mathcal P$ in the obvious way. In particular, we can apply this to the Barratt-Eccles operad (equipped with trivial $G$-action), yielding a genuine $G$-$E_\infty$-operad $E\Sigma_*^{EG}$ with $n$-ary operations given by $\Fun(EG,E\Sigma_*)$. Following \cite[Definition~4.5]{guillou-may}, we will refer to $E\Sigma_*^{EG}$-algebras as \emph{genuine permutative $G$-categories}.

As observed by Guillou and May in Proposition~4.6 of {op.~cit.}, also see \cite[Remark after Theorem $\text{A}'$]{shimakawa}, we can apply the same construction to algebras: since $\Fun(EG,\blank)$ preserves products, $\Fun(EG,C)$ carries a natural $\mathcal O^{EG}$-algebra structure for any $\mathcal O$-algebra $C$. In particular, $\Fun(EG,\blank)$ lifts to a functor from na\"ive permutative $G$-categories to genuine ones.
\end{ex}

\begin{rk}
While permutative categories are rare in practice, MacLane's strictification theorem implies that any symmetric monoidal category is equivalent to a permutative category, and in fact the quasi-localizations of $\cat{PermCat}$ and $\cat{SymMonCat}$ at the underlying equivalences of categories are equivalent.

In the equivariant setting, Guillou, May, Merling, and Osorno \cite{gmmo-sym-mon} introduced genuine and na\"ive \emph{symmetric monoidal $G$-categories} as pseudoalgebras over genuine and na\"ive $G$-$E_\infty$-operads, respectively. Generalizing the non-equivariant situation, they provide a strictification result showing that the homotopy theory of genuine and na\"ive symmetric monoidal $G$-categories (with respect to the underlying $G$-equivalences) is equivalent to the homotopy theory of genuine and na\"ive permutative $G$-categories, respectively, as defined above. In particular, while we will for simplicity talk exclusively about the above categories of strict algebras in this paper, our results carry over to the pseudoalgebra setting immediately.
\end{rk}

\begin{ex}
Let $G$ be a discrete group. A \emph{$G$-global operad} is an operad in $\cat{$\bm{E\mathcal M}$-$\bm G$-Cat}$. We call a $G$-global operad $\mathcal O$ a \emph{$G$-global $E_\infty$-operad} if it is $\Sigma$-free and moreover each $\mathcal O(n)$ is $(G\times\Sigma_n)$-globally equivalent to the $1$-point category, i.e.~the unique map $\mathcal O\to *$ is a $G$-global equivalence.

Analogously to Example~\ref{ex:genuine-G-E-infty}, we can easily construct these from ordinary $E_\infty$-operads: for this, let $\mathcal O$ be any operad in $\cat{$\bm G$-Cat}$ that is an underlying $E_\infty$-operad. The functor $\Fun(E\mathcal M,\blank)\colon\cat{$\bm{(G\times\Sigma_n)}$-Cat}\to\cat{$\bm{(E\mathcal M\times G\times\Sigma_n)}$-Cat}$ sends underlying equivalences to $(G\times\Sigma_n)$-global equivalences, so we obtain a $G$-global $E_\infty$-operad $\mathcal O^{E\mathcal M}$ via $\mathcal O^{E\mathcal M}(n)=\Fun(E\mathcal M,\mathcal O(n))$ with the induced operad structure maps, left $G$-actions, and right actions by the symmetric groups. In particular, we can apply this to the usual Barratt-Eccles operad (equipped with the trivial $G$-action). We call the resulting operad $E\Sigma_*^{E\mathcal M}$ the \emph{$G$-global Barratt-Eccles operad}. Its nerve (denoted by the same symbol) is then in particular a $G$-global $E_\infty$-operad of simplicial sets. We emphasize that $E\Sigma_*^{E\mathcal M}$ still has \emph{trivial} $G$-action.
\end{ex}

\subsection{Change of operad}
Our next goal is to prove that also in the categorical setting the model categories of algebras are invariant under suitable equivalences of $\Sigma$-free operads. The corresponding statement for simplicial sets ultimately relied on geometric realization preserving levelwise weak equivalences (so that it models $\Delta^\op$-shaped homotopy colimits) as well as finite products (as used in Proposition~\ref{prop:forget-geometric-realization}), and we begin by establishing the corresponding results for $\cat{Cat}$.

For this let us first recall that by definition of the $\cat{SSet}$-tensoring of $\cat{Cat}$, the geometric realization of a simplicial object $C_\bullet\colon\Delta^\op\to\cat{Cat}$ is given by the coend
\begin{equation*}
\int^{[n]\in\Delta^\op} \Delta^n\otimes C_n =
\int^{[n]\in\Delta^\op} \grpdfy(\h\Delta^n)\times C_n
\end{equation*}
with the evident $\cat{Cat}$-enriched functoriality; here $\h$ again denotes the left adjoint of the nerve (assigning to a simplicial set its homotopy category) and $\grpdfy$ is the left adjoint of the inclusion $\cat{Grpd}\hookrightarrow\cat{Cat}$.

More generally, if $M$ is any categorical monoid, then $\cat{$\bm M$-Cat}$ likewise acquires a notion of geometric realization, and this can be explicitly computed as the geometric realization in $\cat{Cat}$ together with the induced $M$-action.

\begin{prop}\label{prop:geometric-realization-homotopical}
The geometric realization functor
\begin{equation}\label{eq:geometric-realization}
|\blank|\colon\Fun(\Delta^\op,\cat{$\bm M$-Cat})\to\cat{$\bm M$-Cat}
\end{equation}
sends levelwise $\mathcal F$-equivalences to $\mathcal F$-equivalences for every family $\mathcal F$ of finite subgroups of $\Ob(M)$.
\begin{proof}
We equip $\cat{$\bm M$-Cat}$ with the \emph{injective} $\mathcal F$-model structure. As this is simplicial, $(\ref{eq:geometric-realization})$ is left Quillen with respect to the Reedy model structure on the source. To complete the proof it is then enough by Ken Brown's Lemma that every object of the source is Reedy-cofibrant.

For this we observe that the functor $\discr\circ\Ob\colon\cat{$\bm M$-Cat}\to\cat{SSet}$ is cocontinuous and creates cofibrations. In particular, $X\in\Fun(\Delta^\op,\cat{$\bm M$-Cat})$ is Reedy cofibrant if and only if $\discr\Ob X$ is a Reedy cofibrant bisimplicial set. The claim follows immediately as every bisimplicial set is Reedy cofibrant.
\end{proof}
\end{prop}

\begin{lemma}\label{lemma:geometric-prod}
Let $M$ be any categorical monoid. Then the geometric realization functor $|\blank|\colon\Fun(\Delta^\op,\cat{$\bm M$-Cat})\to\cat{$\bm M$-Cat}$ preserves finite products.
\begin{proof}
It suffices to prove the corresponding statement for $\cat{Cat}$. As the canonical natural transformation $|\blank\times\blank|\Rightarrow|\blank|\times|\blank|$ is $\cat{Cat}$-enriched and since both source and target preserve $\cat{Cat}$-tensors and colimits in each variable separately, it suffices to check this on pairs of represented functors, hence in particular on levelwise discrete simplicial categories. As the functor $\discr\colon\Fun(\Delta^\op,\cat{Set})\to\Fun(\Delta^\op,\cat{Cat})$ preserves products, it is then finally enough to prove the claim after restricting along $\discr$.
\begin{claim*}
The diagram
\begin{equation*}
\begin{tikzcd}
\Fun(\Delta^\op,\cat{Set})\arrow[d, equal]\arrow[r, "\discr"] & \Fun(\Delta^\op,\cat{Cat})\arrow[d, "|\blank|"]\\
\cat{SSet}\arrow[r, "\grpdfy\circ\h"'] & \cat{Cat}
\end{tikzcd}
\end{equation*}
commutes up to natural isomorphisms.
\begin{proof}
It suffices to construct the natural isomorphism after restricting to $\Delta$, where we just take the isomorphisms $\grpdfy(\h\Delta^m)\to\int^{[n]} \grpdfy(\h\Delta^n)\times\Hom([n],[m])$ provided by the co-Yoneda Lemma.
\end{proof}
\end{claim*}
The lemma follows as $\grpdfy\circ\h\colon\cat{SSet}\to\cat{Cat}$ preserves finite products.
\end{proof}
\end{lemma}

With this established we can prove the desired homotopy invariance statement:

\begin{defi}
A map $f\colon\mathcal O\to\mathcal P$ of operads in $\cat{$\bm M$-Cat}$ is called an \emph{$\mathcal F$-equivalence} if $f(n)\colon \mathcal O(n)\to\mathcal P(n)$ is a $\mathcal G_{\mathcal F,\Sigma_n}$-equivalence for every $n\ge 0$ (where we turn the right $\Sigma_n$-action into a left one as usual).
\end{defi}

\begin{thm}\label{thm:change-of-operad-categories}
Let $f\colon\mathcal O\to\mathcal P$ be an $\mathcal F$-equivalence of $\Sigma$-free operads in $\cat{$\bm M$-Cat}$. Then the Quillen adjunction
\begin{equation}\label{eq:change-of-operad}
f_!\colon\Alg_{\mathcal O}(\cat{$\bm M$-Cat})_{\mathcal F}\rightleftarrows\Alg_{\mathcal P}(\cat{$\bm M$-Cat})_{\mathcal F} :\!f^*
\end{equation}
is a Quillen equivalence.
\end{thm}

This will again rely on a monadicity argument:

\begin{prop}\label{prop:free-monadic}
Let $\mathcal O$ be any operad in $\cat{$\bm M$-Cat}$. Then
\begin{equation}\label{eq:forget-inft-cat}
\forget^\infty\colon\Alg_{\mathcal O}(\cat{$\bm M$-Cat})^\infty_{\mathcal F}\to\cat{$\bm M$-Cat}_{\mathcal F}^\infty
\end{equation}
is conservative and preserves $\Delta^\op$-shaped homotopy colimits. In particular, the adjunction $\cat{L}\cat{P}\dashv{\forget}^\infty$ is monadic.
\begin{proof}
Arguing as in Proposition~\ref{prop:free-monadic-sset}, the only non-trivial statement is that $(\ref{eq:forget-inft-cat})$ preserves $\Delta^\op$-homotopy colimits. However, these can again simply be computed via geometric realization as geometric realizations in $\cat{$\bm M$-Cat}$ are fully homotopical by Proposition~\ref{prop:geometric-realization-homotopical}, and hence so are geometric realizations in $\Alg_{\mathcal O}(\cat{$\bm M$-Cat})$ by Proposition~\ref{prop:forget-geometric-realization} together with Lemma~\ref{lemma:geometric-prod}. The claim then simply follows from another application of Proposition~\ref{prop:forget-geometric-realization}.
\end{proof}
\end{prop}

\begin{proof}[Proof of Theorem~\ref{thm:change-of-operad-categories}]
It is again clear that $(\ref{eq:change-of-operad})$ is a a Quillen adjunction with homotopical right adjoint, so it suffices that $(f^*)^\infty$ is an equivalence of quasi-categories. For this we consider the diagram
\begin{equation*}
\begin{tikzcd}
\Alg_{\mathcal P}(\cat{$\bm M$-Cat})^\infty_{\mathcal F}\arrow[r, "(f^*)^\infty"]\arrow[d, "\forget^\infty"'] & \Alg_{\mathcal O}(\cat{$\bm M$-Cat})_{\mathcal F}^\infty\arrow[d, "\forget^\infty"]\\
\cat{$\bm M$-Cat}_{\mathcal F}^\infty\twocell[ur]\arrow[r, equal] & \cat{$\bm M$-Cat}_{\mathcal F}^\infty
\end{tikzcd}
\end{equation*}
which commutes up to the natural isomorphism induced by the identity transformation $\forget\Rightarrow \forget\circ f^*$. By monadicity, it will once more suffice that the canonical mate of the above transformation is an equivalence, which amounts to saying that for each (cofibrant) $C\in\cat{$\bm M$-Cat}$ the map
\begin{equation*}
\coprod_{n\ge 0}\mathcal O(n)\times_{\Sigma_n} C^n\to\coprod_{n\ge 0} \mathcal P(n)\times_{\Sigma_n} C^n
\end{equation*}
induced by $f$ is an $\mathcal F$-equivalence. This is in turn immediate from Lemma~\ref{lemma:free-quotients}.
\end{proof}

\subsection{A comparison functor for \texorpdfstring{$\bm{E_\infty}$}{E∞}-algebras}\label{subsec:explicit-comparison} Let $G$ be a finite group and let $\mathcal O,\mathcal P$ be genuine $G$-$E_\infty$-operads. As a consequence of Theorem~\ref{thm:change-of-operad-categories} we have equivalences of quasi-categories
\begin{equation}\label{eq:comparison-alg-cat-zig-zag}
\Alg_{\mathcal O}(\cat{$\bm G$-Cat})^\infty_{\mathcal F}\xrightarrow{(\pr_{\mathcal O}^*)^\infty} \Alg_{\mathcal O\times\mathcal P}(\cat{$\bm G$-Cat})^\infty_{\mathcal F}\xrightarrow{\cat{L}\pr_{\mathcal P!}}\Alg_{\mathcal P}(\cat{$\bm G$-Cat})^\infty_\mathcal F;
\end{equation}
as the final result in this section, we want to represent this equivalence by an explicit functor on the pointset level. This construction is a generalization of the functor from parsummable categories to permutative categories we constructed in \cite{perm-parsum-categorical}. In fact, everything we do here works in slightly greater generality without any extra cost, so we treat it accordingly.

\begin{defi}
Let $\mathcal F$ be a family of finite subgroups of a discrete group $G$. An \emph{$\mathcal F$-$E_\infty$-operad} is a $\Sigma$-free operad $\mathcal O$ in $\cat{$\bm G$-Cat}$ such that for every $n\ge0$ the unique map $\mathcal O(n)\to *$ is a $\mathcal G_{\mathcal F,\Sigma_n}$-equivalence.
\end{defi}

In particular, if $G$ is finite and $\mathcal F=\mathcal A\ell\ell$, this recovers the notion of genuine $G$-operads; on the other hand, for $\mathcal F=\mathcal T\!riv$ the family consisting only of the trivial subgroup, an $\mathcal F$-$E_\infty$-operad is the same as an underlying $E_\infty$-operad, i.e.~an operad in $\cat{$\bm G$-Cat}$ that becomes an $E_\infty$-operad in the usual sense after forgetting the $G$-action. In particular, any na\"ive $G$-$E_\infty$-operad is a $\mathcal T\!riv$-$E_\infty$-operad.

Throughout, fix a family $\mathcal F$, and let $\pi\colon\mathcal O\to\mathcal P$ be a map of $\mathcal F$-$E_\infty$-operads such that $\mathcal O(n)^\phi\to\mathcal P(n)^\phi$ is \emph{strictly} surjective on objects for every $H\in\mathcal F,\phi\colon H\to\Sigma_n$.

\begin{constr}
Given an $\mathcal O$-algebra $C$, we define a $G$-category $\pi_\lozenge(C)$ as follows:
\begin{enumerate}
\item The objects of $\pi_\lozenge(C)$ agree with the objects of $\cat{P}_{\mathcal P}(\forget_{\mathcal O}C)$, i.e.~they are given by equivalence classes $[P;X]$ with $P\in\mathcal P(n), X\in C^n,n\ge0$ with $[P;X]=[Q;Y]$ for $Q\in\mathcal P(n), Y\in C^n$ if and only if there exists a permutation $\sigma\in\Sigma_n$ (necessarily unique) such that $P.\sigma=Q$ and $\sigma.Y=X$.
\item Given two objects $a,b\in\pi_\lozenge(C)$, the homomorphisms $A\to B$ are given by equivalence classes $[O', X'; f; O, X]$ with
\begin{enumerate}
    \item $O\in\mathcal O(m), X\in C^m$ such that $a=[\pi(O);X]$
    \item $O'\in\mathcal O(n), X'\in C^n$ such that $b=[\pi(O');X']$
    \item $f\colon O_*(X)\to O'_*(X')$ a morphism in $C$.
\end{enumerate}
Here $(O', X'; f; O, X)$ and $(N', Y'; g; N, Y)$ represent the same equivalence class if and only if there are $\sigma\in\Sigma_m,\sigma'\in\Sigma_n$ such that
\begin{enumerate}
    \item $\pi(O)=\pi(N).\sigma$ (whence in particular $Y=\sigma.X$)
    \item $\pi(O')=\pi(N').\sigma'$ (whence in particular $Y'=\sigma'.X'$)
    \item $g$ agrees with the composite
    \begin{equation*}\hskip-41pt
        N_*(Y)=(N.\sigma)_*(X)\xrightarrow{[O,N.\sigma]} O_*(X)\xrightarrow{f} O'_*(X')=(O'.\sigma')_*(Y')\xrightarrow{[N',O'.(\sigma')^{-1}]} N'_*(Y').
    \end{equation*}
    where for every $r\ge0, A,B\in\mathcal O(r)$ we write $[B,A]\colon A_*\Rightarrow B_*$ for the action of the unique edge $(B,A)\colon A\to B$ in $\mathcal O(r)\simeq*$.
\end{enumerate}
Note that there are always \emph{unique} permutations $\sigma,\sigma'$ satisfying the first two conditions by $\Sigma$-freeness. We omit the routine verification that the above is indeed an equivalence relation.
\item The composition $[N',Y';g;N,Y]\circ[O',X';f;O,X]$ is $[N',Y';h;O,X]$ where $h$ is the composite
\begin{equation*}
O_*(X)\xrightarrow{f} O'_*(X')=(O'.\theta)_*(Y)\xrightarrow{[N,O'.\theta]} N_*(Y)\xrightarrow{g} N'_*(Y');
\end{equation*}
here $\theta$ is the unique permutation such that $\pi(O').\theta=\pi(N)$ (hence in particular $X'=\theta.Y$).
\item $G$ acts on both objects and morphisms diagonally, i.e.~$g.[P;X]=[g.P; g.X]$ and $g.[O,X';f;O,X]=[g.O,g.X',g.f, g.O,g.X]$.
\end{enumerate}
\end{constr}

\begin{lemma}\label{lemma:pi-lozenge-hom-sets}
The above is a well-defined $G$-category. Moreover, for every $O\in\mathcal O(m), X\in C^m$ the identity of $[\pi(O);X]$ is given by $[O,X;\id_{O_*(X)};O,X]$, and for any further $O'\in\mathcal O(n), X'\in C^n$ the map
\begin{equation}\label{eq:pi-lozenge-hom-bij}
\begin{aligned}
\Hom_{C}(O_*(X),O'_*(X'))&\to \Hom_{\pi_\lozenge C}([\pi(O);X],[\pi(O');X'])\\
f&\mapsto [O',X';f;O,X]
\end{aligned}
\end{equation}
is bijective.
\begin{proof}
First, we will show that composition is independent of the choices of representatives, for which we pick $a,b,c\in\pi_\lozenge(C)$ and let
\begin{align}
(N_1',Y_1';g_1;N_1,Y_1)&\sim(N_2',Y_2';g_2;N_2,Y_2)\label{eq:relation-g}\\
(O_1',X_1';f_1;O_1,X_1)&\sim(O_2',X_2';f_2;O_2,X_2)\label{eq:relation-f}
\end{align}
represent morphisms $b\to c$ and $a\to b$, respectively. We write $\sigma$ for the unique permutation with $\pi(O_1)=\pi(O_2).\sigma$ (whence $X_2=\sigma.X_1$), $\tau$ for the unique permutation with $\pi(N_1)=\pi(N_2).\tau$ (whence $Y_2=\tau.Y_1$), and analogously we define $\sigma'$ and $\tau'$. Moreover, as $[N_1;Y_1]=[O_1';X_1']$ there is a unique permutation $\theta$ with $\pi(O_1').\theta=\pi(N_1)$ (whence $X_1'=\theta.Y_1$), and similarly we get $\zeta$ with $\pi(O_2').\zeta=\pi(N_2)$ and $X_2'=\zeta.Y_2$. But then $\pi(N_1)=\pi(O_1').\theta=\pi(O_2').\sigma'\theta$ as well as $\pi(N_1)=\pi(N_2).\tau=\pi(O_2').\zeta\tau$, hence
\begin{equation}\label{eq:relation-permutations}
\sigma'\theta=\zeta\tau
\end{equation}
by $\Sigma$-freeness.

We now have to show that
\begin{equation*}
(N_1',Y_1';g_1 \circ[N_1,O'_1.\theta] \circ f_1; O_1,X_1)\sim
(N_2',Y_2';g_2 \circ[N_2,O'_2.\zeta] \circ f_2; O_2,X_2).
\end{equation*}
Plugging in the definition of the equivalence relation, this amounts to saying that the total rectangle in the diagram
\begin{equation*}\hskip-37.85pt\hfuzz=38pt
\begin{tikzcd}[cramped]
O_{2*}(X_2)\arrow[r, "f_2"]\arrow[d,equal] &[-1em] O_{2*}'(X_2')\arrow[r,equal] &[-1em] (O_2'.\zeta)_*(Y_2)\arrow[r, "{[N_2,O_2'.\zeta]}"] &[1.5em] N_{2*}(Y_2)\arrow[d,equal]\arrow[r, "g_2"] &[-1em] N_{2*}'(Y_2')\\
(O_2.\sigma)_*(X_1)\arrow[d, "{[O_1,O_2.\sigma]}"'] & (O_1'.(\sigma')^{-1})_*(X_2')\arrow[d,equal]\arrow[u, "{[O_2',O_1'.(\sigma')^{-1}]}"'] && (N_2.\tau)_*(Y_1)\arrow[u,equal]\arrow[d,"{[N_1,N_2.\tau]}"'] & (N_1'.(\tau')^{-1})_*(Y_2')\arrow[d,equal]\arrow[u, "{[N_2',N_1'.(\tau')^{-1}]}"']\\
O_{1*}(X_1)\arrow[r, "f_1"'] & O_{1*}'(X_1')\arrow[r,equal] & (O_1'.\theta)_*(Y_1)\arrow[r, "{[N_1,O_1'.\theta]}"'] & N_{1*}(Y_1) \arrow[r, "g_1"'] & N_{1*}'(Y_1')
\end{tikzcd}
\end{equation*}
commutes. But indeed, the left hand rectangle and the right hand rectangle commute by the relations $(\ref{eq:relation-f})$ and $(\ref{eq:relation-g})$, respectively, so it only remains to check commutativity of the rectangle in the middle. For this we compute
\begin{align*}
[N_2,O_2'.\zeta]_{Y_2}[O_2',O_1'.(\sigma')^{-1}]_{X_2'}\; &=\; [N_2.\tau, O_2'.\zeta\tau]_{Y_1}[O_2'.\sigma',O_1']_{X_1'}\\
&\stackrel{\kern-8pt(\ref{eq:relation-permutations})\kern-7pt}{=}\; [N_2.\tau,O_2'.\sigma'\theta]_{Y_1}[O_2'.\sigma'\theta, O_1'.\theta]_{Y_1}\\
&=\;[N_2.\tau,O_1'.\theta]_{Y_1}
\end{align*}
where the unlabelled equalities use the functoriality of the $\mathcal O$-action or its compatibility with the symmetric group actions. Plugging this in, we then get
\begin{align*}
[N_1,N_2.\tau]_{Y_1}[N_2,O_2'.\zeta]_{Y_2}[O_2',O_1'.(\sigma')^{-1}]_{X_2'}&=[N_1,N_2.\tau]_{Y_1}[N_2.\tau,O_1'.\theta]_{Y_1}\\
&=[N_1,O_1'.\theta]_{Y_1}
\end{align*}
as desired. This completes the proof that composition is well-defined.

Before we prove that composition is unital and associative, let us show that $(\ref{eq:pi-lozenge-hom-bij})$ is bijective.

For injectivity we have to show that $(O',X';f;O,X)\sim(O',X';g;O,X)$ only if $f=g$. Indeed, the permutations $\sigma,\sigma'$ from the definition of the equivalence relation are the respective identities in this case, whence $g=[O',O']\circ f\circ[O,O]=\id\circ f\circ \id=f$ by design.

For surjectivity, we let $[N',Y';g;N,Y]$ be any morphism $[\pi(O);X]\to[\pi(O');X']$. Then $[\pi(N);Y]=[\pi(O);X]$, so we find $\sigma$ with $\pi(N)=\pi(O).\sigma$ and $X=\sigma.Y$, and similarly we get $\sigma'$ with $\pi(N')=\pi(O').\sigma'$ and $X'=\sigma'.Y'$. But then
\begin{equation*}
[N',Y';g;N,Y]=[O',X';[O', N'.\sigma']\circ g\circ[N,O.\sigma];O,X]
\end{equation*}
by definition of the equivalence relation, proving surjectivity.

Now we can easily prove that $[O,X;\id;O,X]$ is a right unit for $[\pi(O);X]$: indeed, if $[\pi(O');X']$ is any other object, then any morphism $[\pi(O);X]\to[\pi(O');X']$ can be written as $[O',X';f;O,X]$ for $f\colon O_*(X)\to O'_*(X')$ by the above, and then
\begin{equation*}
[O',X';f;O,X][O,X;\id;O,X]=[O',X';f\circ\id;O,X]=[O',X';f;O,X]
\end{equation*}
by definition of the composition in the special case $N=O',Y=X'$. Likewise, one shows that $[O,X;\id;O,X]$ is also a left unit, and one deduces that composition in $\pi_\lozenge C$ is associative from the fact that it is so in $C$.

It remains to show that the above makes $\pi_\lozenge C$ into a $G$-category. For this we first check that $g.\blank\colon \pi_\lozenge C\to\pi_\lozenge C$ is well-defined. Indeed, well-definedness on objects is immediate; to check that this is also well-defined on morphisms, let $(O',X';f_1;O,X)\sim (N',Y';f_2;N,Y)$ and write $\sigma,\sigma'$ for the unique permutations with $\pi(O)=\pi(N).\sigma, Y=\sigma.X$ and $\pi(O')=\pi(N').\sigma', Y'=\sigma'.X'$. Then $f_2=[N',O'.(\sigma')^{-1}]\circ f_1\circ [O, N.\sigma]$, so
\begin{equation*}
g.f_2=[g.N', g.O'.(\sigma')^{-1}]\circ g.f_1\circ [g.O, g.N.\sigma]
\end{equation*}
as $C$ is an $\mathcal O$-algebra in $\cat{$\bm G$-Cat}$. Thus, as desired
\begin{equation*}
(g.O',g.X';g.f_1;g.O,g.X)\sim (g'.N',g.Y';g.f_2; g.N,g.Y).
\end{equation*}

With this established, the equality $g.[O,X;\id;O,X]=[g.O,g.X;\id;g.O,g.X]$ (following directly from the definition) shows that $g.\blank$ preserves identities, while
\begin{align*}
&g.([O'',X'';f_2;O',X'][O',X';f_1;O,X])\\
&\qquad=g.[O'',X'';f_2f_1;O,X]=[g.O'',g.X'';g.(f_2f_1);g.O,g.X]\\
&\qquad=[g.O'',g.X'';g.f_2;g.O',g.X'][g.O',g.X';g.f_1;g.O,g.X]\\
&\qquad=\big(g.[O'',X'';f_2;O',X']\big)\big(g.[O',X';f_1;O,X]\big)
\end{align*}
shows that $g.\blank$ is compatible with compositions, hence a functor. Finally, it is clear from the definition that $gh.\blank = (g.\blank)\circ (h.\blank)$ and $1.\blank=\id$, so this defines a $G$-action as desired.
\end{proof}
\end{lemma}

\begin{constr}
For $F\colon C\to D$ a map in $\Alg_{\mathcal O}(\cat{$\bm G$-Cat})$, we define $\pi_\lozenge F\colon\pi_\lozenge C\to\pi_\lozenge D$ via $(\pi_\lozenge F)[P;X]=[P;F(X)]$ and $(\pi_\lozenge F)[O',X';f;O,X]=[O',F(X');F(f);O,F(X)]$.
\end{constr}

\begin{lemma}\label{lemma:pi-lozenge-iota}
The above is a well-defined and $G$-equivariant functor $\pi_\lozenge C\to\pi_\lozenge D$. Moreover, this makes $\pi_\lozenge$ into a functor $\Alg_{\mathcal O}(\cat{$\bm G$-Cat})\to\cat{$\bm G$-Cat}$.
\begin{proof}
It is clear that $\pi_\lozenge F$ is well-defined on objects. To check that it is also well-defined on morphisms, we first observe that if $(O',X';f_1;O,X)$ represents a morphism in $\pi_\lozenge(C)$, then $F(f_1)$ is a morphism $O_*(F(X))=F(O_*(X))\to F(O_*(X'))=O_*(F(X'))$ as $F$ is a map of $\mathcal O$-algebras. To check that $\pi_\lozenge F$ is independent of the choice of representative, let $(O',X';f_1;O,X)\sim (N',Y';f_2;N,Y)$, i.e.~there are permutations $\sigma,\sigma'$ with $\pi(O)=\pi(N).\sigma, Y=\sigma.X$ and $\pi(O')=\pi(N').\sigma', Y'=\sigma'.X'$, and $f_2=[N',O'.(\sigma')^{-1}]f_1[O,N.\sigma]$. As $F$ is a map of $\mathcal O$-algebras, we then have $F(f_2)=[N',O'.(\sigma')^{-1}]F(f_1)[O,N.\sigma]$, whence $(O',F(X');F(f_1);O,F(X))\sim (N', F(Y');F(f_2);N,F(Y))$ as desired.

The equality
\begin{equation*}
(\pi_\lozenge F)[O,X;\id;O,X]=[O,F(X);F(\id);O,F(X)]=[O,F(X);\id;O,F(X)]
\end{equation*}
shows that $\pi_\lozenge F$ preserves identities. Similarly, one shows that it is compatible with compositions, hence a functor.

We have $(\pi_\lozenge F)(g.[P;X])=(\pi_\lozenge F)[g.P;g.X]=[g.P;F(g.X)]=[g.P, g.F(X)]=g.(\pi_\lozenge F)[P;X]$ by $G$-equivariance of $F$, i.e.~$\pi_\lozenge F$ commutes with the $G$-action on objects. Analogously, one shows that $\pi_\lozenge F$ commutes with the $G$-action on morphisms, i.e.~it is a $G$-equivariant functor as claimed.

Finally, it is clear from the definitions that $\pi_\lozenge(\id)=\id$ and $\pi_\lozenge(F_2F_1)=(\pi_\lozenge F_2)(\pi_\lozenge F_1)$, i.e.~$\pi_\lozenge$ indeed defines a functor to $\cat{$\bm G$-Cat}$.
\end{proof}
\end{lemma}

\begin{lemma}
Let $C$ be an $\mathcal O$-algebra. Then we have a natural $\mathcal F$-equivalence $\iota\colon C\to\pi_\lozenge C$ of $G$-categories sending an object $X$ to $[1;X]$ and a map $f\colon X\to Y$ to $[1,Y;f;1,X]$.
\begin{proof}
It follows immediately from the definitions that $\iota$ is well-defined, equivariant, and natural.

Now let $H\in\mathcal F$ arbitrary; we have to show that $\iota^H\colon C^H\to (\pi_\lozenge C)^H$ is an equivalence. For this we observe that $\iota$ is fully faithful by Lemma~\ref{lemma:pi-lozenge-hom-sets}, whence so is $\iota^H$ as limits of fully faithful functors are again fully faithful. It then only remains to show that $\iota^H$ is essentially surjective. For this we let $P\in\mathcal P(n),X\in C^n$ such that $[P;X]\in(\pi_\lozenge C)^H$. Then we have $h.(P;X)=(h.P;h.X)\sim (P;X)$ for every $h\in H$, i.e.~there exists a $\sigma(h)\in\Sigma_n$ (necessarily unique) such that $h.P=P.\sigma(h)$ and $h.X=\sigma(h)^{-1}.X$. For any further $k\in H$, we have
\begin{equation*}
P.\sigma(hk)=(hk).P=h.k.P=h.P.\sigma(k)=P.\sigma(h)\sigma(k)
\end{equation*}
whence $\sigma(hk)=\sigma(h)\sigma(k)$ by $\Sigma$-freeness, i.e.~$\sigma$ is a homomorphism $H\to\Sigma_n$. With this established, the relation $h.P=P.\sigma(h)$ precisely tells us that $P\in\mathcal P(n)^\sigma$.

Now let $O\in\mathcal O(n)^\sigma$ be a preimage of $P$ (which exists as $\pi^\sigma$ was assumed to be surjective). Then $h.(O_*(X))=(h.O)_*(h.X)=(O.\sigma)_*(h.X)=O_*(\sigma.h.X)=O_*(X)$, i.e.~$O_*(X)\in C^H$. We claim that $\iota(O_*(X))=[1;O_*X]$ is isomorphic to $[P;X]$ in $(\pi_\lozenge C)^H$, which will then complete the proof of the lemma. Indeed, we have a map $[1,O_*(X);\id;O,X]\colon [P;X]\to[1;O_*(X)]$ in $\pi_\lozenge C$ with inverse $[O,X;\id;1,O_*(X)]$, so it only remains to show that this is $H$-fixed, for which we compute
\begin{align*}
h.[1,O_*(X);\id;O,X]&=[h.1, h.(O_*(X));h.\id;h.O,h.X]=[1, O_*(X);\id; h.O,h.X]\\
&=[1, O_*(X); \id; O.\sigma(h), \sigma(h)^{-1}.X]=[1, O_*(X);\id; O, X]
\end{align*}
where the last step uses the definition of the equivalence relation.
\end{proof}
\end{lemma}

By $2$-out-of-$3$ we immediately conclude:

\begin{cor}
The functor $\pi_\lozenge\colon\Alg_{\mathcal O}(\cat{$\bm G$-Cat})_{\mathcal F}\to\cat{$\bm G$-Cat}_{\mathcal F}$ is homotopical.\qed
\end{cor}

Next, we will put a $\mathcal P$-algebra structure on $\pi_\lozenge C$:

\begin{constr}
For every $Q\in\mathcal P(r)$ we define a functor $Q_*\colon (\pi_\lozenge C)^r\to\pi_\lozenge C$ as follows: on objects, $Q_*$ is given by the usual action on $\Ob\cat{P}_{\mathcal P}\forget_\mathcal O C=\Ob\pi_\lozenge C$, i.e.
\begin{equation*}
    Q_*([P_1,X_1],\dots,[P_r;X_r])=[Q\circ(P_1,\dots,P_r);X_1,\dots,X_r].
\end{equation*}
Given morphisms $[O_i',X_i';f_i;O_i,X_i]\colon[P_i;X_i]\to[P_i';X_i']$ for $i=1,\dots,r$
we moreover define
\begin{align*}
&Q_*([O_1',X_1';f_1;O_1,X_1],\dots,[O_r',X_r';f_r;O_r,X_r])\\
&\qquad=[N\circ(O_1',\dots,O_r'), X_1',\dots,X_r'; N_*(f_1,\dots,f_r); N\circ (O_1,\dots,O_r); X_1,\dots,X_r]
\end{align*}
where $N\in\mathcal O(r)$ is some preimage of $Q$ under $\pi$.

Moreover, if $Q'$ is any other object of $\mathcal P(r)$, then we define a natural transformation $[Q',Q]\colon Q_*\Rightarrow Q'_*$ via
\begin{equation*}\hskip-17.13pt
\begin{aligned}
&[Q',Q]_{[P_1;X_1],\dots,[P_r;X_r]}\\
&\qquad=[N'\circ (O_1,\dots,O_r), X_\bullet; [N'\circ (O_1,\dots,O_r), N\circ (O_1,\dots,O_r)]_{X_\bullet}; N\circ (O_1,\dots, O_r);X_\bullet]
\end{aligned}
\end{equation*}
where $N',N\in\mathcal O(r)$ are preimages of $Q'$ and $Q$, respectively, $O_i$ is a preimage of $P_i$ for $i=1,\dots,r$, and we abbreviate $X_\bullet=X_1,\dots,X_r$.
\end{constr}

\begin{prop}
The above is independent of choices and makes $\pi_\lozenge C$ into a $\mathcal P$-algebra. This way, $\pi_\lozenge$ becomes a functor $\Alg_{\mathcal O}(\cat{$\bm G$-Cat})\to\Alg_{\mathcal P}(\cat{$\bm G$-Cat})$.
\begin{proof}
First, we show that $Q_*\colon (\pi_\lozenge C)^r\to\pi_\lozenge C$ is a well-defined functor, independent of all choices. This is clear on objects. On morphisms, we first observe that $\pi(N\circ (O_1,\dots,O_r))=\pi(N)\circ (\pi(O_1),\dots,\pi(O_r))=Q\circ (\pi(O_1),\dots,\pi(O_r))$ as $\pi$ is a map of operads, whence
\begin{equation*}
[\pi(N\circ (O_1,\dots, O_r)),X_\bullet]=Q_*([P_1;X_1],\dots,[P_r,X_r]),
\end{equation*}
and similarly
\begin{equation*}
    [\pi(N\circ (O_1',\dots, O_r'));X'_\bullet]=Q_*([P_1';X_1'],\dots,[P_r';X_r']).
\end{equation*}
As moreover
$N_*(O_{1*}(X_1),\dots)=(N\circ(O_1,\dots,O_r))(X_\bullet)$ and $N_*(O_{1*}'(X_1
),\dots)=(N\circ(O_1',\dots,O_r'))(X_\bullet)$ since $C$ is an $\mathcal O$-algebra, we see that the above indeed defines a morphism $Q_*([P_1;X_1],\dots)\to Q_*([P_1';X_1'],\dots)$.

This is independent of the choice of representatives: indeed, if $(\bar O_i',\bar X_i';\bar f_i;\bar O_i,\bar X_i)\sim(O_i',X_i';f_i;O_i,X_i)$, then we let $\sigma_i$ be the unique permutation with $\pi(O_i)=\pi(\bar O_i).\sigma_i$, $\bar X_i=\sigma.X_i$, and define $\sigma_i'$ analogously. If we then write $\sigma$ for the block sum $\sigma_1\oplus\cdots\oplus\sigma_r$, and $\sigma'=\sigma_1'\oplus\cdots\oplus\sigma'_r$, then $\bar X_\bullet=\sigma.X_\bullet$ and $\bar X'_\bullet=\sigma'.X_\bullet$ as well as
\begin{align*}
(N\circ (O_1,\dots,O_r)\big).\sigma&=N\circ(O_1.\sigma_1,\dots,O_r.\sigma_r)\\
(N\circ (O_1',\dots,O_r')\big).\sigma'&=N\circ(O_1'.\sigma_1',\dots,O_r'.\sigma_r')
\end{align*}
because $\mathcal O$ is an operad. Thus, the desired relation $[N\circ (O_1',\dots),X'_\bullet;N_*(f_1,\dots);\allowbreak N\circ(O_1,\dots), X_\bullet]=[N\circ (\bar O_1',\dots),\bar X_\bullet;N_*(\bar f_1,\dots);N\circ(\bar O_1,\dots),\bar X_\bullet]$ is equivalent to asking that $N_*(\bar f_1,\dots)$ agree with the composite
\begin{equation*}
[N\circ(\bar O_1',\dots), N\circ (O_1'.(\sigma_1')^{-1},\dots)]\circ
N_*(f_1,\dots,f_r) \circ
[N\circ(O_1,\dots), N\circ(\bar O_1.\sigma_1,\dots)]
\end{equation*}
which in turn follows from $\bar f_i=[\bar O', O'.(\sigma'_i)^{-1}]\circ f_i\circ[O_i,\bar O.\sigma_i]$ and $C$ being an $\mathcal O$-algebra.

The above is also independent of the choice of preimage $N$: if $\bar N$ is any other preimage, then
\begin{equation*}
\pi(\bar N\circ (O_1,\dots))=\pi(\bar N)\circ (\pi(O_1),\dots) = \pi(N) \circ (\pi(O_1),\dots) = \pi(N\circ (O_1,\dots))
\end{equation*}
and similarly $\pi(\bar N\circ (O_1',\dots))=\pi(N\circ (O_1',\dots))$, so the desired relation amounts to asking that
\begin{equation*}\hskip-10.48pt\hfuzz=11pt
\bar N_*(f_1,\dots,f_r)=[\bar N\circ(O_1',\dots),N\circ(O_1',\dots)]\circ N_*(f_1,\dots,f_r)\circ [N\circ (O_1,\dots),\bar N\circ(O_1,\dots)].
\end{equation*}
As
\begin{align*}
[\bar N\circ(O_1',\dots),N\circ(O_1',\dots)]_{X'_\bullet}&=[\bar N,N]_{O'_1(X'_1),\dots}\\
[N\circ(O_1,\dots),\bar N\circ (O_1,\dots)]_{X_\bullet}&=[\bar N,N]^{-1}_{O_{1*}(X_1),\dots},
\end{align*}
this follows from naturality of $[\bar N,N]$.

To show that $Q_*$ is a functor, consider maps $[O''_i,X''_i;f_i';O'_i,X_i']$ and  $[O'_i,X_i';f_i;O_i,X_i]$ for $i=1,\dots,r$; then
\begin{align*}
&Q_*([O''_1,X_1'';f_1';O_1',X_1'][O_1',X_1';f_1;O_1,X_1],\dots)\\
&\qquad=Q_*([O''_1,X_1'';f_1'f_1; O_1,X_1],\dots)\\
&\qquad=[N\circ (O_1'', \dots); X''_\bullet; N_*(f_1'f_1,\dots); N\circ (O_1,\dots);X_\bullet]\\
&\qquad=[N\circ (O_1'', \dots); X''_\bullet; N_*(f_1',\dots)N_*(f_1,\dots); N\circ (O_1,\dots);X_\bullet]\\
&\qquad=Q_*([O''_1,X_1'';f_1';O_1',X_1'],\dots)Q_*([O_1',X_1';f_1;O_1,X_1],\dots),
\end{align*}
so $Q_*$ is compatible with compositions. Analogously, one shows that $Q_*$ preserves identities.

A similar computation as the one for $Q_*$ shows that $[Q',Q]$ is independent of the choice of the lifts $N',N\in\mathcal O(n)$. With this established, naturality amounts to the relation
\begin{equation*}\hskip-32.65pt
\begin{aligned}
&[N'\circ (O_1',\dots), X'_\bullet;N'(f_1,\dots);N'\circ(O_1,\dots),X_\bullet][N'\circ (O_1,\dots);X_\bullet;[N',N];N\circ(O_1,\dots),X_\bullet]\\
&\qquad=[N'\circ (O_1',\dots), X'_\bullet;[N',N];N\circ(O_1',\dots),X'_\bullet][N\circ (O_1',\dots), X'_\bullet;N(f_1,\dots);N\circ(O_1,\dots),X_\bullet]
\end{aligned}
\end{equation*}
which follows immediately from naturality of $[N',N]$. Similarly, one shows that $[Q'',Q'][Q',Q]=[Q'',Q]$, using that $[N'',N'][N',N]=[N'',N]$ for all $N,N',N''\in\mathcal O(r)$.

Altogether, we have therefore constructed a functor $\alpha_r\colon\mathcal P(r)\times(\pi_\lozenge C)^{r}\to\pi_\lozenge C$ given on objects by
\begin{equation*}
(Q;[P_1;X_1],\dots)\mapsto Q_*([P_1;X_1],\dots)=[Q\circ(P_1,\dots);X_\bullet]
\end{equation*}
and on morphisms by
\begin{align*}
((Q',Q);[O_1',X_1';f_1;O_1,X_1],\dots)&\mapsto [Q',Q]\circ Q_*([O_1',X_1';f_1;O_1,X_1],\dots)\\&\qquad=Q'_*([O_1',X_1';f_1;O_1,X_1],\dots)\circ [Q',Q]
\end{align*}
which is in turn explicitly given by
\begin{align*}
&[N'\circ(O_1',\dots),X_\bullet'; [N',N]\circ N_*(f_1,\dots); N\circ(O_1,\dots),X_\bullet]\\
&\qquad=[N'\circ(O_1',\dots),X_\bullet'; N'_*(f_1,\dots)\circ[N',N]; N\circ(O_1,\dots),X_\bullet]
\end{align*}
for arbitrary preimages $N,N'\in\mathcal O(r)$ of $Q,Q'$.

This functor is $G$-equivariant: this is clear on objects where this agrees with the usual $\mathcal P$-action on $\cat{P}_{\mathcal P}\forget_{\mathcal O}(C)$. On morphisms, we explicitly compute that
\begin{equation*}
g.((Q',Q);[O_1',X_1';f_1;O_1,X_1],\dots)=((g.Q',g.Q);[g.O_1',g.X_1';g.f_1;g.O_1,g.X_1],\dots)
\end{equation*}
is sent to
\begin{align*}
&[(g.N')\circ(g.O_1',\dots), g.X'_\bullet; [g.N',g.N]\circ(g.N)_*(g.f_1,\dots,); g.N\circ(g.O_1,\dots); g.X_\bullet]\\
&\qquad=[g.(N'\circ(O_1',\dots)),g.X'_\bullet; g.([N',N]\circ N_*(f_1,\dots)); g.(N\circ (O_1,\dots)), g.X_\bullet]
\end{align*}
which agrees with $g.[N\circ(O_1',\dots), X_\bullet; [N',N]\circ N_*(f_1,\dots); N\circ(O_1,\dots), X_\bullet]$ as desired.

Next, we will check that these functors indeed make $\pi_\lozenge C$ into a $\mathcal P$-algebra, i.e.~that the above maps are unital, associative, and compatible with the symmetric group actions. Again, all the required identities hold on the level of objects because $\cat{P}_{\mathcal P}(\forget_{\mathcal O}C)$ is a $\mathcal P$-algebra, so it only remains to check these statements on the level of morphisms. For this we make the following observation that will slightly simplify some computations: if $(O',X_1';f_1;O,X_1)$ and $(O',X_2';f_2,O,X_2)$ represent morphisms \emph{between the same pair of objects}, then necessarily $X_1=X_2$ and $X_1'=X_2'$ (as $\mathcal P$ is $\Sigma$-free); we will therefore simply write $[O';f_1;O]$ and $[O';f_2;O]$ for the corresponding equivalence classes below.

Now we check that $\alpha_r$ is compatible with the symmetric group actions, i.e.
\begin{equation*}
\alpha_r([Q',Q].\sigma;[O_1';f_1;O_1],\dots)=\alpha_r([Q',Q]; [O_{\sigma^{-1}(1)}'; f_{\sigma^{-1}(1)}; O_{\sigma^{-1}(1)}],\dots)
\end{equation*}
for all $\sigma\in\Sigma_r$. Indeed, if $N,N'$ are preimages of $Q,Q'$, then $N.\sigma,N.\sigma'$ are preimages of $Q.\sigma,Q.\sigma'$, so the the left hand side is given by
\begin{equation*}
[(N'.\sigma)\circ(O_1',\dots); [N'.\sigma,N.\sigma]\circ(N.\sigma)_*(f_1,\dots,f_r); (N.\sigma)\circ(O_1,\dots)].
\end{equation*}
Similarly, the right hand side can be computed as
\begin{equation*}
[N'\circ (O_{\sigma^{-1}(1)}',\dots); [N',N]\circ N_*(f_{\sigma^{-1}(1)},\dots); N\circ (O_\sigma^{-1}(1),\dots)],
\end{equation*}
so we even have an equality of representatives by the compatibility of the operad structure maps of $\mathcal O$ with the symmetric group actions, and the compatiblity of the action maps on $C$ with the symmetric group actions.

Similarly, one deduces associativity of the action on $\pi_\lozenge C$ from associativity of the action on $C$. For unitality, it suffices to observe that a preimage of $1\in\mathcal P(1)$ is given by $1\in\mathcal O(1)$, so that $1_*[O';f;O]=[1\circ O';1_*(f);1\circ O]=[O';f;O]$ as desired.

Finally, we have to show that $\pi_\lozenge F$ is a map of $\mathcal P$-algebras for every $\mathcal O$-algebra map $F\colon C\to D$. Again, this is clear on objects, while on morphisms we compute
\begin{align*}
\hskip-2.24pt\alpha_r([Q',Q];[O_1';Ff_1;O_1],\dots)&=[N'\circ(O_1',\dots);[N',N] N_*(Ff_1,\dots); N\circ(O_1,\dots)]\\
&=[N'\circ (O_1',\dots); F([N',N] N_*(f_1,\dots)); N\circ(O_1,\dots)]\\
&=(\pi_\lozenge F)\big(\alpha_r([Q',Q];[O_1';f_1;O_1],\dots)\big)
\end{align*}
where the middle equality uses that $F$ is a map $\mathcal O$-algebras. This completes the proof of the proposition.
\end{proof}
\end{prop}

\begin{thm}
The functor $(\pi_\lozenge)^\infty\colon\Alg_{\mathcal O}(\cat{$\bm G$-Cat})_{\mathcal F}^\infty\to\Alg_{\mathcal P}(\cat{$\bm G$-Cat})_{\mathcal F}^\infty$ is an equivalence and quasi-inverse to $(\pi^*)^\infty$.
\end{thm}

In particular, we see that $(\pi_\lozenge)^\infty$ is a model for $\cat{L}\pi_!$.

\begin{proof}
As we already know that $(\pi^*)^\infty$ is an equivalence (Theorem~\ref{thm:change-of-operad-categories}), it will suffice to show that $\pi_\lozenge\pi^*$ is homotopic to the identity. For this, we define for every $\mathcal P$-algebra $C$ a map $\epsilon\colon\pi_\lozenge\pi^*C\to C$ on objects via $[P;X]\mapsto P_*(X)$ and on morphisms via $[O',X';f;O,X]\mapsto \big(f\colon\pi(O)_*(X)\to \pi(O')_*(X')\big)$; we omit the routine verification that this is well-defined and a $G$-equivariant functor.

The functor $\epsilon$ is even a map of $\mathcal P$-algebras: this is clear on objects (where this is just the counit of $\cat{P}_{\mathcal P}\dashv \forget_{\mathcal P}$), and for the claim on morphisms we compute
\begin{align*}
&\epsilon(Q_*([O_1',X_1';f_1;O_1,X_1],\dots))\\
&\qquad=\epsilon[N\circ (O_1',\dots),X'_\bullet;\pi(N)_*(f_1,\dots); N\circ(O_1,\dots), X_\bullet]\\
&\qquad=\pi(N)_*(f_1,\dots)=Q_*(\epsilon[O_1',X_1';f_1;O_1,X_1],\dots)
\end{align*}
(where $N$ is a preimage of $Q$) and similarly
\begin{equation*}
\epsilon{[Q',Q]_{[P_1;X_1],\dots}}=[Q',Q]_{P_{1*}(X_1),\dots}.
\end{equation*}
Moreover, plugging in the definitions shows that $\epsilon$ is a natural transformation $\pi_\lozenge\pi^*\Rightarrow\id$, so it only remains to show that it is an $\mathcal F$-equivalence for every $\mathcal P$-algebra $C$. But as a map of $G$-categories $\epsilon$ is left-inverse to the $\mathcal F$-equivalence $\iota\colon C=\pi^*C\to\pi_\lozenge\pi^*C$ from Lemma~\ref{lemma:pi-lozenge-iota}, so the claim follows by $2$-out-of-$3$.
\end{proof}

\begin{defi}
Let $\mathcal O,\mathcal P$ be $E_\infty$-$\mathcal F$-operads. We write $\Sigma^{\mathcal O}_{\mathcal P}$ for the composite
\begin{equation*}
\Alg_{\mathcal O}(\cat{$\bm G$-Cat})\xrightarrow{\pr_{\mathcal O}^*} \Alg_{\mathcal O\times\mathcal P}(\cat{$\bm G$-Cat})\xrightarrow{\pr_{\mathcal P\lozenge}} \Alg_{\mathcal P}(\cat{$\bm G$-Cat}).
\end{equation*}
\end{defi}

\begin{cor}
The functor $\Sigma^{\mathcal O}_{\mathcal P}$ is homotopical. The induced functor on associated quasi-categories is an equivalence naturally equivalent to $(\ref{eq:comparison-alg-cat-zig-zag})$.\qed
\end{cor}

\section{\texorpdfstring{$G$}{G}-global vs.~\texorpdfstring{$G$}{G}-equivariant algebras}\label{sec:g-glob-vs-g-equiv}
In this section we will compare the $G$-global and $G$-equivariant algebras introduced above to each other. The argument for the simplicial and categorical cases will be almost entirely parallel, so we denote by $\mathscr C$ one of the categories $\cat{Cat}$ and $\cat{SSet}$.

\subsection{Globally twisted \texorpdfstring{$\bm G$}{G}-operads} For our comparison it will be useful to first introduce a variant of $G$-global operads that combines the operadic structure with the $\core(\mathcal M)$-action used to define the $G$-global weak equivalences:

\begin{defi}
A \emph{globally twisted $G$-operad} (or just \emph{twisted $G$-operad} for short) is an operad $\mathcal O$ in $G\text{-}\mathscr C$ together with a monoid homomorphism $\core(\mathcal M)\to \mathcal O(1)^G$.
\end{defi}

The attribute `twisted' refers to the homomorphism $\core(\mathcal M)\to \mathcal O(1)^G$ that provides an additional group action on any algebra over a twisted $G$-operad; in particular, we will see below (Construction~\ref{constr:external-twisted-diagonal}) that the correct notion of \emph{underlying $G$-space} or \emph{underlying $G$-category} of such an algebra will have to take this action into account.

As usual, however, we will keep the homomorphism from $\core(\mathcal M)$ implicit most of the time and just call $\mathcal O$ itself a twisted $G$-operad.

\begin{constr}\label{constr:external-actions}
If $\mathcal O$ is a twisted $G$-operad, then we obtain for any $n\ge0$ a left $\core(\mathcal M)$-action on $\mathcal O(n)$ by restricting the operad structure map $\mathcal O(1)\times\mathcal O(n)\to\mathcal O(n)$ along the given homomorphism $\core(\mathcal M)\to\mathcal O(1)$. By construction, this commutes with the $G$-action on $\mathcal O(n)$, making the latter into a $(\core(\mathcal M)\times G)$-object; we caution the reader however that $\mathcal O$ is typically \emph{not} an operad in $\core(\mathcal M)\text{-}G\text{-}\mathscr C$ as the structure maps won't be $\core(\mathcal M)$-equivariant.

Analogously, we get $n$ commuting right $\core(\mathcal M)$-actions on $\mathcal O(n)$ via restriction of the structure map $\mathcal O(n)\times\mathcal O(1)^{\times n}\to\mathcal O(n)$. Together with the right $\Sigma_n$-action that is part of the operad structure, this assembles into an action of the wreath product $\Sigma_n\wr\core(\mathcal M)$ that again commutes with the $G$-action.
\end{constr}

\begin{constr}
If $\mathcal O$ is any twisted $G$-operad and $A$ is an $\mathcal O$-algebra in $G\text{-}\mathscr C$, then we can similarly restrict the $\mathcal O(1)$-action on $A$ to $\core(\mathcal M)$ via the given homomorphism, thereby equipping $A$ with the structure of a $(\core(\mathcal M)\times G)$-object in $\mathscr C$. The resulting forgetful functor $\Alg_{\mathcal O}(G\text{-}\mathscr C)\to \core(\mathcal M)\text{-}G\text{-}\mathscr C$ admits a left adjoint $\cat{P}$ given on objects by
\begin{equation*}
\cat{P}X=\coprod_{n\ge0} \mathcal O(n)\times_{\Sigma_n\wr\core(\mathcal M)}X^{\times n}
\end{equation*}
(with respect to the above right $\Sigma_n\wr\core(\mathcal M)$-action on $\mathcal O(n)$ and the natural action on $X^{\times n}$) and likewise on morphisms.
\end{constr}

\begin{prop}\label{prop:ext-g-global-model-structure}
Let $\mathcal O$ be a twisted $G$-operad. Then there is a unique model structure on $\Alg_{\mathcal O}(G\text{-}\mathscr C)$ in which a map $f$ is a weak equivalence or fibration if and only if $\forget f$ is a weak equivalence or fibration, respectively, in the $G$-universal model structure on $\core(\mathcal M)\text{-}G\text{-}\mathscr C$ (see Remarks~\ref{rk:g-universal-sset} and~\ref{rk:g-universal-cat}). We call this the \emph{$G$-global model structure}. It is right proper, $\mathscr C$-enriched (hence simplicial), and filtered colimits in it are homotopical.
\begin{proof}
For $\mathscr C=\cat{Cat}$ it suffices as before to show that the model structure transferred along the forgetful functor $\Alg_{\mathcal O}(\cat{$\bm G$-Cat})\to\cat{$\bm{\core(\mathcal M)}$-$\bm{G}$-Cat}$ exists, for which one can take the same path objects as before.

The argument in the simplicial case is analogous, except that we additionally use $\Ex^\infty$ or $\Sing|\blank|$ again to construct functorial fibrant replacements.
\end{proof}
\end{prop}

\begin{defi}
A twisted $G$-operad $\mathcal O$ is called \emph{$\Sigma$-free}, if the above $\Sigma_n\wr\core(\mathcal M)$-action on $\mathcal O(n)$ is free for all $n\ge 0$. A map of twisted $G$-operads is a map $f$ of $G$-operads such that $f(1)$ is compatible with the maps from $\core(\mathcal M)$. We call $f$ a \emph{$G$-global (weak) equivalence} if $f(n)$ is a $G\times(\Sigma_n\wr\core(\mathcal M))$-universal (weak) equivalence for every $n\ge 0$.
\end{defi}

\begin{prop}\label{prop:ext-g-global-change-of-operad}
Let $f\colon\mathcal O\to\mathcal P$ be a map of twisted $G$-operads. Then the adjunction
\begin{equation}\label{eq:ext-g-global-change-of-operad}
f_!\colon\Alg_{\mathcal O}(G\text{-}\mathscr C)\rightleftarrows\Alg_{\mathcal P}(G\text{-}\mathscr C) :\!f^*
\end{equation}
is a Quillen adjunction with respect to the model structures from Proposition~\ref{prop:ext-g-global-model-structure}. If $f$ is a $G$-global weak equivalence (for $\mathscr C=\cat{SSet})$ or $G$-global equivalence ($\mathscr C=\cat{Cat}$) and both $\mathcal O$ and $\mathcal P$ are $\Sigma$-free, then $(\ref{eq:ext-g-global-change-of-operad})$ is a Quillen equivalence.
\end{prop}

The proof will again rely on a monadicity argument:

\begin{lemma}
Let $\mathcal O$ be a twisted $G$-operad. Then the forgetful functor $\Alg_{\mathcal O}(G\text{-}\mathscr C)^\infty\to{\core(\mathcal M)\text{-} G}\text{-}\mathscr C^\infty$ is monadic.
\begin{proof}
It is clear that the forgetful functor is conservative. To prove that it is monadic it will then be enough to show as before that the forgetful functor of simplicial categories $\Alg_{\mathcal O}(G\text{-}\mathscr C)\to{\core(\mathcal M)\text{-} G}\text{-}\mathscr C$ preserves geometric realizations. This is in turn immediate from Proposition~\ref{prop:forget-geometric-realization} as ${\core(\mathcal M)\text{-}G}\text{-}\mathscr C\to G\text{-}\mathscr C$ is conservative (as a functor of $1$-categories) and preserves both tensors and small colimits.
\end{proof}
\end{lemma}

\begin{proof}[Proof of Proposition~\ref{prop:ext-g-global-change-of-operad}]
It is clear that $f^*$ is compatible with the forgetful functors, so that it preserves (and reflects) weak equivalences as well as fibrations; in particular, it is right Quillen.

Now assume that $f$ is a $G$-global (weak) equivalence and $\mathcal O$ and $\mathcal P$ are $\Sigma$-free. As in the proof of Theorem~\ref{thm:change-of-operad-sset}, it suffices that for every $X\in \core(\mathcal M)\text{-}G\text{-}\mathscr C$ the map
\begin{equation*}
\coprod_{n\ge0} f(n)\times_{\Sigma_n\wr\core(\mathcal M)} X^{\times n}\colon\coprod_{n\ge 0} \mathcal O(n)\times_{\Sigma_n\wr\core(\mathcal M)} X^{\times n}\to \coprod_{n\ge 0} \mathcal P(n)\times_{\Sigma_n\wr\core(\mathcal M)} X^{\times n}
\end{equation*}
is a $G$-global (weak) equivalence. However, each $f(n)\times X^{\times n}$ is a $G\times(\Sigma_n\wr\core(\mathcal M))$-global (weak) equivalence by assumption and $\Sigma_n\wr\core(\mathcal M)$ acts freely on both source and target, so the claim follows from Lemma~\ref{lemma:free-quotients-sset}/\ref{lemma:free-quotients}.
\end{proof}

\begin{defi}
A twisted $G$-operad $\mathcal O$ is called a \emph{twisted $G$-global $E_\infty$-operad} if it is $\Sigma$-free and the unique map $\mathcal O\to *$ to the terminal object is a $G$-global (weak) equivalence.
\end{defi}

\begin{ex}\label{ex:ext-g-global-operad}
We define the (categorical or simplicial) \emph{injection operad} $\mathcal I$ as follows, also cf.~\cite[Construction~A.1]{I-vs-M-1-cat} where the underlying operad of sets is denoted $\mathcal M$:

For any $n\ge0$, the $n$-ary operations are given by $\mathcal I(n)=E\Inj(\bm n\times\omega,\omega)$ where as usual $\bm n=\{1,\dots,n\}$ with the tautological $\Sigma_n$-action. The structure maps on $\mathcal I$ are given by juxtaposition and precomposition, i.e.~they are induced under $E$ by the maps
\begin{align*}
\Inj(\bm r\times\omega,\omega)\times\Inj(\bm{n_1}\times\omega,\omega)\times\cdots\times\Inj(\bm{n_r}\times\omega,\omega)&\to\Inj(\bm n\times\omega,\omega)\\
(u;v_1,\dots,v_r)&\mapsto u\circ (v_1\amalg\cdots\amalg v_r)
\end{align*}
where $n=n_1+\cdots+n_r$ and we have identified $\coprod_{k=1}^r (\bm{n_k}\times\omega)$ with $\bm n\times\omega$ in the obvious way.

We can make $\mathcal I$ (equipped with the trivial $G$-action) into a twisted $G$-operad via the inclusion $\core(\mathcal M)\hookrightarrow E\mathcal M\cong\mathcal I(1)$. This is a twisted $G$-global $E_\infty$-operad: for the freeness it suffices to observe that the natural $(\Sigma_n\wr\core(\mathcal M))$-action on $\bm n\times\omega$ is faithful so that $\Sigma_n\wr\core(\mathcal M)$ acts freely on $\Inj(\bm n\times\omega,\omega)$, whence on $\mathcal I(n)$. To show that $\mathcal I(n)$ is $G\times(\Sigma_n\wr\core(\mathcal M))$-globally contractible it suffices to observe that for every universal $H\subset\mathcal M$ and every $\phi\colon H\to G\times(\Sigma_n\wr\core(\mathcal M))$
\begin{equation*}
\mathcal I(\bm n\times\omega,\omega)^\phi=\mathcal I(\bm n\times\omega,\omega)^{\pr_{\Sigma_n\wr\core(\mathcal M)}\circ\phi}=
E\big(\Inj\big((\pr_{\Sigma_n\wr \core(\mathcal M)}\circ\phi)^*(\bm n\times\omega),\omega\big)^H\big)
\end{equation*}
which is non-empty and hence contractible because there exists an equivariant injection $(\pr_{\Sigma_n\wr \core(\mathcal M)}\circ\phi)^*(\bm n\times\omega)\to\omega$ by universality of $H$.
\end{ex}

\begin{cor}\label{cor:comparison-ext-g-global-operads}
Let $\mathcal O$ be any twisted $G$-global $E_\infty$-operad. Then there exists an explicit zig-zag of Quillen equivalences $\Alg_{\mathcal O}(G\text{-}\mathscr C)\leftrightarrow\Alg_{\mathcal I}(G\text{-}\mathscr C)$.
\begin{proof}
By Proposition~\ref{prop:ext-g-global-change-of-operad}, it suffices to observe that also $\mathcal O\times\mathcal I$ is a twisted $G$-global $E_\infty$-operad and that the projections $\mathcal O\gets\mathcal O\times\mathcal I\to\mathcal I$ are $G$-global (weak) equivalences.
\end{proof}
\end{cor}

\begin{thm}\label{thm:external-vs-internal}
Let $\mathcal O$ be any $G$-global $E_\infty$-operad in $E\mathcal M\text{-}G\text{-}\mathscr C$ and let $\mathcal P$ be a twisted $G$-global $E_\infty$-operad in $G\text{-}\mathscr C$. Then there is an explicit zig-zag of Quillen equivalences $\Alg_{\mathcal O}(E\mathcal M\text{-}G\text{-}\mathscr C)\leftrightarrow\Alg_{\mathcal P}(G\text{-}\mathscr C)$ with respect to the $G$-global model structures on both sides.
\end{thm}

For the proof of the theorem, we will first relate $\Alg_{\mathcal O}(E\mathcal M\text{-}G\text{-}\mathscr C)$ to the algebras over a suitable {twisted} $G$-global $E_\infty$-operad built from $\mathcal O$ and then appeal to the previous result. This relies on the following construction:

\begin{constr}\label{constr:twisted-product}
Let $\mathcal O$ be any operad in $E\mathcal M\text{-}G\text{-}\mathscr C$. We define an operad $\mathcal O\rtimes E\mathcal M$ in $G\text{-}\mathscr C$ as follows: the $n$-ary operations are given by $(\mathcal O\rtimes E\mathcal M)(n)=\mathcal O(n)\times E\mathcal M^{\times n}$ with the diagonal right $\Sigma_n$-action and the induced $G$-action. For $\mathscr C=\cat{SSet}$, the operad structure maps $\gamma^{\mathcal O\rtimes E\mathcal M}$ of $\mathcal O\rtimes E\mathcal M$ are given by
\begin{align*}
&\gamma^{\mathcal O\rtimes E\mathcal M}\big((o,u_\bullet); (f^{(1)},v^{(1)}_\bullet),\dots,(f^{(r)},v^{(r)}_\bullet)\big)\\
&\quad{} =
\big(
\gamma^{\mathcal O}(o, u_1.f^{(1)},\dots,u_r.f^{(r)});
\gamma^{E\mathcal M^\bullet}(u_\bullet;v_\bullet^{(1)},\dots,v_\bullet^{(r)})\big)
\end{align*}
for all $m\ge0$ and $r,n_1,\dots,n_r\ge0$, $u_\bullet\in (E\mathcal M)_m^{r}, v_\bullet^{(i)}\in (E\mathcal M)_m^{n_i},o\in\mathcal O(r)_m,p_i\in\mathcal O(n_i)_m$, and similarly for $\mathscr C=\cat{Cat}$. Here $\gamma^{\mathcal O}$ denotes the operad structure map of $\mathcal O$ and $\gamma^{E\mathcal M^\bullet}$ is given by juxtaposition and precomposition analogously to the definition of $\mathcal I$. We omit the straightforward but lengthy verification that $\mathcal O\rtimes E\mathcal M$ is indeed an operad in $G\text{-}\mathscr C$, and that we have a map of $G$-operads $i\colon\mathcal O\to\mathcal O\rtimes E\mathcal M$ induced in degree $n$ by the inclusion $\{(1,1,\dots,1)\}\hookrightarrow E\mathcal M^n$. Similarly, we obtain a monoid homomorphism $k\colon E\mathcal M\to (\mathcal O\rtimes E\mathcal M)(1)$ and restricting this to $\core(\mathcal M)$ then makes $\mathcal O\rtimes E\mathcal M$ into a twisted $G$-operad.
\end{constr}

\begin{rk}
The above construction works more generally for every monoid $M$ in $\mathscr C$, also see \cite[Definition~2.1]{framed-wahl} or \cite[3.5]{framed-markl}, where this appears for groups acting on topological spaces.
\end{rk}

\begin{rk}
Plugging in the definitions, we see that the left $\core(\mathcal M)$-action on $(\mathcal O\rtimes E\mathcal M)(n)$ in the sense of Construction~\ref{constr:external-actions} is the diagonal of the natural $\core(\mathcal M)$-action on $E\mathcal M^n$ and the restriction of the given $E\mathcal M$-action on $\mathcal O(n)$. On the other hand, the right $\core(\mathcal M)^n$-action is given by acting in the evident way on $E\mathcal M^n$ and trivially on $\mathcal O(n)$.
\end{rk}

\begin{prop}\label{prop:ext-vs-internal-1-cat}
Let $\mathcal O$ be an operad in $E\mathcal M\text{-}G\text{-}\mathscr C$ and let $A$ be an $\mathcal O\rtimes E\mathcal M$-algebra in $G\text{-}\mathscr C$. Then $A$ defines an $\mathcal O$-algebra in $E\mathcal M\text{-}G\text{-}\mathscr C$ by restricting the $\mathcal O\rtimes E\mathcal M$-algebra structure along the above inclusion $i\colon\mathcal O\hookrightarrow\mathcal O\rtimes E\mathcal M$ and equipping $A$ with the $E\mathcal M$-action obtained by restricting along the above homomorphism $k\colon E\mathcal M\to (\mathcal O\rtimes E\mathcal M)(1)$. This construction extends to an equivalence of ordinary categories $\Alg_{\mathcal O\rtimes E\mathcal M}(G\text{-}\mathscr C)\to \Alg_{\mathcal O}(E\mathcal M\text{-}G\text{-}\mathscr C)$ by sending a map $f\colon C\to D$ of $(\mathcal O\rtimes E\mathcal M)$-algebras to the same map viewed as a morphism of $\mathcal O$-algebras.
\end{prop}

The corresponding statement for groups acting on topological spaces also appears without proof as \cite[Proposition~2.3]{framed-wahl}.

\begin{proof}
Let us first show that this is well-defined, for which it only remains to show that the above $\mathcal O$-action on $A$ is $E\mathcal M$-equivariant. We will show this for $\mathscr C=\cat{SSet}$, the proof for categories being analogous. To this end we note that by definition for each $u\in (E\mathcal M)_m$, $o\in\mathcal O(n)_m$ and $a_1,\dots,a_n\in A_m$ the action maps satisfy
\begin{align*}
&\alpha^{\mathcal O}_n(u.o; u.a_1,\dots,u.a_n)\\
&\quad{}=\alpha^{\mathcal O\rtimes E\mathcal M}_n\big((u.o,\bm1); \alpha^{\mathcal O\rtimes E\mathcal M}_1((1,u);a_1),\dots,\alpha^{\mathcal O\rtimes E\mathcal M}_1((1,u);a_n)\big)
\end{align*}
for $\bm1=(1,\dots,1)\in E\mathcal M^n_m$. Using that $A$ was an $(\mathcal O\rtimes E\mathcal M)$-algebra together with the definition of the operad structure on $\mathcal O\rtimes E\mathcal M$ this then equals
\begin{align*}
&\alpha^{\mathcal O\rtimes E\mathcal M}_n\big(\gamma\big((u.o,\bm1);(1,u),\dots,(1,u)\big); a_1,\dots,a_n\big)\\
&\quad{}= \alpha^{\mathcal O\rtimes E\mathcal M}_n\big(\big(u.o,(u,\dots,u)\big); a_1,\dots,a_n\big)\\
&\quad{}= \alpha^{\mathcal O\rtimes E\mathcal M}_n\big(\gamma\big((1,u);\big(o,(1,\dots,1\big)\big); a_1,\dots,a_n\big)
\end{align*}
and again using that $A$ is an $(\mathcal O\rtimes E\mathcal M)$-algebra and plugging in the definitions this then agrees with $u.\alpha^{\mathcal O}(o;a_1,\dots,a_n)$ as desired.

To finish the proof, we now observe that the above functor $\Psi$ fits into a commutative diagram of $1$-categories
\begin{equation*}
\begin{tikzcd}
\Alg_{\mathcal O\rtimes E\mathcal M}(G\text{-}\mathscr C)\arrow[d, "\forget"']\arrow[r,"\Psi"] & \Alg_{\mathcal O}(E\mathcal M\text{-}G\text{-}\mathscr C)\arrow[d, "\forget"]\\
G\text{-}\mathscr C\arrow[r,equal] &G\text{-}\mathscr C
\end{tikzcd}
\end{equation*}
in which the vertical functors are monadic. By the (1-categorical) Barr-Beck Monadicity Theorem it is then again enough that the canonical mate of the identity transformation is an isomorphism $\cat{P}_{\mathcal O}X\to\cat{P}_{\mathcal O\rtimes E\mathcal M} X$ for every $X\in G\text{-}\mathscr C$. However, plugging in the definitions this is indeed just the canonical isomorphism $\coprod_{n\ge 0} \mathcal O(n)\times_{\Sigma_n}(E\mathcal M\times X)^{\times n}\to \coprod_{n\ge 0} (\mathcal O(n)\times E\mathcal M^n)\times_{\Sigma_n}X^{\times n}$.
\end{proof}

\begin{rk}
One immediately checks from the definitions that with respect to the $G$-global model structures on both sides, the above equivalence preserves and reflects fibrations and weak equivalences; thus, it also preserves and reflects cofibrations.
\end{rk}

\begin{prop}\label{prop:associated-external}
Let $\mathcal O$ be a $G$-global operad.
\begin{enumerate}
\item If $\mathcal O$ is $\Sigma$-free, then so is $\mathcal O\rtimes E\mathcal M$.
\item If $\mathcal O$ is a $G$-global $E_\infty$-operad, then $\mathcal O\rtimes E\mathcal M$ is a twisted $G$-global $E_\infty$-operad.\label{item:ae-e-infty}
\end{enumerate}
\begin{proof}
For the first statement we restrict to the case $\mathscr C=\cat{SSet}$; the proof for categories is analogous or alternatively follows by passing to the nerve. For this, let $(o;u_0,\dots,u_n)\in\mathcal O(n)_m\times E\mathcal M_m^n$ and $(\sigma;v_0,\dots,v_n)\in\Sigma_n\wr\core(\mathcal M)$ with $(o;u_0,\dots,u_n).(\sigma;v_0,\dots,v_n) = (o;u_0,\dots,u_n)$, i.e.
\begin{equation*}
(o.\sigma; u_{\sigma(0)}v_0,\dots,u_{\sigma(n)}v_n)=(o;u_0,\dots,u_n).
\end{equation*}
But then $\sigma=1$ as $\Sigma_n$ acts freely on $\mathcal O(n)$. Thus, $u_iv_i=u_i$ for all $i=0,\dots,n$ and hence also $v_i=1$ as desired since $\core(\mathcal M)$ acts freely on $\mathcal M$ from the right. This completes the proof of the first statement.

For the second statement, it remains to show that $(\mathcal O\rtimes E\mathcal M)(n)\to *$ is a  $\Sigma_n\wr\core(\mathcal M)$-global (weak) equivalence if $\mathcal O$ is a $G$-global $E_\infty$-operad. For this we let $H\subset\mathcal M$ be universal and $\phi\colon H\to G\times(\Sigma_n\wr\core(\mathcal M))$ be any homomorphism. If we write $p\colon G\times(\Sigma_n\wr\core(\mathcal M))\to G\times\Sigma_n$ and $q\colon G\times(\Sigma_n\wr\core(\mathcal M))\to\Sigma_n\wr\core(\mathcal M)$ for the projections, then
\begin{equation*}
(\mathcal O\rtimes E\mathcal M)(n)^\phi = \mathcal O(n)^{p\circ\phi}\times E\big((\mathcal M^n)^{q\circ\phi}\big).
\end{equation*}
But $\mathcal O(n)^{p\circ\phi}$ is (weakly) contractible as $\mathcal O$ was assumed to be a $G$-global $E_\infty$-operad, while $(\mathcal M^n)^{q\circ\phi}\supset\Inj(\bm n\times\omega,\omega)^{q\circ\phi}$ is non-empty as seen in Example~\ref{ex:ext-g-global-operad}, so that also the second factor is contractible as desired.
\end{proof}
\end{prop}

\begin{proof}[Proof of Theorem~\ref{thm:external-vs-internal}]
The zig-zag is given by
\begin{equation*}
\Alg_{\mathcal O}(E\mathcal M\text{-}G\text{-}\mathscr C)\simeq\Alg_{\mathcal O\rtimes E\mathcal M}(G\text{-}\mathscr C)\leftrightarrows \Alg_{(\mathcal O\rtimes E\mathcal M)\times\mathcal P}(G\text{-}\mathscr C)\rightleftarrows \Alg_{\mathcal P}(G\text{-}\mathscr C)
\end{equation*}
where the equivalence on the left is the one from Proposition~\ref{prop:ext-vs-internal-1-cat} and the remaining two Quillen equivalences are induced by the projections as in Corollary~\ref{cor:comparison-ext-g-global-operads}.
\end{proof}

\subsection{\texorpdfstring{$\bm G$}{G}-global vs.~\texorpdfstring{$\bm G$}{G}-equivariant coherent commutativity}
Throughout, let $G$ be a finite group and fix an injective homomorphism $j\colon G\to\mathcal M$ with universal image. We write $\delta\colon G\to E\mathcal M\times G$ for the homomorphism sending $g$ to $(j(g),g)$.

If $\mathcal O$ is any $G$-global operad, then we can restrict the $E\mathcal M$-$G$-action along $\delta$ to yield a $G$-operad $\delta^*\mathcal O$, and likewise for algebras. In this subsection we will prove:

\begin{thm}\label{thm:internal-g-global-vs-g-equiv}
Let $\mathcal O$ be a $\Sigma$-free $G$-global operad. Then
\begin{equation}\label{eq:delta-star-int-g-global}
\delta^*\colon\Alg_{\mathcal O}(E\mathcal M\text{-}G\text{-}\mathscr C)\to\Alg_{\delta^*\mathcal O}(G\text{-}\mathscr C)
\end{equation}
induces a left and right Bousfield localization at the $G$-equivariant equivalences (for $\mathscr C=\cat{Cat}$) or $G$-equivariant weak equivalences (for $\mathscr C=\cat{SSet}$).
\end{thm}

\begin{rk}
If $\mathcal O$ is a $G$-global $E_\infty$-operad, then $\delta^*\mathcal O$ is evidently a genuine $G$-$E_\infty$-operad. Thus the above theorem in particular allows us to express the homotopy theory of (categorical or simplicial) genuine $G$-$E_\infty$-algebras as a Bousfield localization of the homotopy theory of $G$-global $E_\infty$-algebras.
\end{rk}

The proof of the theorem will occupy the rest of this subsection. For this, we will first prove an analogous result for \emph{twisted $G$-operads}, from which we will then in a second step deduce the above comparison for ordinary operads.

\begin{constr}
If $\mathcal O$ is a twisted $G$-operad, then we get induced left and right $\core(\mathcal M)$-actions on $\mathcal O(n)$ for every $n\ge 0$ that commute with each other and the $G$-action. While the structure maps of $\mathcal O(n)$ are usually not equivariant with respect to either of these $\core(\mathcal M)$-actions, one immediately checks that they are equivariant with respect to the \emph{conjugate} action, making $\mathcal O$ into a $(\core(\mathcal M)\times G)$-operad. We write $\mathcal O_G$ for the operad in $G\text{-}\mathscr C$ obtained via restricting this action along the homomorphism $\delta=(j,\id)\colon G\to\core(\mathcal M)\times G$.
\end{constr}

\begin{ex}
Let $\mathcal I$ be the categorical/simplicial injection operad from Example~\ref{ex:ext-g-global-operad}. Then the above $G$-action on $\mathcal I_G$ is given by $g.u= j(g)\circ u\circ(\bm n\times j(g^{-1}))$. Put differently, $\mathcal I_G(n)= E\Inj(\bm n\times\mathcal U_G,\mathcal U_G)$ where $\mathcal U_G\mathrel{:=}j^*\omega$ is a (specific) complete $G$-set universe. Operads of this form were considered for example by Guillou and May in \cite[Definition~7.4]{guillou-may}, and they are examples of genuine $G$-$E_\infty$-operads. In fact, if $\mathcal O$ is any twisted $G$-global $E_\infty$-operad, then $\mathcal O_G$ is a genuine $G$-$E_\infty$-operad as
\begin{equation*}
\mathcal O_G(n)^\phi=\mathcal O(n)^{(j^{-1},\phi j^{-1})\colon j(H)\to G\times\Sigma_n}\simeq*
\end{equation*}
for every $H\subset G$, $\phi\colon H\to\Sigma_n$.
\end{ex}

\begin{constr}\label{constr:external-twisted-diagonal}
Let $j\colon G\to\mathcal M$ be as above and let $\mathcal O$ be a twisted $G$-operad. Then any $\mathcal O$-algebra $A$ in $G\text{-}\mathscr C$ carries an `external' $G$-action by restricting the action of $\mathcal O(1)$ along
\begin{equation*}
G\xrightarrow j\core(\mathcal M)\xrightarrow k \mathcal O(1)^G,
\end{equation*}
and this commutes with the `internal' $G$-action coming from the fact that we started with an algebra in $G$-categories. Equipping $A$ with the diagonal of these two actions, we therefore obtain a functor $\delta\colon\Alg_{\mathcal O}(G\text{-}\mathscr C)\to G\text{-}\mathscr C$. One easily checks that the original action maps $\mathcal O(n)\times A^{\times n}\to A$ are $G$-equivariant when viewed as maps $\mathcal O_G(n)\times (\delta A)^n\to \delta A$, so that this construction lifts to a functor
\begin{equation}\label{eq:action-twist}
\Delta\colon\Alg_{\mathcal O}(G\text{-}\mathscr C)\to\Alg_{\mathcal O_G}(G\text{-}\mathscr C).
\end{equation}
\end{constr}

\begin{lemma}\label{lemma:Delta-right-Bousfield}
Let $\mathcal O$ be a twisted $G$-operad. Then $(\ref{eq:action-twist})$ induces a right Bousfield localization
\begin{equation*}
\Delta^\infty\colon\Alg_{\mathcal O}(G\text{-}\mathscr C)_{\textup{$G$-global}}^\infty\to\Alg_{\mathcal O_G}(G\text{-}\mathscr C)_{\textup{$G$-equivariant}}^\infty.
\end{equation*}
\begin{proof}
It is clear that $(\ref{eq:action-twist})$ precisely inverts the $G$-equivariant (weak) equivalences; it remains to show that $\Delta^\infty$ is a quasi-localization and that it admits a left adjoint.

For the first statement it suffices to observe that $\Delta$ is an isomorphism of $1$-categories: an inverse is given by equipping an $\mathcal O_G$-algebra $C$ with the $G$-action in which $g\in G$ acts via the composition
\begin{equation*}
A\xrightarrow{g.\blank} A\xrightarrow{\alpha(kj(g^{-1}),\blank)} A
\end{equation*}
for $k$ as in Construction~\ref{constr:twisted-product},
where $\blank.\blank$ denotes the internal $G$-action and $\alpha$ denotes the operadic action; note that this is indeed a $G$-action as $g.(kj(h))=kj(ghg^{-1})$ for all $g,h\in G$ \emph{with respect to the $G$-action on $\mathcal O_G$}.

Finally, to construct the left adjoint it suffices to observe that $\Delta$ is right Quillen with respect to the $G$-global model structure on the source and the $G$-equivariant one on the target.
\end{proof}
\end{lemma}

\begin{proof}[Proof of Theorem~\ref{thm:internal-g-global-vs-g-equiv}]
It is clear that $\delta^*$ inverts precisely the $G$-equivariant equivalences. It remains to show that $(\delta^*)^\infty$ is a right Bousfield localization and that it in addition admits a right adjoint.

\begin{claim*}
The map $i\colon\mathcal O\to\mathcal O\rtimes E\mathcal M$ defines a $G$-equivariant (weak) equivalence $\delta^*\mathcal O\to (\mathcal O\rtimes E\mathcal M)_G$ of operads.
\begin{proof}
Plugging in the definitions, we immediately see that $(\mathcal O\rtimes E\mathcal M)_G(n)=(\delta^*\mathcal O)(n)\times E(\mathcal M^{\textup{conj}})^n$ as $(G\times\Sigma_n)$-objects in $\mathscr C$, where $\mathcal M^{\textup{conj}}$ denotes $\mathcal M$ with $G$-action given by $g.u=j(g)uj(g^{-1})$ and $\Sigma_n$ acts in the evident way on each factor. In particular, $i$ indeed defines a map $\delta^*\mathcal O\to (\mathcal O\rtimes E\mathcal M)_G$ and it only remains to show that $E(\mathcal M^{\textup{conj}})^n$ has contractible $\phi$-fixed points for every $\phi\colon G\to\Sigma_n$. However, these are just the fixed points for
\begin{equation*}
j(G)\xrightarrow{j^{-1}} G\xrightarrow{(\phi;\id,\dots,\id)} \Sigma_n\wr G,
\end{equation*}
and as $j(G)$ is universal, the claim then follows from the proof of Proposition~\ref{prop:associated-external}-$(\ref{item:ae-e-infty})$.
\end{proof}
\end{claim*}

We now have a strictly commutative diagram
\begin{equation}\label{diag:Delta-vs-delta}
\begin{tikzcd}
\Alg_{\mathcal O}(E\mathcal M\text{-}G\text{-}\mathscr C)\arrow[d,"\cong"']\arrow[r, "\delta^*"] & \Alg_{\delta^*\mathcal O}(G\text{-}\mathscr C)\\
\Alg_{\mathcal O\rtimes E\mathcal M}(G\text{-}\mathscr C)\arrow[r, "\Delta"'] & \Alg_{(\mathcal O\rtimes E\mathcal M)_G}(G\text{-}\mathscr C)\arrow[u, "i^*"']
\end{tikzcd}
\end{equation}
where the unlabelled isomorphism is the one from Proposition~\ref{prop:ext-vs-internal-1-cat}. Upon passing to associated quasi-categories, the lower horizontal arrow induces a right Bousfield localization by Lemma~\ref{lemma:Delta-right-Bousfield}, while the right hand arrow induces an equivalence by Theorem~\ref{thm:change-of-operad-sset}/\ref{thm:change-of-operad-categories} together with Proposition~\ref{prop:associated-external}-$(\ref{item:ae-e-infty})$. Thus, also $(\delta^*)^\infty$ is a right Bousfield localization.

Finally, to see that $(\delta^*)^\infty$ admits a right adjoint, we observe that we have a Quillen adjunction
\begin{equation*}
\delta^*\colon E\mathcal M\text{-}G\text{-}\mathscr C\rightleftarrows G\text{-}\mathscr C :\!\delta_*;
\end{equation*}
while this is not quite an instance of Lemma~\ref{lemma:alpha-star-sset}/\ref{lemma:alpha-star}, the left adjoint is homotopical and it moreover preserves cofibrations as the ones on the right hand side are simply the injective cofibrations. Since both adjoints preserve products, this lifts to an adjunction
\begin{equation*}
\Alg_{\mathcal O}(E\mathcal M\text{-}G\text{-}\mathscr C)\rightleftarrows\Alg_{\delta^*\mathcal O}(G\text{-}\mathscr C),
\end{equation*}
which is again a Quillen adjunction as fibrations and weak equivalences on both sides are created in the underlying categories. The claim follows immediately.
\end{proof}

\begin{rk}
Conversely, we can now conclude from $(\ref{diag:Delta-vs-delta})$ that also $\Delta^\infty$ admits a right adjoint, exhibiting it as \emph{left} Bousfield localization as well.
\end{rk}

\section{Parsummability}\label{sec:parsummability}
Let $G$ be any discrete group. In this section we will compare the above categories of $G$-global $E_\infty$-algebras to the models of `$G$-globally coherently commutative monoids' studied in \cite[Chapter~2]{g-global} and in particular to so-called \emph{$G$-parsummable categories} and \emph{$G$-parsummable simplicial sets}.

\subsection{Tameness and the box product}
In order to talk about parsummability we first have to introduce a technical condition on $E\mathcal M$-actions that is called \emph{tameness}:

\begin{defi}
Let $X$ be an $\mathcal M$-set. An element $x\in X$ is \emph{supported} on a finite set $A\subset\omega$ if $u.x=x$ for all $u\in\mathcal M$ fixing $A$ pointwise; we write $X_{[A]}\subset X$ for the subset of all elements supported on $A$. Moreover, we say that $x$ is \emph{finitely supported} if it is supported on some finite set $A\subset\omega$, and we call $X$ \emph{tame} if all its elements are finitely supported.

If $X$ is an $E\mathcal M$-simplicial set, then we say that an $n$-simplex $x\in X_n$ is supported on the finite set $A\subset\omega$ if it is supported on $A$ as an element of $X_n$ equipped with the diagonal $\mathcal M$-action. Analogously, we define the support of a finitely supported $n$-simplex, and the notion of tameness of $E\mathcal M$-simplicial sets. We write $\cat{$\bm{E\mathcal M}$-SSet}^\tau\subset\cat{$\bm{E\mathcal M}$-SSet}$ for the full subcategory spanned by the tame $E\mathcal M$-simplicial sets.

Finally, a small $E\mathcal M$-category is \emph{tame} if its set of objects is so. We write $\cat{$\bm{E\mathcal M}$-Cat}^\tau\subset\cat{$\bm{E\mathcal M}$-Cat}$ for the full subcategory spanned by the tame $E\mathcal M$-categories.
\end{defi}

For a detailed treatment of the combinatorics of tame actions we refer the reader to \cite[Section~2]{I-vs-M-1-cat} and \cite[Section~1.3]{g-global}.

\begin{rk}\label{rk:tameness-strong}
While the above definition of support for $E\mathcal M$-simplicial sets might look too weak at first, one can actually show that if $X$ is an $E\mathcal M$-simplicial set and $x\in X_n$ is supported on $A$ in the above sense, then $(u_0,\dots,u_m).f^*x=f^*x$ for all $f\colon[m]\to[n]$ in $\Delta$ and $(u_0,\dots,u_m)\in (E\mathcal M)_m$ such that each $u_i$ fixes $A$ pointwise, see \cite[Theorem~1.3.17]{g-global}. A similar statement for $E\mathcal M$-categories can be found as \cite[Proposition~2.13-(ii)]{schwede-k-theory}.
\end{rk}

\begin{rk}\label{rk:coreflective}
The inclusions $\cat{$\bm{E\mathcal M}$-Cat}^\tau\hookrightarrow\cat{$\bm{E\mathcal M}$-Cat}$,  $\cat{$\bm{E\mathcal M}$-SSet}^\tau\hookrightarrow\cat{$\bm{E\mathcal M}$-SSet}$ both admit right adjoints $(\blank)^\tau$ given by passing to the full subcategory of finitely supported objects or the subcomplex of finitely supported simplices, respectively, see \cite[Example~2.15]{schwede-k-theory} and \cite[proof of Corollary~1.3.23]{g-global}. It follows formally, that these subcategories are closed under small colimits and that they possess small limits, which can be computed by forming limits in the ambient category and then applying $(\blank)^\tau$. Moreover, one immediately checks that both $\cat{$\bm{E\mathcal M}$-Cat}^\tau$ and $\cat{$\bm{E\mathcal M}$-SSet}^\tau$ are closed under all \emph{finite} limits.
\end{rk}

\begin{thm}\label{thm:positive-model-structure-sset}
The category $\cat{$\bm{E\mathcal M}$-$\bm G$-SSet}^\tau$ of $G$-objects in $\cat{$\bm{E\mathcal M}$-SSet}^\tau$ admits a unique cofibrantly generated model structure with weak equivalences the $G$-global weak equivalences and generating cofibrations given by
\begin{align*}
I=\big\{\big(E\Inj(A,\omega)\times_\phi G\big)\times (\del\Delta^n\hookrightarrow\Delta^n) : {}&\text{$H$ finite group,}\\ &\text{$A\not=\varnothing$ finite faithful $H$-set,}\\
&\text{$\phi\colon H\to G$, $n\ge0$}\big\}.
\end{align*}
We call this the \emph{positive $G$-global model structure}. It is combinatorial, simplicial, proper, and filtered colimits in it are homotopical. Moreover, the adjunction
\begin{equation*}
\incl\colon\cat{$\bm{E\mathcal M}$-$\bm G$-SSet}^\tau\rightleftarrows\cat{$\bm{E\mathcal M}$-$\bm G$-SSet}_\textup{injective $G$-global}:\!(\blank)^\tau
\end{equation*}
is a Quillen equivalence.
\end{thm}
Strictly speaking, we of course have to restrict to \emph{sets} of finite sets $A$ and finite groups $H$ (covering all isomorphism classes) in the definition of the generating cofibrations to get an honest set, but we will ignore this technicality below.

We noreover remark that there is also an \emph{absolute $G$-global model structure} with the same weak equivalences as above but where we additionally allow $A=\varnothing$ in the definition of the generating cofibrations. While this absolute model structure is more natural in most contexts, the positive one will be necessary for our applications below.
\begin{proof}[Proof of Theorem~\ref{thm:positive-model-structure-sset}]
See \cite[Theorem~1.4.60]{g-global}.
\end{proof}

Our next goal now is to construct an analogous model structure on $\cat{$\bm{E\mathcal M}$-$\bm G$-Cat}^\tau$ and to show that the inclusion $\cat{$\bm{E\mathcal M}$-$\bm G$-Cat}^\tau\hookrightarrow\cat{$\bm{E\mathcal M}$-$\bm G$-Cat}$ is similarly part of a Quillen equivalence.

\begin{thm}\label{thm:positive-model-structure}
The category $\cat{$\bm{E\mathcal M}$-$\bm G$-Cat}^\tau$ admits a unique cofibrantly generated model structure with weak equivalences the $G$-global equivalences and generating cofibrations given by
\begin{align*}
I=\big\{\big(E\Inj(A,\omega)\times_\phi G\big)\times i : {}&\text{$H$ finite group, $A\not=\varnothing$ finite faithful $H$-set,}\\
&\text{$\phi\colon H\to G$, $i\in I_{\cat{Cat}}$}\big\}.
\end{align*}
We call this the \emph{positive $G$-global model structure}. It is left proper, $\cat{Cat}$-enriched (hence in particular simplicial), and combinatorial.
\end{thm}

Before we can prove the theorem, we need the following technical lemma:

\begin{lemma}
The adjunction $\h\colon\cat{SSet}\rightleftarrows\cat{Cat} :\!\nerve$ lifts to an adjunction
\begin{equation}\label{eq:tame-cat-vs-sset}
\h\colon\cat{$\bm{E\mathcal M}$-$\bm G$-SSet}^\tau\rightleftarrows\cat{$\bm{E\mathcal M}$-$\bm G$-Cat}^\tau :\nerve.
\end{equation}
\begin{proof}
Since both $\h$ and $\nerve$ preserve finite products, they naturally lift to an adjunction $\cat{$\bm{E\mathcal M}$-$\bm G$-SSet}\rightleftarrows\cat{$\bm{E\mathcal M}$-$\bm G$-Cat}$, and it only remains to show that they preserve tameness. For the nerve this appears in \cite[Example~2.7]{sym-mon-global}, while the claim for $\h$ follows immediately from the definitions.
\end{proof}
\end{lemma}

\begin{proof}[Proof of Theorem~\ref{thm:positive-model-structure}]
The category $\cat{$\bm{E\mathcal M}$-$\bm G$-Cat}^\tau$ is locally presentable as the adjunction $(\ref{eq:tame-cat-vs-sset})$ exhibits it as an accessible Bousfield localization of the category $\cat{$\bm{E\mathcal M}$-$\bm G$-SSet}^\tau$ which in turn is locally presentable by \cite[Theorem~1.4.60]{g-global}.

We will now verify the assumptions of \cite[Proposition~A.2.6.13]{htt}: first, we have to show that the $G$-global equivalences on $\cat{$\bm{E\mathcal M}$-$\bm G$-Cat}^\tau$ are perfect in the sense of \cite[Definition~A.2.6.10]{htt}. However, the $G$-global equivalences \emph{on the whole category $\cat{$\bm{E\mathcal M}$-$\bm G$-Cat}$} are part of a combinatorial model structure and closed under filtered colimits (Corollary~\ref{cor:g-global-model-cat}), hence perfect by \cite[Remark~A.2.6.14]{htt}. Thus, the claim follows from \cite[Corollary~A.2.6.12]{htt} applied to the inclusion $\cat{$\bm{E\mathcal M}$-$\bm G$-Cat}^\tau\hookrightarrow\cat{$\bm{E\mathcal M}$-$\bm G$-Cat}$.

Next, let $j$ be any pushout of a map in $I$ (along an arbitrary map); we have to show that pushouts along $j$ preserves $G$-global equivalences. But $j$ is in particular an injective cofibration, so this is simply an instance of Proposition~\ref{prop:inj-po}.

Finally, we have to show that every map $f$ with the right lifting property against $I$ is a $G$-global equivalence. But indeed, looking at corepresented functors shows that in this case $f^\phi_{[A]}$ is an acyclic fibration in $\cat{Cat}$ for each universal subgroup $H$, each finite non-empty $H$-subset $A\subset\omega$, and each homomorphism $\phi\colon H\to G$. Passing to the filtered colimit over $A$ (and using that $f$ is a map of tame $E\mathcal M$-$G$-categories), then shows that $f^\phi$ is an equivalence of categories (as a filtered colimit of equivalences) as desired.

Thus, \cite[Proposition~A.2.6.13]{htt} shows that the model structure exists and that it is combinatorial and left proper. It only remains to show that it is $\cat{Cat}$-enriched as a model category, i.e.~to verify the Pushout Product Axiom. For the statement about cofibrations, we may restrict to generating cofibrations, where this follows easily from the fact that $\cat{Cat}$ itself is $\cat{Cat}$-enriched. The part about acyclic cofibrations then follows as in the proof of Theorem~\ref{thm:injective-equivariant-model}.
\end{proof}

\begin{thm}\label{thm:inclusion-tame-cat}
The adjunction
\begin{equation*}
\incl\colon\cat{$\bm{E\mathcal M}$-$\bm G$-Cat}^\tau\rightleftarrows\cat{$\bm{E\mathcal M}$-$\bm G$-Cat}_{\textup{injective $G$-global}} :\!(\blank)^\tau
\end{equation*}
is a Quillen equivalence.
\begin{proof}
It is clear that the left adjoint preserves cofibrations and creates weak equivalences. It therefore only remains to show that the counit $C^\tau\hookrightarrow C$ is a $G$-global equivalence for each injectively fibrant $E\mathcal M$-$G$-category $C$.

This will be analogous to the argument we gave in the simplicial situation as part of \cite[Corollary~1.3.28]{g-global}. Namely, we will show that $C^\phi_{[A]}\hookrightarrow C^\phi$ is an equivalence of categories for each universal $H\subset\mathcal M$, non-empty finite faithful $H$-set $A\subset\omega$, and homomorphism $\phi\colon H\to G$; the claim will then again follow by passing to the filtered colimit over all such $A$.

To prove the claim we consider for each such $C$ the map
\begin{equation}\label{eq:corep-A-phi}
\Fun^{E\mathcal M\times G}(E\Inj(A,\omega)\times_\phi G,C)\to\Fun^{E\mathcal M\times G}(E\mathcal M\times_\phi G,C)
\end{equation}
induced by the restriction $E\mathcal M\to E\Inj(A,\omega)$. It is clear that $(\ref{eq:corep-A-phi})$ is conjugate to $C_{[A]}^\phi\hookrightarrow C^\phi$. As $C$ is injectively fibrant and since the injective $G$-global model structure is $\cat{Cat}$-enriched, it is therefore enough to show that $E\mathcal M\times_\phi G\to E\Inj(A,\omega)\times_\phi G$ is a $G$-global equivalence.

Clearly, the $H$-actions via precomposition on $E\mathcal M$ and $E\Inj(A,\omega)$ are free. Unravelling the definition of $\blank\times_\phi G$ and appealing to Lemma~\ref{lemma:free-quotients} it is then enough to show that $E\mathcal M\to E\Inj(A,\omega)$ is an $H$-global equivalence, for which it in turn suffices that $\big(E\Inj(B,\omega)\big)^\psi$ is equivalent to the terminal category for every universal $K\subset\mathcal M$, every $\psi\colon K\to H$, and each countable $H$-set $B$. As before this just amounts to saying that $\Inj(B,\omega)^\psi$ is non-empty, i.e.~that there exists an $K$-equivariant injection $\psi^*B\rightarrowtail\omega$. This is in turn immediate from universality of $K$ (also see \cite[Example~1.2.35]{g-global}).
\end{proof}
\end{thm}

Now we are ready to define \emph{parsummable categories} as first introduced in \cite{schwede-k-theory} as well as their simplicial counterpart, the \emph{parsummable simplicial sets} \cite{sym-mon-global}.

\begin{defi}
Let $C,D\in\cat{$\bm{E\mathcal M}$-Cat}^\tau$. Their \emph{box product} is the full subcategory $C\boxtimes D\subset C\times D$ of the Cartesian product spanned by all those $(c,d)\in C\times D$ such that $c$ and $d$ are \emph{disjointly supported}, i.e.~$\supp(c)\cap\supp(d)=\varnothing$.
\end{defi}

One can show that $C\boxtimes D$ is again tame and that $\boxtimes$ defines a subfunctor of the Cartesian product. The usual coherence isomorphisms of the Cartesian symmetric monoidal structure then restrict to make $\boxtimes$ the monoidal product of a preferred symmetric monoidal structure on $\cat{$\bm{E\mathcal M}$-Cat}^\tau$, see~\cite[Proposition~2.35]{schwede-k-theory}. Taking diagonal $G$-actions this then more generally provides a symmetric monoidal structure on $\cat{$\bm{E\mathcal M}$-$\bm G$-Cat}^\tau$ for every (discrete) group $G$.

\begin{defi}
We write
\begin{equation*}
\cat{$\bm G$-ParSumCat}\mathrel{:=}\CMon(\cat{$\bm{E\mathcal M}$-$\bm G$-Cat}^\tau,\boxtimes)
\end{equation*}
for the category with objects the commutative monoids (with respect to the box product) in $\cat{$\bm{E\mathcal M}$-$\bm G$-Cat}^\tau$ and morphisms the monoid homomorphisms. We call its objects \emph{$G$-parsummable categories}.
\end{defi}

Similarly, there is a box product of tame $E\mathcal M$-simplicial sets which is however slightly more intricate to define:

\begin{defi}
Let $X,Y$ be tame $E\mathcal M$-simplicial sets. We call $x\in X_n,y\in Y_n$ \emph{disjointly supported} if they are disjointly supported as elements of the $\mathcal M$-sets $i_k^*X_n$ and $i_k^*Y_n$, respectively, for every $0\le k\le n$, where $i_k\colon\mathcal M\to\mathcal M^{n+1}$ denotes the inclusion of the $(k+1)$-th factor; note that these are indeed tame $\mathcal M$-sets by Remark~\ref{rk:tameness-strong}. The \emph{box product} $X\boxtimes Y$ is the subsimplicial set of $X\times Y$ given by all pairs of disjointly supported simplices.
\end{defi}

One can show that this is indeed a subsimplicial set, that it is preserved by the diagonal $E\mathcal M$-action \cite[Proposition~2.17]{sym-mon-global}, and that this yields a symmetric monoidal structure analagous to the above \cite[Proposition~2.18]{sym-mon-global}.

\begin{lemma}\label{lemma:nerve-symmetric-monoidal}
The usual strong symmetric monoidal structure for the Cartesian products restricts to a strong symmetric monoidal structure on $\nerve\colon\cat{$\bm{E\mathcal M}$-$\bm G$-Cat}^\tau\to\cat{$\bm{E\mathcal M}$-$\bm G$-SSet}^\tau$ with respect to the box products. In particular, the nerve lifts to a functor $\cat{$\bm G$-ParSumCat}\to\cat{$\bm G$-ParSumSSet}$.
\begin{proof}
See \cite[Proposition~2.20]{sym-mon-global}.
\end{proof}
\end{lemma}

\begin{ex}\label{ex:box-corep}
Let $A,B$ be finite sets. The restriction maps $E\Inj(A\amalg B,\omega)\to E\Inj(A,\omega)$ and $E(A\amalg B,\omega)\to E\Inj(B,\omega)$ induce an isomorphism $E(A\amalg B,\omega)\cong E\Inj(A,\omega)\boxtimes E\Inj(B,\omega)$ in $\cat{$\bm{E\mathcal M}$-Cat}^\tau$ (and hence also in $\cat{$\bm{E\mathcal M}$-SSet}^\tau$): namely, both $E\Inj(A\amalg B,\omega)$ and the box product are indiscrete, so it suffices to show this on the level of objects, where this follows from \cite[Example~2.15]{I-vs-M-1-cat}.
\end{ex}

\begin{rk}
The iterated box products on $\cat{$\bm{E\mathcal M}$-Cat}^\tau$ on $\cat{$\bm{E\mathcal M}$-SSet}^\tau$ each admit a natural \emph{unbiased} description that we will tacitly use below. Namely, if $C_1,\dots,C_n$ are any tame $E\mathcal M$-categories, then for any bracketing of the left hand side $C_1\boxtimes\cdots\boxtimes C_n\to C_1\times\cdots\times C_n$ induces an isomorphism onto the full subcategory spanned by the tuples $(c_1,\dots,c_n)\in C_1\times\cdots\times C_n$ with pairwise disjoint support, and analogously in the simplicial case.
\end{rk}

\subsection{The box product as an operadic product}
In order to compare $G$-parsum\-mable categories and $G$-parsummable simplicial sets to algebras over (twisted) $G$-global $E_\infty$-operads, we will first devise an alternative, `operadic' description of iterated box products.

Throughout the remainder of this section, let $\mathscr C\in\{\cat{Cat},\cat{SSet}\}$ again; we will write $\text{ParSum}\text{-}{\mathscr C}$ for $\CMon(E\mathcal M\text{-}\mathscr C^\tau,\boxtimes)$ and refer to its objects as \emph{parsummable $\mathscr C$-objects}.

\begin{constr}
Let $n\ge 0$. Then $\mathcal M^n$ acts from the right on $\Inj(\bm n\times\omega,\omega)$ via
\begin{equation*}
f.(u^{(1)},\dots,u^{(n)})=f \circ (u^{(1)}\amalg\cdots\amalg u^{(n)}),
\end{equation*}
which as usual induces a right action of $E\mathcal M^n$ on $E\Inj(\bm n\times\omega,\omega)$.

Now let $X_1,\dots,X_n$ be $E\mathcal M$-objects, which then yields a left $E\mathcal M^n$-action on $X_1\times\cdots\times X_n$. We write $E\Inj(\bm n\times\omega,\omega)\times_{E\mathcal M^n}(X_1\times\cdots\times X_n)$ for the $E\mathcal M$-object obtained from $E\Inj(\bm n\times\omega,\omega)\times (X_1\times\cdots X_n)$ by coequalizing the right $E\mathcal M^n$-action on the first factor with the left $E\mathcal M^n$-action on the second one. Acting with $E\Inj(\bm n\times\omega,\omega)\subset E\mathcal M^n$ then induces a natural map
\begin{equation}\label{eq:Phi-comparison}
\Phi\colon E\Inj(\bm n\times\omega,\omega)\times_{E\mathcal M^n}(X_1\times\cdots\times X_n)\to X_1\times\cdots X_n.
\end{equation}
\end{constr}

\begin{thm}\label{thm:box-product-operadic}
For all $X_1,\dots,X_n\in E\mathcal M\text{-}G\text{-}\mathscr C^\tau$ the natural map $(\ref{eq:Phi-comparison})$ restricts to an isomorphism onto $X_1\boxtimes\cdots\boxtimes X_n$.
\begin{proof}
For $\mathscr C=\cat{SSet}$ this appears as \cite[Theorem~2.1.19]{g-global}; as the nerve is fully faithful and strong symmetric monoidal (Lemma~\ref{lemma:nerve-symmetric-monoidal}), the corresponding statement for categories then follows formally by applying this to $\nerve(X_1),\dots,\nerve(X_n)$ and passing to homotopy categories.
\end{proof}
\end{thm}

Let us draw some immediate consequences from this description:

\begin{cor}\label{cor:box-prod-cocont}
The box product on $E\mathcal M\text{-}G\text{-}\mathscr C^\tau$ is cocontinuous in each variable. Thus, the symmetric monoidal structure on $E\mathcal M\text{-}G\text{-}\mathscr C^\tau$ described above is closed.
\begin{proof}
For $\mathscr C=\cat{SSet}$ this is \cite[Corollary~2.1.21]{g-global}, and the same proof works for $\mathscr C=\cat{Cat}$: namely, the first claim follows from the previous theorem as $E\Inj(\bm2\times\omega,\omega)\times_{E\mathcal M^2}(\blank\times\blank)$ is clearly cocontinuous in each variable. The second claim is then immediate as $E\mathcal M\text{-}G\text{-}\mathscr C^\tau$ is locally presentable (Theorem~\ref{thm:positive-model-structure-sset}/\ref{thm:positive-model-structure}).
\end{proof}
\end{cor}

\begin{prop}
\begin{enumerate}
\item Let $X_1,\dots,X_n\in E\mathcal M\text{-}G\text{-}\mathscr C$. Then the natural map $\Phi\colon E\Inj(\bm n\times\omega,\omega)\times_{E\mathcal M^n}(X_1\times\cdots\times X_n)\to X_1\times\cdots\times X_n$ is a $G$-global (weak) equivalence.
\item Let $n\ge 0$ and let $X\in E\mathcal M\text{-}G\text{-}\mathscr C$. We equip $E\Inj(\bm n\times\omega,\omega)\times_{E\mathcal M^n}X^{\times n}$ with the $\Sigma_n$-action via $\sigma.[u_0,\dots,u_m;x_1,\dots,x_n]=[u_0\circ(\sigma^{-1}\times\id),\dots,u_m\circ(\sigma^{-1}\times\id); x_{\sigma^{-1}(1)},\dots,x_{\sigma^{-1}(n)}]$. Then the natural map $E\Inj(\bm n\times\omega,\omega)\times_{E\mathcal M^n} X^{\times n}\to X^{\times n}$ is a $(G\times\Sigma_n)$-global (weak) equivalence.
\end{enumerate}
\begin{proof}
For the first statement, it suffices to show that $E\Inj(\bm n\times\omega,\omega)\hookrightarrow E\mathcal M^n$ is a right-$E\mathcal M^n$-left-$H$-equivariant equivalence of categories for every universal $H\subset\mathcal M$, for which suffices to give a left-$H$-right-$\mathcal M^n$-equivariant map $\mathcal M^n\to\Inj(\bm n\times\omega,\omega)$. For this we simply pick an $H$-equivariant injection $u\colon\bm n\times\omega\rightarrowtail\omega$ (which exists by universality of the target) and consider the map $(v_1,\dots,v_n)\mapsto u\circ(v_1\amalg\cdots\amalg v_n)$.

For the second statement it suffices similarly to construct for each given $H$-action on $\bm n$ a left-$H$-right-$\mathcal M^n$-equivariant map $\mathcal M^n\to\Inj(\bm n\times\omega,\omega)$, which can be done in the same way.
\end{proof}
\end{prop}

We immediately conclude the following comparisons between the box product and the Cartesian product, which for $G=1$ and $\mathscr C=\cat{Cat}$ also appear in \cite[proof of Theorem~2.33]{schwede-k-theory} and \cite[proof of Theorem~4.13]{schwede-k-theory}, respectively, while the corresponding simplicial statements can be found in \cite[2.1.2.1]{g-global}:

\begin{cor}\label{cor:box-prod-vs-product}
\begin{enumerate}
\item Let $X_1,\dots,X_n\in E\mathcal M\text{-}G\text{-}\mathscr C^\tau$. Then the inclusion $X_1\boxtimes\cdots\boxtimes X_n\hookrightarrow X_1\times\cdots\times X_n$ is a $G$-global (weak) equivalence.
\item Let $X\in E\mathcal M\text{-}G\text{-}\mathscr C^\tau$. Then the inclusion $X^{\boxtimes n}\hookrightarrow X^{\times n}$ is a $(G\times\Sigma_n)$-global (weak) equivalence for every $n\ge0$.\qedhere\qed
\end{enumerate}
\end{cor}

\begin{cor}\label{cor:box-product-homotopical}
\begin{enumerate}
\item The box product is homotopical in each variable.
\item If $f\colon X\to Y$ is a $G$-global (weak) equivalence in $E\mathcal M\text{-}G\text{-}\mathscr C^\tau$, then $f^{\boxtimes n}$ is a $(G\times\Sigma_n)$-global (weak) equivalence.
\end{enumerate}
\begin{proof}
The first statement follows immediately from the first part of the previous corollary as the Cartesian product is homotopical. The second statement follows similarly from Corollary~\ref{cor:twisted-products-sset}/\ref{cor:twisted-products}.
\end{proof}
\end{cor}

\subsection{\texorpdfstring{$\bm{\mathcal I}$}{I}-algebras vs.~parsummability} As mentioned without proof in \cite[Remark~4.20]{schwede-k-theory}, parsummable categories can be identified with tame $\mathcal I$-algebras, also see~\cite[Theorem~A.13]{I-vs-M-1-cat} for the corresponding $\cat{Set}$-level statement. Using Theorem~\ref{thm:box-product-operadic} we can now easily give a full proof of this as well as of its simplicial counterpart, for which we begin with the following observation:

\begin{lemma}\label{lemma:free-tame-vs-free-general}
The free-forgetful adjunction $\cat{P}\colon E\mathcal M\text{-}G\text{-}\mathscr C\rightleftarrows\Alg_{\mathcal I}(G\text{-}\mathscr C) :\!\forget$ restricts to an adjunction
\begin{equation}\label{eq:free-forgetful-tame-I}
\cat{P}\colon E\mathcal M\text{-}G\text{-}\mathscr C^\tau\rightleftarrows\Alg_{\mathcal I}(G\text{-}\mathscr C)^\tau :\!\forget
\end{equation}
\begin{proof}
We have to show that $\cat{P}X$ is tame for every $X\in E\mathcal M\text-G\text{-}\mathscr C^\tau$, for which we note that as $E\mathcal M$-$G$-objects $\cat{P}X=\coprod_{n\ge 0}(\mathcal I(n)\times_{E\mathcal M^n}X^{\times n})/\Sigma_n$, which is isomorphic to $\coprod_{n\ge0} X^{\boxtimes n}/\Sigma_n$ by Theorem~\ref{thm:box-product-operadic}; the claim follows immediately as $E\mathcal M\text{-}G\text{-}\mathscr C^\tau$ is closed under colimits and the box product.
\end{proof}
\end{lemma}

\begin{constr}
Let $A$ be a tame $\mathcal I$-algebra. We define $+\colon A\boxtimes A\to A$ via
\begin{equation*}
A\boxtimes A\xrightarrow{\Phi^{-1}} E\Inj(\bm2\times\omega,\omega)\times_{E\mathcal M^2} A^{\times 2}\to A
\end{equation*}
where the right hand map is induced by the action map. Moreover, we define $0\in A$ as the image of the unique $0$-ary operation.
\end{constr}

\begin{prop}\label{prop:parsum-vs-tame-I}
For any tame $\mathcal I$-algebra $A$, the above defines the structure of a parsummable $\mathscr C$-object. Moreover, this construction extends to an equivalence of ordinary categories $\Alg_{\mathcal I}(G\text{-}\mathscr C)^\tau\to G\textup{-ParSum-}\mathscr C$ by sending an $\mathcal I$-algebra homomorphism $A\to B$ to the monoid homomorphism with the same underlying map.
\begin{proof}
For clarity, let us write $\cat{P}_\boxtimes$ for the left adjoint of the forgetful functor $G\text{-ParSum-}\mathscr C\to E\mathcal M\text{-}G\text{-}\mathscr C^\tau$ and $\cat{P}_{\mathcal I}$ for the left adjoint in $(\ref{eq:free-forgetful-tame-I})$.

The forgetful functor $\Alg_{\mathcal I}(G\text{-}\mathscr C)\to G\text{-}\mathscr C$ preserves reflective coequalizers, hence so does $\Alg_{\mathcal I}(G\text{-}\mathscr C)^\tau\to E\mathcal M\text{-}G\text{-}\mathscr C$ as colimits on the right hand side are created in $G\text{-}\mathscr C$. As it is moreover clearly conservative, we see that $(\ref{eq:free-forgetful-tame-I})$ is monadic, i.e.~the canonical functor $\Alg_{\mathcal I}(G\text{-}\mathscr C)^\tau\to\Alg_{\cat{P}_{\mathcal I}}$ is an equivalence.

On the other hand, we have an equivalence of categories $\Alg_{\cat{P}_{\boxtimes}}\to G\text{-ParSum-}\mathscr C$ that sends an algebra $(A,\cat{P}_{\boxtimes}A\to A)$ to $A$ equipped with the sum induced by the restriction of $\cat{P}_{\boxtimes}A\to A$ to the summand $A^{\boxtimes 2}/\Sigma_2$ and $0$-object the image of the zeroth summand, and that sends a map $f$ of $\cat{P}_{\boxtimes}$-algebras to the monoid homomorphism with the same underlying map.

Finally, Theorem~\ref{thm:box-product-operadic} shows that the natural transformation $\coprod_{n\ge 0}\Phi\colon\cat{P}_{\mathcal I}\to\cat{P}_{\boxtimes}$ is an isomorphism. Moreover, this is a map of monads either by direct inspection, or alternatively using that both sides are naturally submonads of the monad $\cat{P}_{\times}\colon X\mapsto\coprod_{n\ge0} X^{\times n}/\Sigma_n$ for strictly commutative monoids (viewing $\mathcal I$ as a suboperad of $*\rtimes E\mathcal M$). Thus, pulling back along the inverse isomorphism yields an equivalence $\Alg_{\cat{P}_{\mathcal I}}\to \Alg_{\cat{P}_{\boxtimes}}$. The resulting equivalence
\begin{equation*}
\Alg_{\mathcal I}(G\text{-}\mathscr C)^\tau\to\Alg_{\cat{P}_{\mathcal I}}\to \Alg_{\cat{P}_{\boxtimes}}\to G\text{-ParSum-}\mathscr C
\end{equation*}
then clearly admits the above description.
\end{proof}
\end{prop}

In order to achieve the desired comparison between $G\text{-ParSum-}\mathscr C$ and all of $\Alg_{\mathcal I}(G\text{-}\mathscr C)$ we will introduce suitable model structures next.

\begin{thm}\label{thm:parsumcat-model-structure}
There is a unique model structure on $G\textup{-ParSum-}\mathscr C$ in which a map is weak equivalence or fibration if and only if it is so in the positive $G$-global model structure on $E\mathcal M\text{-}G\text{-}\mathscr C^\tau$. We call this the \emph{positive $G$-global model structure}. It is proper, $\mathscr C$-enriched (hence in particular simplicial), and combinatorial. Moreover, the adjunction
\begin{equation}\label{eq:free-forgetful-parsumcat}
\cat{P}\colon E\mathcal M\text{-}G\text{-}\mathscr C^\tau\rightleftarrows G\textup{-ParSum-}\mathscr C :\!\forget
\end{equation}
is a Quillen adjunction.
\begin{proof}
For $\mathscr C=\cat{SSet}$ this appears as \cite[Theorem~2.1.36]{g-global}; the proof for $\mathscr C=\cat{Cat}$ is analogous, so we will be somewhat terse here.

We will verify the assumptions of \cite[Theorem~3.2]{white-cmon}, whose terminology we follow. The smallness assumptions are automatically satisfied as $\cat{$\bm{E\mathcal M}$-$\bm G$-Cat}^\tau$ is locally presentable.

Next, let us show that $\cat{$\bm{E\mathcal M}$-$\bm G$-Cat}^\tau$ is a symmetric monoidal model category with respect to the box product. Indeed, the box product provides a closed symmetric monoidal structure by Corollary~\ref{cor:box-prod-cocont}; moreover, the Unit Axiom is satisfied as the box product is fully homotopical (Corollary~\ref{cor:box-product-homotopical}). In order to verify the Pushout Product Axiom for cofibrations we may restrict to the generating cofibrations, in which case we note that we can identify the pushout product of standard generating cofibrations $E\Inj(A,\omega)\times_\phi G\times i$ and $E\Inj(B,\omega)\times_\psi G\times i'$ (where $A$ is a finite non-empty faithful $H$-set and $B$ is a finite non-empty faithful $K$-set for some finite groups $H,K$) up to isomorphism with
\begin{equation}\label{eq:generating-ppo}
\big(E\Inj(A\amalg B,\omega)\times G^2\times (i\ppo i')\big)/(H\times K)
\end{equation}
where $H\times K$ acts on $A\amalg B$ in the obvious way (which is then clearly faithful) and on $G^2$ via $\phi$ and $\psi$, and where the pushout product of $i$ and $i'$ is formed with respect to the Cartesian product on $\cat{Cat}$. As $\cat{Cat}$ is Cartesian, $i\ppo i'$ is a cofibration, and as $G$ acts freely from the left on $G^2$, we conclude from the following claim that $(\ref{eq:generating-ppo})$ is a cofibration as desired:

\begin{claim*}
Let $L$ be a finite group, let $C$ be a finite faithful $L$-set, and let $f\colon X\to Y$ be a map in $\cat{$\bm{(G\times L)}$-Cat}$ such that $f$ is injective on objects and $G$ acts freely on $\Ob(Y)\setminus f(\Ob X)$. Then $E\Inj(C,\omega)\times_L f$ is a cofibration in $\cat{$\bm{E\mathcal M}$-$\bm G$-Cat}^\tau$.
\begin{proof}
The functor $E\Inj(C,\omega)\times_L\blank$ is a left adjoint and it sends generating cofibrations of the $\mathcal G_{L,G}$-model structure to generating cofibrations in the $G$-global model structure by direct inspection. Thus the claim follows from Lemma~\ref{lemma:charact-cof}.
\end{proof}
\end{claim*}

The Pushout Product Axiom for acyclic cofibrations then follows again as in the proof of Theorem~\ref{thm:injective-equivariant-model}. Thus, $\cat{$\bm{E\mathcal M}$-$\bm G$-Cat}^\tau$ is symmetric monoidal.

For the Monoid Axiom we observe that for any acyclic cofibration $j$ and any tame $E\mathcal M$-$G$-category $C$ the map $C\boxtimes j$ is a $G$-global equivalence and an injective cofibration. Thus, also any pushout of $C\boxtimes j$ is a $G$-global equivalence by Proposition~\ref{prop:inj-po}, and as the $G$-global equivalences are stable under filtered colimits, so is any transfinite composition of such.

Next, we verify the Strong Commutative Monoid Axiom, for which we again may restrict to generating (acyclic) cofibrations by \cite[Lemma~A.1]{white-cmon}. For the part about cofibrations, we therefore have to show that the map $i^{\ppo n}/\Sigma_n$ is a cofibration for each of the standard generating cofibrations $(E\Inj(A,\omega)\times_\phi G)\times i'$ ($i'\in I_{\cat{Cat}}$, $H$ finite, $A\not=\varnothing$ a finite faithful $H$-set, $\phi\colon H\to G$), where $i^{\ppo n}$ denotes the iterated pushout product and $\Sigma_n$ acts by permuting the factors. However, using Example~\ref{ex:box-corep}, this map agrees up to conjugation by isomorphisms with
\begin{equation}\label{eq:generating-cmon-ax}
(E\Inj(\bm n\times A,\omega)\times G^n \times (i')^{\ppo n})/(\Sigma_n\wr H)
\end{equation}
where the wreath product acts on $\Inj(\bm n\times A,\omega)$ via the action on $\bm n\times A$ given by
\begin{equation*}
(\sigma;h_1,\dots,h_n).(k,a)=(\sigma(k),h_k.a)
\end{equation*}
and similarly on $G^n\times (i')^{\ppo n}$. One easily checks that the $(\Sigma_n\wr H)$-action on $\bm n\times A$ is faithful (which uses $A\not=\varnothing$). Thus, we again conclude from the above claim that $(\ref{eq:generating-cmon-ax})$ is a cofibration.

The sources of the standard generating cofibrations of $\cat{$\bm{E\mathcal M}$-$\bm G$-Cat}^\tau$ are cofibrant, so we can pick a set of generating \emph{acyclic} cofibrations with cofibrant sources by  \cite[Corollary~2.7]{barwick-tractable}. Thus, to verify the Strong Commutative Monoid Axiom for acyclic cofibrations, it is enough by \cite[Corollaries~10 and~23]{sym-powers} to show that $j^{\boxtimes n}/\Sigma_n$ is a $G$-global equivalence for any $G$-global equivalence $j\colon C\to D$ between cofibrant objects. As $j^{\boxtimes n}$ is a $(G\times\Sigma_n)$-global equivalence by Corollary~\ref{cor:box-product-homotopical}, it suffices by Lemma~\ref{lemma:free-quotients} to show that $\Sigma_n$ acts freely on $A^{\boxtimes n}$ for every cofibrant $A$. But indeed, one easily checks inductively that such an $A$ has no $\mathcal M$-fixed objects (hence no $\mathcal M$-fixed points at all), so the claim follows by applying the argument from \cite[proof of Corollary~2.1.17]{g-global} to the nerve of $A$.

Altogether, this shows that we may apply \cite[Theorem~3.2]{white-cmon} to conclude that the transferred model structure along $(\ref{eq:free-forgetful-parsumcat})$ exists, which easily implies all of the above claims except for the left properness, which is instead an instance of \cite[Theorem~4.17]{white-cmon}.
\end{proof}
\end{thm}

\begin{cor}
There is a unique model structure on $\Alg_{\mathcal I}(G\text{-}\mathscr C)^\tau$ in which a map is a weak equivalence or fibration if and only if it so in the positive $G$-global model structure on $E\mathcal M\text{-}G\text{-}\mathscr C^\tau$. We call this the \emph{positive $G$-global model structure} again. It is combinatorial, right proper, and simplicial. Moreover, the adjunction $(\ref{eq:free-forgetful-tame-I})$ is a Quillen adjunction.
\begin{proof}
As the forgetful functors are compatible with the equivalence (of ordinary categories) from Proposition~\ref{prop:parsum-vs-tame-I}, Theorem~\ref{thm:parsumcat-model-structure} implies that the model structure transferred along $(\ref{eq:free-forgetful-tame-I})$ exists. The claims follow as before.
\end{proof}
\end{cor}

\begin{prop}
\begin{enumerate}
\item Geometric realization in $E\mathcal M\text{-}G\text{-}\mathscr C^\tau$ and $E\mathcal M\text{-}G\text{-}\mathscr C$ is homotopical.
\item The forgetful functors $\Alg_{\mathcal I}(G\text{-}\mathscr C)^\tau\to E\mathcal M\text{-}G\text{-}\mathscr C^\tau$, $\Alg_{\mathcal I}(G\text{-}\mathscr C)\to E\mathcal M\text{-}G\text{-}\mathscr C$ preserve geometric realizations.
\end{enumerate}
\begin{proof}
The first statement for $E\mathcal M\text{-}G\text{-}\mathscr C$ is an instance of Proposition~\ref{prop:geometric-realization-homotopical} (or its classical simplicial analogue), and this immediately yields the corresponding statement for $E\mathcal M\text{-}G\text{-}\mathscr C^\tau$ as the inclusion preserves tensors and colimits.

Likewise, the second statement for $E\mathcal M\text{-}G\text{-}\mathscr C$ is an instance of Proposition~\ref{prop:forget-geometric-realization} (using that geometric realization is created in $\mathscr C$ and $\Alg_{\mathcal I}(\mathscr C)$, respectively), and for the tame statement it suffices to show that $\Alg_{\mathcal I}(G\text{-}\mathscr C)^\tau$ is closed under geometric realizations. However, as geometric realizations in $\Alg_{\mathcal I}(G\text{-}\mathscr C)$ can be computed in $E\mathcal M\text{-}\mathscr C$, this follows as for the first statement.
\end{proof}
\end{prop}

The same argument as in the proof of Proposition~\ref{prop:free-monadic-sset} now shows:

\begin{prop}\label{prop:g-global-monadic}
The adjunctions
\begin{align*}
\cat{LP}\colon E\mathcal M\text{-}G\text{-}\mathscr C^\infty&\rightleftarrows\Alg_{\mathcal I}(G\text{-}\mathscr C)^\infty :\!{\forget}^\infty\\
\cat{LP}\colon (E\mathcal M\text{-}G\text{-}\mathscr C^\tau)^\infty&\rightleftarrows(\Alg_{\mathcal I}(G\text{-}\mathscr C)^\tau)^\infty :\!{\forget}^\infty
\end{align*}
are monadic.\qed
\end{prop}

\begin{thm}\label{thm:taming-I-algebras}
The inclusion $G\textup{-ParSum-}\mathscr C\simeq\Alg_{\mathcal I}(G\text{-}\mathscr C)^\tau\hookrightarrow\Alg_{\mathcal I}(G\text{-}\mathscr C)$ descends to an equivalence of associated quasi-categories. In particular, if $\mathcal O$ is any twisted $G$-global $E_\infty$-operad, then $G\textup{-ParSum-}\mathscr C^\infty\simeq \Alg_\mathcal O(G\text{-}\mathscr C)^\infty$.
\begin{proof}
By Corollary~\ref{cor:comparison-ext-g-global-operads} it suffices to prove the first statement. For this we consider the diagram
\begin{equation*}
\begin{tikzcd}
(\Alg_{\mathcal I}(G\text{-}\mathscr C)^\tau)^\infty\arrow[r, "\incl^\infty"]\arrow[d, "\forget^\infty"']&\Alg_{\mathcal I}(G\text{-}\mathscr C)^\infty\arrow[d,"\forget^\infty"]\\
(E\mathcal M\text{-}G\text{-}\mathscr C^\tau)^\infty\arrow[r, "\incl^\infty"]\arrow[d,"\incl^\infty"']\twocell[ur] & E\mathcal M\text{-}G\text{-}\mathscr C^\infty\arrow[d,equal]\\
E\mathcal M\text{-}G\text{-}\mathscr C\arrow[r,equal]\twocell[ur]& E\mathcal M\text{-}G\text{-}\mathscr C^\infty
\end{tikzcd}
\end{equation*}
where the two squares commute up to the natural equivalences induced by the respective identity transformations. The vertical composites are monadic by the previous proposition together with Theorem~\ref{thm:positive-model-structure-sset}/\ref{thm:inclusion-tame-cat}. As in the proof of Theorem~\ref{thm:change-of-operad-sset} it then suffices to show that the canonical mate of the total rectangle is an equivalence. As the functors in the bottom square are equivalences, it is enough to prove this for the top square, where this is immediate from the construction of the adjunctions (cf.~Lemma~\ref{lemma:free-tame-vs-free-general}).
\end{proof}
\end{thm}
For $\mathscr C=\cat{SSet}$ this implies the following result which in particular subsumes Theorem~\ref{introthm:genuine-e-infty-vs-gamma} from the introduction:

\begin{thm}\label{thm:genuine-e-infty-vs-gamma}
The following quasi-categories are equivalent:
\begin{itemize}
\item the quasi-category $\Alg_{\mathcal O}(\cat{$\bm{E\mathcal M}$-$\bm G$-SSet})^\infty$ for any $G$-global $E_\infty$-operad $\mathcal O$,
\item the quasi-category $(\cat{$\bm\Gamma$-$\bm G$-$\bm{E\mathcal M}$-SSet}^\textup{special}_*)^\infty$ of special $G$-global $\Gamma$-spaces in the sense of \cite[Definition~2.2.50]{g-global}, and
\item the quasi-category $\cat{$\bm G$-UCom}^\infty$ of $G$-ultra-commutative monoids \cite[Definition~2.1.25]{g-global}.
\end{itemize}
\begin{proof}
Let $\mathcal O$ be a $G$-global $E_\infty$-operad; the previous theorem together with Theorem~\ref{thm:external-vs-internal} implies that $\Alg_{\mathcal O}(\cat{$\bm{E\mathcal M}$-$\bm G$-SSet})^\infty$ is equivalent to the quasi-category of $G$-parsummable simplicial sets. These are in turn equivalent to special $G$-global $\Gamma$-spaces by \cite[Theorem~2.3.1 and Corollary~2.2.53]{g-global} and to $G$-ultra-commutative monoids by Corollary~2.1.38 of {op.~cit.}
\end{proof}
\end{thm}

Assume now that $G$ is finite. As a consequence of the above, we get the following result which for $\mathscr C=\cat{Cat}$ generalizes Theorem~\ref{introthm:genuine-sym-mon-vs-parsum-cat} from the introduction:

\begin{thm}\label{thm:genuine-sym-vs-parsum}
The composition
\begin{equation*}
(G\textup{-ParSum-}\mathscr C)\hookrightarrow\Alg_{\mathcal I}(G\text-\mathscr C)\xrightarrow{\Delta}\Alg_{\mathcal I_G}(G\text-\mathscr C)
\end{equation*}
induces a quasi-localization at the $G$-equivariant (weak) equivalences. In particular, if $\mathcal O$ is any genuine $G$-$E_\infty$-operad, then we have an equivalence
\begin{equation*}
(G\textup{-ParSum-}\mathscr C)_{\textup{$G$-equivariant}}^\infty\simeq
\Alg_{\mathcal O}(G\text-\mathscr C)_{\textup{$G$-equivariant}}^\infty.
\end{equation*}
\begin{proof}
The first statement is immediate from the Theorem~\ref{thm:taming-I-algebras} together with Lemma~\ref{lemma:Delta-right-Bousfield}. With this established, the second statement then follows from Theorem~\ref{thm:change-of-operad-categories}.
\end{proof}
\end{thm}

\begin{rk}\label{rk:what-we-can-say-about-the-equivalence}
For $G=1$ and $\mathcal O=E\Sigma_*$ the above in particular yields an equivalence between permutative categories and parsummable categories, both viewed with respect to the underlying equivalences of categories. On the other hand, we previously proved in \cite{perm-parsum-categorical} that a specific functor $\Phi$ constructed by Schwede \cite[Construction~11.1]{schwede-k-theory} from permutative to parsummable categories descends to an equivalence of associated quasi-categories. This induced functor is in fact inverse to the above equivalence for abstract reasons: namely, both $\Phi$ and the above equivalence preserve underlying categories in the sense that they come with an equivalence ${\forget}\circ\Phi\simeq\forget$ and similarly for the above composition. It then follows formally from \cite[Corollary~2.5-(iii)]{ggn} that there are essentially unique equivalences between the two composites and the respective identities that are compatible with the chosen equivalences for the underlying categories.

On the other hand, while the results of Subsection~\ref{subsec:explicit-comparison} allow us to make the comparison between $G$-parsummable categories and genuine permutative $G$-categories explicit, the resulting functor is quite complicated. I do not know of any simple comparison in this case, or in fact even an interesting direct functor from genuine permutative $G$-categories to $G$-parsummable ones.

However, we can at least describe the $H$-fixed points ($H\subset G$) of the genuine permuative $G$-category associated to a $G$-parsummable category concisely: namely, by the same trick as in the case $H=G=1$ it will suffice to give a functor $\Psi_H\colon\cat{$\bm G$-ParSumCat}\to\cat{SymMonCat}$ that on underlying categories agrees with the $H$-fixed points of the above equivalence. For this we will in fact give a functor $\Psi$ to $\cat{$\bm G$-SymMonCat}$ that on underlying $G$-categories agrees with our equivalence, i.e.~that lifts the functor $\delta^*\colon\cat{$\bm{E\mathcal M}$-$\bm G$-Cat}^\tau\to\cat{$\bm G$-Cat}$.

This is a straightforward adaption of the construction for $G=1$ presented in \cite[Section~5]{schwede-k-theory}: for any $n\ge0$ and any injection $\phi\colon\bm n\times\omega\to\omega$ we define $\phi_*\colon C^{\times n}\to C$ on objects via $\phi_*(x_1,\dots,x_n)=\sum_{i=1}^n\phi(i,\blank).x_i$ and analogously on morphisms; note that the sum is indeed well-defined as $\phi(i,\blank).x_i$ and $\phi(j,\blank).x_j$ are disjointly supported for $i\not=j$ by \cite[Proposition~2.13-(iii)]{schwede-k-theory}. For any other such injection $\psi$ we then get a natural transformation $[\psi,\phi]\colon\phi_*\Rightarrow\psi$ given on $(x_1,\dots,x_n)$ by $\sum_{i=1}^n(\psi(i,\blank),\phi(i,\blank)).x_i$. One immediately checks from the definitions that if $H\subset\mathcal M$ is any subgroup such that $\phi$ is $H$-equivariant with respect to the tautological $H$-action on $\omega$, then $\phi_*$ is $H$-equivariant, and similarly for $[\psi,\phi]$. Moreover, $\phi_*$, $\psi_*$ and $[\psi,\phi]$ are clearly $G$-equivariant (without any assumptions on $\phi$ and $\psi$) because the $E\mathcal M$-action on $C$ commutes with the $G$-action. In particular, we see that if $\phi$ and $\psi$ are $j(G)$-equivariant (for our chosen embedding $j\colon G\to\mathcal M$), then $\phi_*$ and $\psi_*$ define $G$-equivariant functors $(\delta^*C)^{\times n}\to\delta^*C$ and $[\psi,\phi]$ is a $G$-equivariant natural transformation between them.

Now \cite[Construction~5.5]{schwede-k-theory} associates to any choice of an injection $\mu\colon\bm2\times\omega\to\omega$ a symmetric monoidal category with underlying category $C$ and tensor product $C^{\times 2}\to C$ given by $\mu_*$. The tensor unit is the object $0$ and the structure isomorphisms are given by the canonical isomorphisms provided by the $E\mathcal M$-action: for example, the left unitality isomorphism $\mu_*(0,x)=\mu(2,\blank)_*x\to x$ is simply given by $[\mu(2,\blank),1]_x$; we refer the reader to {loc.cit.} for further details. As an upshot of the above, we see that if $\mu$ is $j(G)$-equivariant, then this equips $\delta^*C$ with the structure of a symmetric monoidal $G$-category $C$. Moreover, for any map $f\colon C\to D$ of $G$-parsummable categories the induced $G$-equivariant functor $\delta^*C\to\delta^*D$ is actually strict symmetric monoidal with respect to these symmetric monoidal structures. Thus, any choice of a $j(G)$-equivariant $\mu$ (which is possible as $j(G)$ is universal) provides a lift of $\delta^*$ to a functor $\cat{$\bm G$-ParSumCat}\to\cat{$\bm G$-SymMonCat}$ as desired.
\end{rk}

\begin{rk}\label{rk:comp-on-sat-promise}
We can also make the comparison functor explicit on the genuine permutative $G$-categories arising from na\"ive ones via the Guillou-May-Shimakawa construction (Example~\ref{ex:genuine-G-E-infty}), which covers most examples from practice. However, as this requires some additional terminology and techniques that we will only establish later, this description is given in the appendix as Proposition~\ref{prop:comparison-on-saturated}.
\end{rk}

Finally, Theorem~\ref{thm:genuine-sym-vs-parsum} together with \cite[Theorem~2.3.18]{g-global} implies the following result, which was previously proven (using quite different and more explicit means) by May, Merling, and Osorno \cite[10.2]{may-merling-osorno}:

\begin{cor}
Let $\mathcal O$ be any genuine $G$-$E_\infty$-operad in $\cat{$\bm G$-SSet}$. Then there is an equivalence of quasi-categories $\Alg_{\mathcal O}(\cat{$\bm G$-SSet})^\infty\simeq\cat{$\bm\Gamma$-$\bm G$-SSet}_*^\textup{special}$ between $\mathcal O$-algebras and Shimakawa's special $\Gamma$-$G$-spaces \cite{shimakawa}.\qed
\end{cor}

\section{Equivariant algebraic \texorpdfstring{$K$}{K}-theory and the \texorpdfstring{$K$}{K}-theory of group rings}\label{sec:k-theory-group-rings} As alluded to in the introduction, $G$-parsummable categories are arguably easier to construct than genuine symmetric monoidal $G$-categories. To demonstrate this, we will use Theorem~\ref{thm:genuine-sym-vs-parsum} to produce certain genuine symmetric monoidal $G$-categories with interesting equivariant algebraic $K$-theory of which I do not know any direct construction avoiding the use of $G$-parsummable categories.

To put this into context, we recall that Guillou and May defined the \emph{equivariant algebraic $K$-theory} $\cat{K}_G(C)$ of any genuine permutative $G$-category $C$ in \cite[Definition~4.12]{guillou-may}, which is a genuine $G$-spectrum in the sense of equivariant stable homotopy theory; we will not need any specifics of their construction. This was then used by Merling to define the equivariant algebraic $K$-theory of a $G$-ring $R$ \cite[Definition~5.23]{merling}:

\begin{constr}
    For a ring $R$, let $\mathscr P(R)$ denote the symmetric monoidal category of finitely generated projective $R$-modules and $R$-linear isomorphisms under direct sum. For technical convenience, we insist that the direct sums be obtained by fixing a choice of coproducts of abelian groups once and for all, and then equipping the chosen coproduct of underlying abelian groups with the usual $R$-module structure; as a consequence of this specific choice, the underlying abelian group of $M\oplus N$ only depends on the underlying groups of $M$ and $N$ \emph{up to equality}, and not just up to isomorphism.

    Now assume $G$ acts on $R$ through ring automorphisms. Then we define a $G$-action on $\mathscr P(R)$ by sending an $R$-module $M$ to the module $g.M$ with the same underlying abelian group, but with scalar multiplication
    \begin{align*}
        R\times (g.M)&\longrightarrow g.M\\
        (r,m)&\longmapsto (g^{-1}.r)m;
    \end{align*}
    moreover, an $R$-linear isomorphism $f\colon M\to N$ is sent to the same map of underlying abelian groups, considered as a morphism $g.M\to g.N$. As an upshot of our specific choices of direct sums, this defines a $G$-action through \emph{strictly} symmetric monoidal functors.

    The \emph{equivariant algebraic $K$-theory} $\cat K_G(R)$ of the $G$-ring $R$ is then defined by taking a (small) na\"ive $G$-permutative replacement $ \mathfrak P(R)\to\mathscr P(R)$ and then applying the equivariant algebraic $K$-theory functor $\cat{K}_G$ of Guillou and May to the genuine permutative $G$-category $\Fun(EG,\mathfrak P(R))$ (see Example~\ref{ex:genuine-G-E-infty}).
\end{constr}

For her construction, Merling provided a description of the categorical fixed points in terms of the $K$-theory of \emph{twisted group rings}:

\begin{constr}
    We write $R_GG$ for the ring with underlying abelian group $\bigoplus_{g\in G}R$ and with multiplication given by
    \begin{equation*}
        \bigg(\sum_{g\in G} r_gg\bigg)\bigg(\sum_{h\in G} s_hh\bigg)=\sum_{g,h\in G} \big(r_g (g.s_h)\big)(gh).
    \end{equation*}
    In particular, when $G$ acts trivially on $R$ this recovers the usual group algebra $RG$. Beware however that for non-trivial actions this is \emph{not} an $R$-algebra as the multiplication maps $g.\blank\colon R_GG\to R_GG$ are not $R$-linear, but instead \emph{$R$-semilinear}, i.e.~they are additive and satisfy $g(rx)=(g.r)(gx)$.

    More generally, given any $R_GG$-module $M$, the multiplication maps $g.\blank\colon M\to M$ define a $G$-action through $R$-semilinear maps, and conversely any such action on an $R$-module extends uniquely to an $R_GG$-module structure (in the obvious way). Moreover, an $R$-linear map $M\to N$ is $R_GG$-linear if and only if it commutes with the action maps, also see \cite[Observation 4.3]{merling}.
\end{constr}

\begin{thm}[Merling]\label{thm:merling-k-theory-group-rings}
If $H\subset G$ is a subgroup such that $|H|\in R^\times$, then we have a preferred equivalence
\begin{equation*}
F^H\cat{K}_G(R)\simeq\cat{K}(R_HH).
\end{equation*}
\begin{proof}
See \cite[Theorem~5.28]{merling}.
\end{proof}
\end{thm}

The assumptions of the theorem are in particular satisfied if $R$ is a $\mathbb Q$-algebra. However, in absence of the above invertibility condition the $H$-fixed points only recover the $K$-theory of finitely generated $R_HH$-modules which are projective \emph{over $R$} (as opposed to $R_HH$), which yields the wrong result already for $R=\mathbb Z$ and $H=G$ the cyclic group of order $2$ (acting trivially).

We will now define a genuine permutative $G$-category $\mathbb P_G(R)$ whose equivariant algebraic $K$-theory upon taking fixed points recovers the algebraic $K$-theory of twisted group rings over $R$ without any such invertibility assumption. To this end, we first begin with a concrete construction of a $G$-parsummable category $\mathcal P_G(R)$, which is a variant of the parsummable category used in Schwede's construction \cite[Construction~10.1]{schwede-k-theory} of the global algebraic $K$-theory of rings.

\begin{constr}
Let $R[\omega\times G]$ be the free left $R$-module with basis $\omega\times G$; we let $\mathcal M\times G$ act on $R[\omega\times G]$ via
\begin{equation*}
    (u,g).\bigg(\sum_{(v,h)\in\mathcal M\times G}r_{(v,h)}(v,h)\bigg)=
    \sum_{(v,h)\in\mathcal M\times G}(g.r_{(v,h)}) (uv, gh).
\end{equation*}
In particular, $\mathcal M$ acts $R$-linearly, while $G$ acts semilinearly; note that this is still enough to ensure that $g.\blank$ maps $R$-submodules to $R$-submodules.

We will now write $\mathcal P_G(R)$ for the following category: an object of $\mathcal P_G(R)$ is a pair $(M,r)$ of a finitely generated $R$-submodule $M\subset R[\omega\times G]$ together with an $R$-linear retraction $r\colon R[\omega\times G]\to M$ of the inclusion such that $r$ sends almost all of the standard basis vectors to zero. A map $\phi$ from $(M,r)$ to $(N,s)$ is an abstract $R$-linear isomorphism $M\cong N$; no compatibility of $\phi$ with the chosen retractions (or the embeddings) is required.

If $u\in\mathcal M$ is any injection, then acting with $u$ on $R[\omega\times G]$ sends any finitely generated $R$-submodule $M$ to a finitely generated submodule $u.M$. Moreover, if $r\colon R[\omega\times G]\to M$ is any $R$-linear map, then we get a map $r^u$ defined by $R$-linearly extending the map $\omega\times G\to M$ sending $(u(x),g)$ to $u.(r(x,g))$ and $(y,g)$ to $0$ for any $y\notin\im(u)$. Clearly $r^u\circ (u.\blank)=(u.\blank)\circ r$, and in particular $r^u$ is a retraction to the inclusion $u.M\hookrightarrow R[\omega\times G]$. As moreover $r^u$ again sends almost all basis vectors to zero by direct inspection, we may now define $u.(M,r)=(u.M,r^u)$ for every $(M,r)\in \mathcal P_G$. One easily checks that this defines an $\mathcal M$-action on $\Ob\mathcal P_G(R)$. Moreover, we have a natural isomorphism $u_\circ^{(M,r)}\colon(M,r)\to (u.M,r^u)=u.(M,r)$ given by acting with $u$; these maps clearly satisfy the relations
\begin{equation*}
v_\circ^{u.(M,r)}u_\circ^{(M,r)}=(vu)^{(M,r)}_\circ
\end{equation*}
for all $v,u\in\mathcal M$, so there is by \cite[Proposition~2.6]{schwede-k-theory} a unique way to define an $E\mathcal M$-action on $\mathcal P_G(R)$ such that the underlying $\mathcal M$-action on objects is as above and such that in addition the natural isomorphism $\id\Rightarrow (u.\blank)$ induced by the map $(u,1)$ in $E\mathcal M$ is given for each $u\in\mathcal M$ by the maps $u_\circ$.

In addition, we define a $G$-action on $\mathcal P_G$ as follows: for any $g\in G$, an object $(M,r)$ is sent to $(g.M,r^g)$ where $g.M$ is again given as the image of $M$ under the $G$-action on $R[\omega\times G]$, while $r^g=(g.\blank)\circ r\circ (g^{-1}.\blank)$. If now $\phi\colon (M,r)\to (N,s)$ is any map in $\mathcal P_G(R)$, then we define $g.\phi$ as the map $(g.\blank)\circ\phi\circ(g^{-1}.\blank)$; one easily checks that this defines a $G$-action on $\mathcal P_G(R)$. Moreover, the $G$-action on objects clearly commutes with the $\mathcal M$-action, and moreover $g.(u^{(M,r)}_\circ)=u^{g.(M,r)}$ as both sides are given as maps of $R$-modules simply by $(g.\blank)\circ (u.\blank) \circ (g^{-1}.\blank)=u.\blank$. Thus, \cite[Corollary~1.3]{perm-parsum-categorical} shows that $g.\blank$ is a map of $E\mathcal M$-categories, i.e.~we altogether get an $(E\mathcal M\times G)$-action.
\end{constr}

\begin{lemma}
The $E\mathcal M$-$G$-category $\mathcal P_G(R)$ is tame. The support of an element $(M,r)$ is given by $\supp(M)\cup\supp(r)$ where we define
\begin{align*}
\supp(M)&\mathrel{:=}\{ i\in\omega : \pr_{(i,g)}(M)\not=0 \text{ for some $g\in G$}\}\\
\supp(r)&\mathrel{:=}\{ i\in\omega : r(i,g)\not=0 \text{ for some $g\in G$}\};
\end{align*}
here $\pr_{(i,g)}\colon R[\omega\times G]\to R$ denotes the projection onto the basis vector $(i,g)$.
\begin{proof}
First observe that $\supp(M)$ is finite as $M$ is finitely generated, while $\supp(r)$ is finite by assumption on $r$.

If $u$ fixes $\supp(M)$ pointwise, then clearly $u.M=M$. Moreover, if $u$ fixes $\supp(r)$ pointwise, then $r(u(i),g)=r(i,g)$ for all $(i,g)\in\omega\times G$: namely, if $i\in\supp(r)$, then already $(u(i),g)=(i,g)$, while otherwise also $u(i)\notin\supp(r)$ by injectivity, so that $r(u(i),g)=0=r(i,g)$. Thus, if $u$ also fixes $\supp(M)$ pointwise, then $r^u(u(i),g)=u.r(i,g)=r(i,g)=r(u(i),g)$ for all $(i,g)\in\omega\times G$, while $r^u(i,g)=0=r(i,g)$ when $i\notin\im(u)$ as $\im(u)\supset\supp(r)$.

Conversely, let $A\subset\omega$ be a finite set such that $(M,r)$ is supported on $A$. We will show that $A\supset\supp(r)$; the argument that $A\supset\supp(M)$ is similar. To this end, assume for contradiction that $A\not\supset\supp(r)$ and pick $i\in\supp(r)\setminus A$ as well as $u\in\mathcal M$ with $u(a)=a$ for all $a\in A$, but $i\notin\im u$. Since $i\in\supp(r)$, there exists a $g\in G$ with $r(i,g)\not=0$. On the other hand $r^u(i,g)=0$ as $i\notin\im(u)$ so $r^u\not=r$ contradicting the assumption that $(M,r)$ be supported on $A$.
\end{proof}
\end{lemma}

\begin{constr}
We make the tame $E\mathcal M$-$G$-category $\mathcal P_G(R)$ into a $G$-parsum\-mable category as follows: the neutral element is given by the pair $(0,0)$ of the zero submodule and the zero map $R[\omega\times G]\to 0$; this has empty support by the previous lemma. If now $(M,r)$ and $(N,s)$ are disjointly supported, then we define $(M,r)+(N,s)$ as $(M+N,r+s)$ where $M+N$ is the internal sum as submodules and $(r+s)(x)=r(x)+s(x)$ as usual. Note that the sum $M+N$ is actually direct as $\supp(M)\cap\supp(N)=\varnothing$; moreover, $r+s$ is indeed a retraction: we will show that $s(x)=0$ for every $x\in M$; analogously one shows that $r(x)=0$ for every $x\in N$ which then easily yields the claim. To this end, we observe that if $x\in M$ and $\pr_{(i,g)}(x)\not=0$, then $i\in\supp(M)$ and hence $i\notin\supp(s)$ by the previous lemma, which shows $s(i,g)=0$. As we can express $x$ as an $R$-linear combination of $(i,g)$'s with $\pr_{(i,g)}(x)\not=0$, the claim follows. Finally, $M+N$ is clearly finitely generated while $r+s$ still sends almost all standard basis vectors to $0$, so that $(M+N,r+s)$ is a well-defined element of $\mathcal P_G(R)$. If $\phi\colon (M,r)\to (N,s)$ and $\phi'\colon (M',r')\to (N',s')$ are maps in $\mathcal P_G(R)$ such that $\supp(M,r)\cap\supp(M',r')=\varnothing=\supp(N,s)\cap\supp(N,s')$, then we define $(\phi+\phi')(m+m')=\phi(m)+\phi'(m')$; this is well-defined as the internal sum $M+M'$ is direct, and it is again bijective as also the internal sum $N+N'$ is direct by assumption. Altogether, we have defined a map $\mathcal P_G(R)\boxtimes\mathcal P_G(R)\to\mathcal P_G(R)$, and one trivially verifies that this is strictly unital, associative, and commutative.

Finally, the sum is clearly $\mathcal M$-equivariant on objects and it satisfies the relation $u^{(M,r)}_\circ+u^{(N,s)}_\circ=u^{(M+N,r+s)}_\circ$ whenever the sum is defined as both sides are simply given by restricting the action of $u$ on $R[\omega\times G]$. We conclude from another application of \cite[Corollary~1.3]{perm-parsum-categorical} that the sum is $E\mathcal M$-equivariant; as it is moreover $G$-equivariant by an similar computation, we have altogether defined a $G$-parsummable category $\mathcal P_G(R)$.
\end{constr}

Applying our equivalence $\cat{$\bm G$-ParSumCat}^\infty\simeq\Alg_{E\Sigma_*^{EG}}(\cat{$\bm G$-Cat})^\infty$ from Theorem~\ref{thm:genuine-sym-vs-parsum} now produces a genuine permutative $G$-category $\mathbb P_G(R)$ from this, which we can feed into the Guillou-May machinery:

\begin{thm}
For every subgroup $H\subset G$ there is a preferred equivalence
\begin{equation}\label{eq:k-theory-group-rings}
F^H\cat{K}_G(\mathbb P_G(R))\simeq\cat{K}(R_HH).
\end{equation}
\end{thm}
We emphasize again that unlike for Merling's construction (Theorem~\ref{thm:merling-k-theory-group-rings}) there is no invertibility condition on $|H|$ here; in particular, we can apply this to $R=\mathbb Z$ and any finite $G$ (acting trivially) to get a genuine $G$-spectrum whose $H$-fixed points for $H\subset G$ recover the $K$-theory of the integral group ring $\mathbb ZH$.
\begin{proof}
By \cite[Theorem~4.14]{guillou-may} we can identify the left hand side of $(\ref{eq:k-theory-group-rings})$ with the $K$-theory of the $H$-fixed points $\mathbb P_G(R)^H$ viewed as a permutative category in the usual way. On the other hand, Remark~\ref{rk:what-we-can-say-about-the-equivalence} gives an explicit description of an equivalent symmetric monoidal structure on $(\delta^*\mathcal P_G(R))^H$. We will now show that the latter is symmetric monoidally equivalent to $\mathscr P(R_HH)$, which will then complete the proof of the theorem.

To this end, observe that restricting the $(\mathcal M\times G)$-action on $R[\omega\times G]$ to $H$ along $\delta$ yields a semilinear $H$-action, so we may view $R[\omega\times G]$ as an $R_HH$-module $R_H[\omega\times G]$, also cf.~\cite[Proposition~4.5]{merling}. By definition, an object $(M,r)$ of $\delta^*\mathcal P_G(R)$ is now fixed by $h\in H$ if and only if acting with $h$ on $R_H[\omega\times G]$ sends $M$ to itself (not necessarily identically) and $r$ commutes with $h.\blank$; thus, $(M,r)$ is $H$-fixed if and only if $M$ is an $R_HH$-submodule of $R_H[\omega\times G]$ and $r$ is $R_HH$-linear. Similarly, a morphism $\phi\colon (M,r)\to (N,s)$ of $H$-fixed objects is $H$-fixed if and only if it is $R_HH$-linear. Thus, we get a well-defined functor from $\mathcal P_G(R)^H$ into the category of $R_HH$-modules and $R_HH$-linear isomorphisms by sending an $H$-fixed object $(M,r)$ to $M$ with the above $H$-action and an $H$-fixed morphism $\phi$ simply to $\phi$. This actually factors through $\mathscr P(R_HH)$: namely, $M$ is finitely generated as an $R$-module by assumption, and hence also as an $R_HH$-module, and the $R_HH$-linear map $r\colon R_H[\omega\times G]\to M$ exhibits $M$ as an $R_HH$-linear summand of $R_H[\omega\times G]$; the latter is a free $R_HH$-module as $\delta^*(\omega\times G)$ is a free $H$-set, so $M$ is projective over $R_HH$ as desired.

We now claim that this functor $\mathcal P_G(R)^H\to\mathscr P(R_HH)$ is an equivalence of categories. Indeed, it is fully faithful by the above discussion, so it only remains to show essential surjectivity. For this we let $M$ be any finitely generated projective $R_HH$-module, and we pick generators $m_1,\dots,m_r$. We now choose $r$ distinct $H$-orbits $O_1,\dots,O_r$ of $\delta^*(\omega\times G)$, which yields an $R_HH$-linear surjection $R_H[\omega\times G]\to R_H[O_1\cup\cdots\cup O_r]\cong (R_HH)^r$ which we can postcompose with the map $(R_HH)^r\to M$ sending the $i$-th standard basic vector to $m_i$ to yield an epimorphism $p\colon R_H[\omega\times G]\to M$, which then admits a section $s$ by projectivity. As $M\cong s(M)$ in $\mathscr P(R_HH)$ via $s$, it will then be enough to show that $(s(M),sp)$ defines an element of $\mathcal P_G(R)^H$. Indeed, $s(M)$ is finitely generated as an $R_HH$-module and hence also as an $R$-module as $H$ is finite; moreover, $O_1\cup\cdots\cup O_r\subset\omega\times G$ is finite, so $sp(i,g)=0$ for almost all $(i,g)$ and hence $(s(M),sp)$ is an object of $\mathcal P_G(R)$. However, $s(M)$ is an $R_HH$-submodule as $s$ is $R_HH$-linear, and since also $p$ is $R_HH$-linear, so is $sp$, whence $(s(M),sp)$ is indeed $H$-fixed by the above description of the fixed points.

It remains to show that the forgetful functor $(\delta^*\mathcal P_G(R))^H\to\mathscr P(R_HH)$ is naturally a strong symmetric monoidal functor with respect to the symmetric monoidal structure from Remark~\ref{rk:what-we-can-say-about-the-equivalence}. For this we will make a clever choice of coproducts of $R_HH$-modules: namely, if $\mu\colon\bm 2\times\omega\to\omega$ is our chosen injection and $M,N\subset R[\omega\times G]$ are any subgroups, then a coproduct of $M$ and $N$ in the category of abelian groups is given by $\mu_*(M,N)$ with structure maps
\begin{equation}\label{ex:explicit-coprod}
M\xrightarrow{[\mu(1,\blank),1]}\mu_*(M,0)\xrightarrow{\mu_*(\id_M,0)}\mu_*(M,N)\xleftarrow{\mu_*(0,\id_N)} \mu_*(0,N)\xleftarrow{[\mu(2,\blank),1]} N.
\end{equation}
If we now simply agree to take the coproducts on the category of $R_HH$-modules defining the symmetric monoidal structure on $\mathscr P(R_HH)$ to be given by $(\ref{ex:explicit-coprod})$ whenever $M$ and $N$ are literally $R_HH$-submodules of $R_H[\omega\times G]$, then our above functor is even strict symmetric monoidal by direct inspection, also cf.~\cite[Remark~4.1.37]{g-global}. This completes the proof of the theorem.
\end{proof}

\begin{rk}
We end this section by giving a comparison map $\cat{K}_G(\mathbb P_G(R))\to\cat K_G(R)$ in the quasi-category of $G$-spectra that on $H$-fixed points recovers the inclusion of finitely generated $R_HH$-modules that are projective over $R_HH$ into those modules that are only required to be projective over $R$.

Given the complicated nature of the genuine permutative $G$-category $\mathbb P_G(R)$, this is more easily done using the language of $G$-parsummable categories instead. We first define a $G$-parsummable category $\mathcal Q_G(R)$ analogously to $\mathcal P_G(R)$ where now its objects are simply finitely generated submodules of $R[\omega\times G]$ such that there \emph{exists} an $R$-linear retraction to the inclusion (but the retraction is no longer part of the data). This then receives a natural fully faithful map from $\mathcal P_G(R)$ given by forgetting the retraction, and by an analogous computation to the proof of the above theorem the effect on $\delta^*(\blank)^H$ can be identified with the inclusion of modules that are projective over $R_HH$ into those projective over $R$.

On the other hand, $\mathcal Q_G(R)$ receives a map from the $G$-parsummable category $\mathcal P(R)$ modelling the $G$-global algebraic $K$-theory of rings (defined in the same way as $\mathcal Q_G$, but using the set $\omega$ instead of $\omega\times G$ as a basis) induced by the embedding $R[\omega]\to R[\omega\times G], i\mapsto\sum_{g\in G}(i,g)$. One then easily checks that this is even a $G$-equivalence.

Applying our comparison between $G$-parsummable categories and genuine permutative $G$-categories, we then obtain a map in the underlying quasi-category from $\mathbb P_G(R)$ into the genuine permutative $G$-category $P'$ associated to $\mathcal P(R)$ that on fixed points models the above inclusion. However, by Theorem~\ref{thm:comparison-k-theory} in the appendix together with \cite[Remark~4.1.43]{g-global} the equivariant $K$-theory spectrum $\cat{K}_G(P')$ is equivalent to $\cat{K}_G(R)$, yielding the desired comparison map. (In fact, we already have an equivalence between $P'$ and $\Fun(EG,\mathfrak P(R))$ by Proposition~\ref{prop:comparison-on-saturated} together with \cite[Remark~4.1.37]{g-global}).
\end{rk}

\section{Categorical vs.~simplicial algebras}\label{sec:cat-vs-simpl}
A classical result of Quillen (appearing for example in \cite[VI.3]{illusie}) says that the nerve provides an equivalence between the homotopy theory of small categories (with respect to weak homotopy equivalences) and the usual homotopy theory of spaces. This comparison was later lifted to a comparison between the corresponding notions of $E_\infty$-algebras by Mandell \cite[Theorem~1.9]{mandell}. In \cite[Section~4.3]{g-global} we proved analogous comparisons between permutative $G$-categories (i.e.~\emph{na\"ive} categorical $G$-$E_\infty$-algebras) and various models of genuine $G$-equivariantly or $G$-globally `coherently commutative monoids.' As an upshot of the above results we now also get a corresponding statement for \emph{genuine} categorical algebras:

\begin{thm}\label{thm:g-global-mandell}
Let $\mathcal O$ be a $G$-global $E_\infty$-operad in $\cat{$\bm{E\mathcal M}$-$\bm G$-Cat}$. Then the nerve
\begin{equation*}
\nerve\colon\Alg_{\mathcal O}(\cat{$\bm{E\mathcal M}$-$\bm G$-Cat})_{\textup{$G$-global w.e.}}\to\Alg_{\nerve\mathcal O}(\cat{$\bm{E\mathcal M}$-$\bm G$-SSet})_{\textup{$G$-global w.e.}}
\end{equation*}
descends to an equivalence of quasi-localizations; here we again call a functor of $E\mathcal M$-$G$-categories a \emph{$G$-global weak equivalence} if its nerve is a $G$-global weak equivalence in the usual sense.
\begin{proof}
Appealing to Propositions~\ref{prop:ext-vs-internal-1-cat} and~\ref{prop:associated-external}-$(\ref{item:ae-e-infty})$, it will be enough to prove that $\nerve\colon\allowbreak\Alg_{\mathcal P}(\cat{$\bm G$-Cat})\to\Alg_{\nerve\mathcal P}(\cat{$\bm G$-SSet})$ induces an equivalence for every twisted $G$-global $E_\infty$-operad $\mathcal P$.

To this end, let $\mathfrak P$ be the class of all twisted $G$-global $E_\infty$-operads $\mathcal P$ for which this holds. We will show that $\mathfrak P$ is not empty and that it is closed under $G$-global equivalences; as any two twisted $G$-global $E_\infty$-operads can be connected by a zig-zag of equivalences, this will then imply that $\mathfrak P$ indeed consists of all twisted $G$-global operads.

The closure under $G$-global equivalences is immediate from Proposition~\ref{prop:ext-g-global-change-of-operad}. For the remaining statement we will show that $\mathfrak P$ contains the categorified injection operad $\mathcal I$. For this we observe that we have a commutative diagram
\begin{equation*}
\begin{tikzcd}
\cat{$\bm G$-ParSumCat}\arrow[r, "\nerve"]\arrow[d, hook] & \cat{$\bm G$-ParSumSSet}\arrow[d,hook]\\
\Alg_{\mathcal I}(\cat{$\bm G$-Cat})\arrow[r, "\nerve"'] & \Alg_{\nerve\mathcal I}(\cat{$\bm G$-SSet})
\end{tikzcd}
\end{equation*}
where the vertical inclusions come from Proposition~\ref{prop:parsum-vs-tame-I}. By Theorem~\ref{thm:taming-I-algebras} we are therefore reduced to showing that the top horizontal arrow induces an equivalence after localizing at the $G$-global weak equivalences, which we proved as \cite[Theorem~5.8]{sym-mon-global}.
\end{proof}
\end{thm}

For finite $G$, we also get a version for $G$-equivariant algebras:

\begin{thm}\label{thm:g-equiv-mandell}
Let $\mathcal O$ be a genuine $G$-$E_\infty$-operad in $\cat{$\bm G$-Cat}$. Then the nerve
\begin{equation*}
\nerve\colon\Alg_{\mathcal O}(\cat{$\bm G$-Cat})\to\Alg_{\nerve\mathcal O}(\cat{$\bm G$-SSet})
\end{equation*}
descends to an equivalence of the quasi-localizations at the $G$-weak equivalences.
\begin{proof}
We pick a homomorphism $j\colon G\to\mathcal M$ with universal image, and we consider the homomorphism $\delta\colon G\to E\mathcal M\times G,\delta(g)=(j(g),g)$. If now $\mathcal P$ is any $G$-global $E_\infty$-operad, then $\delta^*\mathcal P$ is a genuine $G$-$E_\infty$-operad, and arguing as before it suffices to prove the theorem for $\mathcal O=\delta^*\mathcal P$. But in the commutative diagram
\begin{equation*}
\begin{tikzcd}
\Alg_{\mathcal P}(\cat{$\bm{E\mathcal M}$-$\bm G$-Cat})_{\textup{$G$-global w.e.}}\arrow[d,"\delta^*"']\arrow[r,"\nerve"] & \Alg_{\nerve\mathcal P}(\cat{$\bm{E\mathcal M}$-$\bm G$-SSet})_{\textup{$G$-global w.e.}}\arrow[d,"\delta^*"]\\
\Alg_{\delta^*\mathcal P}(\cat{$\bm G$-Cat})_{\textup{$G$-equivariant w.e.}}\arrow[r, "\nerve"']&\Alg_{\nerve(\delta^*\mathcal P)}(\cat{$\bm G$-SSet})_{\textup{$G$-equivariant w.e.}}
\end{tikzcd}
\end{equation*}
the vertical arrows induce quasi-localizations by Theorem~\ref{thm:internal-g-global-vs-g-equiv}, while the top horizontal arrow is an equivalence by the previous theorem. The claim follows immediately as the lower horizontal arrow preserves and reflects weak equivalences by definition and as the weak equivalences on the target are saturated, being part of a model structure.
\end{proof}
\end{thm}

\section{Modelling genuine \texorpdfstring{$G$}{G}-\texorpdfstring{$E_\infty$}{E ͚}-algebras by na\"ive ones}\label{sec:main-thm}
Let $G$ be a finite group. In this section we will finally prove that the homotopy theory of genuine permutative $G$-categories with respect to the $G$-equivariant \emph{weak} equivalences can already be modelled by the \emph{na\"ive} ones. For this, we first introduce the following new notion of weak equivalence:

\begin{defi}\label{defi:Ghfp}
A map $f\colon C\to D$ in $\cat{$\bm G$-Cat}$ is called a \emph{$G$-`homotopy' fixed point weak equivalence} (or \textit{$G$-`h'fp weak equivalence} for short) if $\Fun(EH,f)^H$ is a weak homotopy equivalence for every $H\subset G$.
\end{defi}

\begin{warn}
The scare quotes around `homotopy' refer to the fact that the above are homotopy fixed points with respect to the \emph{underlying equivalences} of categories, not with respect to the ($G$-equivariant) weak equivalences, also see Example~\ref{ex:weak-vs-hfp}.
\end{warn}

\begin{rk}\label{rk:g-global-we-g-cat}
The $G$-`h'fp weak equivalences are part of a model structure on $\cat{$\bm G$-Cat}$ modelling ordinary $G$-equivariant unstable homotopy theory, see \cite[Theorem~4.13]{cat-g-global}.

Moreover, we could also impose the stronger condition that $\Fun(EH,\phi^*f)^H$ be a weak homotopy equivalence for {every} finite group $H$ and every homomorphism $\phi\colon H\to G$. This yields the so-called \emph{$G$-global weak equivalences} (or maybe more systematically \emph{$G$-global `h'fp weak equivalences}), and $\cat{$\bm G$-Cat}$ with respect to the $G$-global weak equivalences is another model of unstable $G$-global homotopy theory, see \cite[Theorem~4.3]{cat-g-global}.
\end{rk}

We can now formally state the main result of this section, which for $\mathcal O=E\Sigma_*$ in particular gives a precise version of Theorem~\ref{introthm:main} from the introduction:

\begin{thm}\label{thm:main}
Let $\mathcal O$ be an operad in $\cat{$\bm G$-Cat}$ whose underlying non-equivariant operad is an $E_\infty$-operad. Then the Guillou-May-Shimakawa construction
\begin{equation*}
\Fun(EG,\blank)\colon\Alg_{\mathcal O}(\cat{$\bm G$-Cat})_{\textup{$G$-`h'fp w.e.}}\to\Alg_{\mathcal O^{EG}}(\cat{$\bm G$-Cat})_{\textup{$G$-w.e.}}
\end{equation*}
(see Example~\ref{ex:genuine-G-E-infty}) induces an equivalence of associated quasi-categories.
\end{thm}

\subsection{Comparison of weak equivalences}
By now, we have encountered a variety of different classes of weak equivalences we can put on $\cat{$\bm G$-Cat}$, in particular:
\begin{enumerate}
\item the $G$-equivariant equivalences (functors inducing \emph{equivalences of categories} on all fixed points)
\item the $G$-equivariant \emph{weak} equivalences (functors inducing \emph{weak homotopy equivalences} on all fixed points, or equivalently functors that induce $G$-weak equivalences on nerves)
\item the underlying equivalences of categories (i.e.~$\mathcal T\!riv$-equivalences)
\item the $G$-`h'fp weak equivalences (Definition~\ref{defi:Ghfp} above).
\end{enumerate}
As the interplay of these notions will be crucial to the proof of Theorem~\ref{thm:main}, let us pause for a moment to make the relationship between these different classes precise. Clearly, every $G$-equivariant equivalence is a $G$-equivariant weak equivalence and an underlying equivalence, and moreover any underlying equivalence is a $G$-`h'fp weak equivalence by Remark~\ref{rk:injectively-fibrant}, yielding the following diagram of implications:
\begin{equation}\label{diag:we-comparison}
\begin{tikzcd}
\text{$G$-equivariant equivalence} \arrow[d, Rightarrow]\arrow[dr, Rightarrow] \arrow[r, Rightarrow] & \text{underlying equivalence}\arrow[d, Rightarrow]\\
\text{$G$-equivariant weak equivalence} & \text{$G$-`h'fp weak equivalence}
\end{tikzcd}
\end{equation}
These are in fact \emph{all} implications between these notions unless $G=1$:

\begin{ex}
The unique functor $EG\to *$ is an underlying equivalence, but not a $G$-equivariant (weak) equivalence.
\end{ex}

\begin{ex}\label{ex:weak-vs-hfp}
Not every $G$-equivariant weak equivalence is a $G$-`h'fp weak equivalence, as the following example shows:

Let $f\colon C\to BG$ be a cofibrant replacement in the Thomason model structure on $\cat{Cat}$; in particular, $C$ is a poset \cite[Proposition~5.7]{thomason}. If we now equip both sides with the trivial $G$-action, then $f$ becomes a $G$-equivariant weak equivalence in $\cat{$\bm G$-Cat}$, and we claim that this is not a $G$-`h'fp weak equivalence. Indeed, as both sides carry trivial $G$-action, this would in particular mean that $\Fun(BG,f)\colon \Fun(BG,C)\to\Fun(BG,BG)$ again were a weak homotopy equivalence. But the left hand side is equivalent to $C$ (as $C$ contains no non-trivial isomorphisms), hence in particular has connected nerve, while the identity of $BG$ and the trivial functor $BG\to *\to BG$ belong to different components of the right hand side.
\end{ex}

\begin{ex}
Let $f\colon[0]\to[1]$ be the inclusion of either object, considered as a map of $G$-categories with trivial actions. Obviously, $f$ is not an underlying equivalence of categories; however, it is a $G$-`h'fp weak equivalence: any functor from a groupoid into a poset is constant, so $f^{\myh H}$ is again isomorphic to $f$ for any $H\subset G$, whence in particular a weak homotopy equivalence.
\end{ex}

\subsection{Saturation} We now return to the proof of Theorem~\ref{thm:main}. For this let us first observe that we can easily prove a variant of the theorem with respect to another of the above notions of weak equivalence:

\begin{lemma}\label{lemma:genuine-vs-naive-non-equivariant}
The functor $\Fun(EG,\blank)\colon\Alg_{\mathcal O}(\cat{$\bm G$-Cat})\to\Alg_{\mathcal O^{EG}}(\cat{$\bm G$-Cat})$ descends to an equivalence between the localizations at the \emph{underlying equivalences of categories}.
\begin{proof}
The functor $\rho$ given by restricting along the inclusion $\mathcal O\to\mathcal O^{EG}$ is a left homotopy inverse; on the other hand, $\mathcal O\to\mathcal O^{EG}$ is a $\mathcal T\!\textit{riv}$-equivalence of $\Sigma$-free operads, so Theorem~\ref{thm:change-of-operad-categories} shows that $\rho$ induces an equivalence. The claim follows by $2$-out-of-$3$.
\end{proof}
\end{lemma}

Unfortunately, this of course does not yet tell us anything about Theorem~\ref{thm:main}: while any underlying equivalence is a $G$-`h'fp weak equivalence, the underlying equivalences are unrelated to the $G$-equivariant weak equivalences in general. Nevertheless, we will be able to reduce the theorem to the above lemma; the crucial insight for this is that there is an interesting and wide class of objects for which the notions of $G$-`h'fp weak equivalences and $G$-equivariant weak equivalences coincide:

\begin{defi}
A $G$-category is called \emph{$G$-equivariantly saturated} if the natural map $C^H\to C^{\myh H}=\Fun(EH,C)^H$ is a weak homotopy equivalence for every $H\subset G$.
\end{defi}

\begin{lemma}\label{lemma:sat-vs-all-cat}
Let $\mathcal O$ be any underlying $E_\infty$-operad in $\cat{$\bm G$-Cat}$. Then the inclusion
\begin{equation*}
\Alg_{\mathcal O}(\cat{$\bm G$-Cat})^{\textup{sat}}\hookrightarrow\Alg_{\mathcal O}(\cat{$\bm G$-Cat})
\end{equation*}
of $G$-equivariantly saturated $\mathcal O$-algebras is a homotopy equivalence with respect to the \emph{underlying equivalences of categories} on both sides.
\begin{proof}
A homotopy inverse is given by applying the functor $\Fun(EG,\blank)$ and restricting the operad action along the natural map $\mathcal O\to\mathcal O^{EG}$.
\end{proof}
\end{lemma}

\begin{thm}
Let $\mathcal P$ be any genuine $G$-$E_\infty$-operad in $\cat{$\bm G$-Cat}$. Then the inclusion
\begin{equation}\label{eq:sat-vs-all}
\Alg_{\mathcal P}(\cat{$\bm G$-Cat})^{\textup{sat}}\hookrightarrow\Alg_{\mathcal P}(\cat{$\bm G$-Cat})
\end{equation}
induces an equivalence after quasi-localizing at the \emph{$G$-equivariant weak equivalences} on both sides.
\begin{proof}
Let us consider the class $\mathfrak P$ of all genuine $G$-$E_\infty$-operads $\mathcal P$ for which the statement of the theorem holds. It again suffices to show that $\mathfrak P$ is non-empty and closed under $G$-equivariant equivalences.

For the first statement, we will show that $\mathfrak P$ contains the operad $\mathcal I_G$ where we have identified $G$ with a universal subgroup of $\mathcal M$. For this we apply the isomorphism of $1$-categories $\Alg_{\mathcal I_G}(\cat{$\bm G$-Cat})\cong \Alg_{\mathcal I}(\cat{$\bm G$-Cat})$ from Lemma~\ref{lemma:Delta-right-Bousfield}. Under this identification, the saturated $\mathcal I_G$-algebras correspond precisely to those $\mathcal I$-algebras $C$ for which the natural map
\begin{equation}\label{eq:Alg-I-sat-comparison}
C^\iota\to C^{\myh\iota}=\Fun(EH,C)^\iota
\end{equation}
is a weak homotopy equivalence for every universal $H\subset\mathcal M$ and every \emph{injective} homomorphism $\iota$. We call such $\mathcal I$-algebras \emph{$G$-equivariantly saturated} again, and it will then be enough to show that the inclusion $\Alg_{\mathcal I}(\cat{$\bm G$-Cat})^{\textup{sat}}\hookrightarrow\Alg_{\mathcal I}(\cat{$\bm G$-Cat})$ is an equivalence with respect to the $G$-equivariant weak equivalences. For this we will prove more generally:

\begin{claim*}
Consider the commutative diagram
\begin{equation}\label{diag:tame-sat-inclusions}
\begin{tikzcd}
{\Alg_{\mathcal I}(\cat{$\bm G$-Cat})^{\tau,\textup{sat}}}\arrow[r,hook]\arrow[d,hook]&{\Alg_{\mathcal I}(\cat{$\bm G$-Cat})^{\tau}}\arrow[d,hook]\\
{\Alg_{\mathcal I}(\cat{$\bm G$-Cat})^{\textup{sat}}}\arrow[r,hook]&{\Alg_{\mathcal I}(\cat{$\bm G$-Cat})}
\end{tikzcd}
\end{equation}
of inclusions, where $\Alg_{\mathcal I}(\cat{$\bm G$-Cat})^{\tau,\textup{sat}}$ denotes the subcategory of those $\mathcal I$-algebras that are both tame and saturated. Then all of the above functors induce equivalences (a) after quasi-localizing at the underlying equivalences of categories, as well as (b) after quasi-localizing at the $G$-equivariant weak equivalences.
\begin{proof}
For the first statement we observe that the right hand arrow induces an equivalence with respect to the $G$-global equivalences by Theorem~\ref{thm:taming-I-algebras}, hence in particular also with respect to the underlying equivalences. Moreover, the lower horizontal arrow has a homotopy inverse induced by $\Fun(EG,\blank)$ according to Lemma~\ref{lemma:sat-vs-all-cat}, and one immediately checks that this also provides a homotopy inverse to the top horizontal arrow. Thus, also the left hand vertical arrow induces an equivalence by $2$-out-of-$3$.

For the second statement, we conclude in the same way as before that the right hand vertical arrow descends to an equivalence, and so does the left hand vertical arrow by the first statement since the $G$-equivariant weak equivalences \emph{between saturated $G$-categories} are coarser than the underlying equivalences of categories. By another application of $2$-out-of-$3$ it will then be enough to show that the top horizontal arrow is a homotopy equivalence.

For this we observe that the equivalence $\Alg_{\mathcal I}(\cat{$\bm G$-Cat})\simeq\cat{$\bm G$-ParSumCat}$ from Proposition~\ref{prop:parsum-vs-tame-I} preserves underlying categories, so that it suffices to show that the inclusion $\cat{$\bm G$-ParSumCat}^{\textup{sat}}\hookrightarrow\cat{$\bm G$-ParSumCat}$ of the $G$-equivariantly saturated $G$-parsummable categories is a homotopy equivalence. This is however immediate from \cite[Theorem~5.9]{sym-mon-global}.
\end{proof}
\end{claim*}

It remains to show that $\mathfrak P$ is closed under $G$-equivariant equivalences. For this we consider any $G$-equivariant equivalence $f\colon\mathcal P\to\mathcal Q$, which induces a commutative square
\begin{equation*}
\begin{tikzcd}
\Alg_{\mathcal Q}(\cat{$\bm G$-Cat})^{\textup{sat}}\arrow[r,hook]\arrow[d, "f^*"'] & \Alg_{\mathcal Q}(\cat{$\bm G$-Cat})\arrow[d, "f^*"]\\
\Alg_{\mathcal P}(\cat{$\bm G$-Cat})^{\textup{sat}}\arrow[r,hook] & \Alg_{\mathcal P}(\cat{$\bm G$-Cat}).
\end{tikzcd}
\end{equation*}
The vertical arrow on the right induces an equivalence after quasi-localizing at the $G$-equivariant equivalences by Theorem~\ref{thm:change-of-operad-categories}, hence in particular with respect to the $G$-equivariant weak equivalences or with respect to the underlying equivalences of categories. By $2$-out-of-$3$ it will then suffice to show that also the left hand vertical arrow induces an equivalence with respect to the $G$-equivariant weak equivalences, for which it is as before enough to prove this for the underlying equivalences of categories. This is however in turn immediate from Lemma~\ref{lemma:sat-vs-all-cat} and $2$-out-of-$3$.
\end{proof}
\end{thm}

\begin{proof}[Proof of Theorem~\ref{thm:main}]
It is clear that $\Fun(EG,\blank)$ factors through the full subcategory $\Alg_{\mathcal O^{EG}}(\cat{$\bm G$-Cat})^\textup{sat}$ of $G$-equivariantly saturated $\mathcal O^{EG}$-algebras. As $\mathcal O^{EG}$ is a genuine $G$-$E_\infty$-operad, it is therefore enough to show by the previous theorem that
\begin{equation}\label{eq:Fun-EG-factorization}
\Fun(EG,\blank)\colon\Alg_{\mathcal O}(\cat{$\bm G$-Cat})\to\Alg_{\mathcal O^{EG}}(\cat{$\bm G$-Cat})^{\textup{sat}}
\end{equation}
induces an equivalence with respect to the $G$-weak equivalences on the target and the $G$-`h'fp weak equivalences on the source. For this we observe that $\Fun(EG,\blank)$ preserves and reflects weak equivalences and that the weak equivalences on both source and target are coarser than the underlying equivalences of categories; this again critically uses that we restricted to the \emph{saturated} algebras in the target. It is therefore once more enough to show that $(\ref{eq:Fun-EG-factorization})$ induces an equivalence when we equip source and target with the underlying equivalences of categories. By another application of Lemma~\ref{lemma:sat-vs-all-cat} we are then altogether reduced to showing that $\Fun(EG,\blank)\colon\Alg_{\mathcal O}(\cat{$\bm G$-Cat})\to\Alg_{\mathcal O^{EG}}(\cat{$\bm G$-Cat})$ descends to an equivalence when we equip both sides with the {underlying equivalences of categories}, which is precisely the content of Lemma~\ref{lemma:genuine-vs-naive-non-equivariant}.
\end{proof}

In fact, similar arguments also show (for not necessarily finite $G$):

\begin{thm}
Let $\mathcal O$ be any operad in $\cat{$\bm G$-Cat}$ that forgets to an $E_\infty$-operad in $\cat{Cat}$. Then
\begin{equation*}
\Fun(E\mathcal M,\blank)\colon\Alg_{\mathcal O}(\cat{$\bm G$-Cat})\to\Alg_{\mathcal O^{E\mathcal M}}(\cat{$\bm{E\mathcal M}$-$\bm G$-Cat})
\end{equation*}
descends to an equivalence with respect to the $G$-global weak equivalences on both sides.\qed
\end{thm}

\begin{rk}\label{rk:counterexamples}
We can also equip $\Alg_{\mathcal O^{EG}}(\cat{$\bm G$-Cat})$ with the $G$-equivariant equivalences of categories; however, in this case the resulting functor $\Alg_{\mathcal O}(\cat{$\bm G$-Cat})\to\Alg_{\mathcal O^{EG}}(\cat{$\bm G$-Cat})$ is not even essentially surjective on homotopy categories unless $G=1$ as we will show now:

Since the class of saturated $G$-categories is closed under $G$-equivalences, it suffices to construct a non-saturated $\mathcal O^{EG}$-algebra. To this end, we pick $g\in G\setminus\{1\}$ and write $H\mathrel{:=}\langle g\rangle$ for the subgroup generated by $g$. This is cyclic, hence in particular abelian. Thus, the category $C$ with one object and endomorphism monoid given by $H$ can be equipped (in a unique way) with the structure of a commutative monoid in $\cat{Cat}$. Equipping $C$ with the trivial $G$-action therefore yields an algebra in $\cat{$\bm G$-Cat}$ over the terminal operad, and restricting along the unique operad map $\mathcal O^{EG}\to *$ then gives $C$ the structure of an $\mathcal O^{EG}$-algebra. To finish the proof it suffices now to show that the composition $C\hookrightarrow\Fun^H(EH,C)\cong \Fun(EH/H,C)\cong \Fun(C,C)$ is not an equivalence. But clearly the identity of $C$ is not contained in the essential image.
\end{rk}

\begin{rk}\label{rk:counterexample-interesting}
For $G$ the cyclic group of order $2$ (and many other groups), the genuine permutative $G$-category $\mathbb P_G(\mathbb Z)$ from Section~\ref{sec:k-theory-group-rings} provides another example of a non-saturated permutative $G$-category. Note that while the above theorem asserts that this is $G$-weakly equivalent to $\Fun(EG,C)$ for some na\"ive permutative $G$-category $C$, this $C$ is far from admitting any nice tractible description. In particular, unlike $\mathbb P_G(\mathbb Z)$, it will have to contain non-invertible morphisms in order to elicit the correct equivariant behaviour, and these have no evident algebraic interpretation.
\end{rk}

With this established, we can deduce Theorem~\ref{introthm:gm-thomason} from our results in \cite{g-global}:

\begin{thm}\label{thm:guillou-may-k-theory}
The Guillou-May construction
\begin{equation*}
\cat{K}_G\colon\Alg_{E\Sigma_*^{EG}}(\cat{$\bm G$-Cat})^\infty_\textup{$G$-w.e.}\to\cat{$\bm G$-Spectra}^\infty
\end{equation*}
of equivariant algebraic $K$-theory \cite[Definition~4.12]{guillou-may} exhibits the quasi-category of connective genuine $G$-spectra as a Bousfield localization of the category of genuine permutative $G$-categories.
\begin{proof}
By Theorem~\ref{thm:main} it is enough to show this after restricting along the above functor $\Fun(EG,\blank)\colon\cat{$\bm G$-PermCat}\to\Alg_{E\Sigma_*^{EG}}(\cat{$\bm G$-Cat})$. This is in turn immediate from \cite[Theorem~4.3.9]{g-global} together with the comparison of the operadic and $\Gamma$-space approach to equivariant algebraic $K$-theory given in \cite{may-merling-osorno}.
\end{proof}
\end{thm}

\section{Categories internal to \texorpdfstring{$G$}{G}-spaces and their algebras}\label{sec:internal}
Let $G$ be a finite group again. Guillou and May \cite{guillou-may} more generally consider genuine and na\"ive permutative (or symmetric monoidal) $G$-categories in the context of categories internal to $\cat{$\bm G$-Top}$, or, equivalently, categories internal to $\cat{Top}$ that are equipped with an additional $G$-action. In this final section, we want to extend the above comparisons to this context.

\subsection{Realization} For this we first need to recall some basics about simplicial spaces and bisimplicial sets (i.e.~`simplicial simplicial sets'), and more generally about $G$-simplicial spaces and $G$-bisimplical sets.

\begin{constr}
The categories
\begin{equation*}
\cat{$\bm G$-STop}=\Fun(\Delta^\op,\cat{$\bm G$-Top})
\qquad\text{and}\qquad
\cat{$\bm G$-BiSSet}=\Fun(\Delta^\op,\cat{$\bm G$-SSet})
\end{equation*}
come with Reedy model structures, and as both $\cat{$\bm G$-Top}$ and $\cat{$\bm G$-SSet}$ are simplicial model categories, we get left Quillen realization functors
\begin{equation*}
\|\blank\|\colon\cat{$\bm G$-STop}\to\cat{$\bm G$-Top}
\qquad\text{and}\qquad
\|\blank\|\colon\cat{$\bm G$-BiSSet}\to\cat{$\bm G$-SSet}.
\end{equation*}
\end{constr}

\begin{defi}
We say that a map of $G$-simplicial spaces or $G$-bisimplicial sets is a \emph{realization $G$-weak equivalence} if it is sent to a weak equivalence under the respective left derived realization functor.
\end{defi}

As every bisimplical set is Reedy cofibrant, the realization functor is actually fully homotopical for $\cat{BiSSet}$. Moreover, we see that applying the usual adjunction $|\blank|\colon\cat{SSet}\rightleftarrows\cat{Top} :\!\Sing$ levelwise yields a Reedy cofibrant replacement in $\cat{$\bm G$-STop}$. As in addition the diagram
\begin{equation*}
\begin{tikzcd}[column sep=large]
\Fun(\Delta^\op,\cat{$\bm G$-SSet})\arrow[r, "{\Fun(\Delta^\op,|\blank|)}"]\arrow[d, "\|\blank\|"'] &[2em] \Fun(\Delta^\op,\cat{$\bm G$-Top})\arrow[d, "\|\blank\|"]\\
\cat{$\bm G$-SSet}\arrow[r, "|\blank|"'] & \cat{$\bm G$-Top}
\end{tikzcd}
\end{equation*}
commutes up to isomorphism (since the usual geometric realization functor is a \emph{simplicial} left adjoint), we conclude:

\begin{lemma}\label{lemma:Sing-create-we}
The functor $\Fun(\Delta^\op,\Sing)\colon\cat{$\bm G$-STop}\to\cat{$\bm G$-BiSSet}$ creates realization $G$-weak equivalences.\qed
\end{lemma}

\subsection{Topological vs.~simplicial categories}\label{subsec:top-vs-simpl} We write $\cat{Cat}_\cat{SSet}$ and $\cat{Cat}_{\cat{Top}}$ for the category of categories internal to simplicial sets and topological spaces, respectively. Just like every ordinary category gives rise to a simplicial set via the usual nerve construction, we have:

\begin{constr}
Let $C\in\cat{$\bm G$-Cat}_{\cat{SSet}}$ or $C\in\cat{$\bm G$-Cat}_{\cat{Top}}$. The \emph{nerve} $\cat{N}C$ of $C$ is the bisimplicial set or simplicial space, respectively, given in degree $n$ by
\begin{equation*}
(\cat{N}C)_n=\underbrace{\Mor(C)\times_{\Ob(C)}\cdots\times_{\Ob(C)}\Mor(C)}_{\textup{$n$ factors}}
\end{equation*}
together with the evident structure maps. Defining $\cat{N}$ analogously on morphisms we then get functors $\cat{N}\colon\cat{$\bm G$-Cat}_{\cat{SSet}}\to\cat{$\bm G$-BiSSet}$ and $\cat{N}\colon\cat{$\bm G$-Cat}_{\cat{Top}}\to\cat{$\bm G$-STop}$.
\end{constr}

\begin{defi}
A functor $f\colon C\to D$ in $\cat{$\bm G$-Cat}_{\cat{Top}}$ or $\cat{$\bm G$-Cat}_{\cat{SSet}}$ is a \emph{$G$-weak equivalence} if $\NERVE f$ is a realization $G$-weak equivalence.
\end{defi}

As both geometric realization as well as the singular set functor preserve finite products, they lift to an adjunction
\begin{equation}\label{eq:internal-cat-real-sing}
|\blank|\colon\cat{$\bm G$-Cat}_{\cat{SSet}}\rightleftarrows\cat{$\bm G$-Cat}_{\cat{Top}} :\!\Sing
\end{equation}
for every (finite) group $G$.

\begin{prop}\label{prop:top-vs-simpl-categories}
The adjunction $(\ref{eq:internal-cat-real-sing})$ is a homotopy equivalence with respect to the $G$-weak equivalences on both sides.
\begin{proof}
As $\Sing$ preserves limits, the diagram
\begin{equation*}
\begin{tikzcd}[column sep=large]
\cat{$\bm G$-Cat}_{\cat{Top}}\arrow[d, "\NERVE"']\arrow[r, "\Sing"] &[2em] \cat{$\bm G$-Cat}_{\cat{SSet}}\arrow[d, "\NERVE"]\\
\cat{$\bm G$-STop}\arrow[r, "{\Fun(\Delta^\op,\Sing)}"'] & \cat{$\bm G$-BiSSet}
\end{tikzcd}
\end{equation*}
commutes up to natural isomorphism; more precisely, the canonical maps
\begin{align*}
&\Sing\Mor(C)\times_{\Sing\Ob(C)}\cdots\times_{\Sing\Ob(C)}\Sing\Mor(C)\\&\quad\to\Sing\big({\Mor(C)\times_{\Ob(C)}\cdots\times_{\Ob(C)}\Mor(C)}\big)
\end{align*}
assemble into an isomorphism $\NERVE(\mathop\Sing C)\cong\mathop\Sing\NERVE(C)$ for every $C\in\cat{$\bm G$-Cat}_{\cat{Top}}$. Together with Lemma~\ref{lemma:Sing-create-we} we conclude that $\Sing\colon\cat{$\bm G$-Cat}_{\cat{Top}}\to\cat{$\bm G$-Cat}_{\cat{SSet}}$ creates $G$-weak equivalences. It therefore only remains to show that the unit $C\to\Sing|C|$ is a $G$-weak equivalence for every simplicial $G$-category $C$. However, using that also $\|\blank\|$ preserves finite limits, the induced map $\NERVE(C)_n\to\NERVE(\Sing |C|)_{n}$ agrees up to isomorphism with the unit
\begin{equation*}
\Mor(C)\times_{\Ob(C)}\cdots\times_{\Ob(C)}\Mor(C)\to\mathop\Sing\big|{\Mor(C)\times_{\Ob(C)}\cdots\times_{\Ob(C)}\Mor(C)}\big|,
\end{equation*}
so it is a $G$-weak equivalence. Thus, $\NERVE(\eta)\colon\NERVE(C)\to\NERVE(\Sing |C|)$ is a levelwise $G$-weak equivalence in $\cat{$\bm G$-BiSSet}$ and hence in particular a realization $G$-weak equivalence as desired.
\end{proof}
\end{prop}

\begin{constr}
Both $\cat{$\bm G$-Cat}_{\cat{SSet}}$ and $\cat{$\bm G$-Cat}_{\cat{Top}}$ are Cartesian closed, so we get internal function categories that we denote by $\Fun$ again. We will make these explicit in the case that $I$ is an ordinary $G$-category, which is the only case we will need below:

For $\cat{$\bm G$-Cat}_\cat{SSet}$, we observe that we can identify categories internal to simplicial sets with simplicial objects in $\cat{Cat}$ as products in functor categories can be computed levelwise. Using this, we can now describe $\Fun(I,C)$ for any $C\in\cat{$\bm G$-Cat}_{\cat{SSet}}$ very easily: namely, it is the simplicial $G$-category whose $G$-category of $n$-simplices is the ordinary functor category $\Fun(I,C_n)$; the simplicial structure maps are given in the obvious way. Put differently, $\Ob\Fun(I,C)$ is the simplicial subset of $\Ob(C)^{\Ob(I)}\times\Mor(C)^{\Mor(I)}$ whose $n$-simplices are the functors $I\to C_n$, and similarly $\Mor\Fun(I,C)$ is defined as a subcomplex of $\Ob(C)^{\Ob(I\times[1])}\times\Mor(C)^{\Mor(I\times[1])}$. The unit and counit are given by applying the usual unit and counit of $I\times\blank\colon\cat{Cat}\rightleftarrows\cat{Cat}:\!\Fun(I,\blank)$ levelwise.

On the other hand, if $D\in\cat{$\bm G$-Cat}_{\cat{Top}}$, then $\Ob\Fun(I,D)$ is the set of all (ordinary) functors $I\to D$, topologized as a subspace of $\Ob(D)^{\Ob(I)}\times\Mor(D)^{\Mor(I)}$, and similarly for $\Mor\Fun(I,D)$. The unit and counit are given by the usual unit and counit of $I\times\blank\colon\cat{Cat}\rightleftarrows\cat{Cat}:\!\Fun(I,\blank)$.
\end{constr}

\begin{defi}
A functor $f$ in $\cat{$\bm G$-Cat}_{\cat{SSet}}$ or $\cat{$\bm G$-Cat}_{\cat{Top}}$ is called a \emph{$G$-`h'fp weak equivalence} if $\Fun(EG,f)$ is a $G$-weak equivalence.
\end{defi}

We now want to prove the following comparison complementing Proposition~\ref{prop:top-vs-simpl-categories}:

\begin{prop}\label{prop:top-vs-simpl-homotopy}
The adjunction $(\ref{eq:internal-cat-real-sing})$ is also a homotopy equivalence with respect to the $G$-`h'fp weak equivalences on both sides.
\end{prop}

Before we can do this, we have to establish a certain compatibility property of the above adjunction with respect to internal function categories:

\begin{constr}
As $|\blank|\colon\cat{$\bm G$-Cat}_{\cat{SSet}}\to\cat{$\bm G$-Cat}_{\cat{Top}}$ preserves finite products, the calculus of mates provides us with a natural map
\begin{equation*}
\alpha_{I,C}\colon |\Fun(I,C)|\to \Fun(|I|,|C|)
\end{equation*}
for every $\cat{$\bm G$-Cat}_{\cat{SSet}}$, and passing to mates again then yields a natural map
\begin{equation*}
\beta_{I,D}\colon \Fun(I,\mathop\Sing D)\to\mathop\Sing\Fun(|I|,D)
\end{equation*}
for every $D\in\cat{$\bm G$-Cat}_{\cat{Top}}$. By the calculus of mates, $\beta$ is actually an isomorphism (being the total mate of the isomorphism $|I\times\blank|\cong I\times|\blank|$), and $\alpha,\beta$ are compatible with the unit of the adjunction $|\blank|\dashv\Sing$ in the sense that the diagram
\begin{equation}\label{eq:alpha-beta-eta}
\begin{tikzcd}
\Fun(I,C)\arrow[r, "\eta"]\arrow[d, "{\Fun(I,\eta)}"'] &[.25em] \mathop\Sing|\Fun(I,C)|\arrow[d,"{\mathop\Sing\alpha_{I,C}}"]\\
\Fun(I,\mathop\Sing|C|)\arrow[r,"{\beta_{I,|C|}}"',"\cong"] & \mathop\Sing\Fun(|I|,|C|)
\end{tikzcd}
\end{equation}
commutes for all $I\in\cat{$\bm G$-Cat}$ and $C\in\cat{$\bm G$-Cat}_{\cat{SSet}}$.
\end{constr}

\begin{lemma}
Let $I$ be a \emph{finite} $G$-category and let $C\in\cat{$\bm G$-Cat}_{\cat{SSet}}$. Then the above map $\alpha\colon|\Fun(I,C)|\to\Fun(I,|C|)$ is an isomorphism.
\begin{proof}
We will show that $\alpha$ induces a homeomorphism on objects; the corresponding claim for morphisms will then follow by replacing $I$ by $I\times[1]$ everywhere.

An $n$-simplex $I\to C_n$ of $\Ob\Fun(I,C)$ corresponds to a functor $I\times\Delta^n\to C$, which gives us a map $|\Delta^n|\times I\cong|\Delta^n\times I|\to |C|$; unravelling definitions, we see that $\alpha\colon|{\Ob\Fun(I,C)}|\to{\Ob\Fun(I,|C|)}$ is given by gluing together all these maps. In particular, the diagram
\begin{equation}\label{diag:alpha-fun}
\begin{tikzcd}
{|{\Ob\Fun(I,C)}|}\arrow[r,hook]\arrow[d, "\alpha"'] & {|{\Ob C^{\Ob I}\times\Mor C^{\Mor I}}|}\arrow[d, "\cong"', "\Phi"]\\
{{\Ob\Fun(I,|C|)}}\arrow[r,hook] & {|{\Ob C}|^{\Ob I}\times|{\Mor C}|^{\Mor I}}
\end{tikzcd}
\end{equation}
commutes, where the map $\Phi$ on the right is the canonical homeomorphism induced by the projections, i.e.~it sends $[F_0,F_1;\alpha]$ to the family that sends $i\in\Ob I$ to $[F_0(i),\sigma]$ and $f\in\Mor(I)$ to $[F_1(f);\sigma]$.

The top horizontal arrow in $(\ref{diag:alpha-fun})$ is an embedding as the geometric realization of an inclusion of simplicial sets, and so is the lower horizontal arrow by definition. To complete the proof it suffices now to show that any $\mathfrak F\in|\Ob C^{\Ob I}\times\Mor C^{\Mor I}|$ for which $\Phi\mathfrak F$ is a functor $I\to|C|$, is already contained in $|\Ob\Fun(I,C)|$.

To this end, we write $\mathfrak F=[F_0,F_1;\sigma]$ with $\sigma\in(\Delta^n)^{\circ}$ for some $n\ge 0$, and it will be enough to prove that $(F_0,F_1)$ defines a functor $I\to C_n$. For this we will only show that $\src\circ F_1=F_0\circ\src$, the argument for the other functoriality properties being similar. Indeed, if $f\colon i\to j$ is any morphism in $I$, then $[F_0(i),\sigma]=(\Phi\mathfrak F)(i)=\src(\Phi\mathfrak F)(f)=[\src F_1(f),\sigma]$; as $\sigma$ lies in the interior of $\Delta^n$ by construction, this already implies that $F_0(i)=\src F_1(f)$ by the basic combinatorics of simplicial sets as desired.
\end{proof}
\end{lemma}

\begin{proof}[Proof of Proposition~\ref{prop:top-vs-simpl-homotopy}]
As $\Sing$ creates $G$-weak equivalences, the isomorphism $\beta_{EG,\blank}$ shows that $\Sing$ also creates $G$-`h'fp weak equivalences. It therefore only remains to show that the unit $\eta\colon C\to\Sing|C|$ is a $G$-`h'fp weak equivalence for every $C\in\cat{$\bm G$-Cat}_{\cat{SSet}}$.

For this, we specialize the commutative diagram $(\ref{eq:alpha-beta-eta})$ to the finite $G$-category $I=EG$. Then the lower horizontal map is an isomorphism, and so is the right hand vertical map by the previous lemma. As moreover the top horizontal arrow is a $G$-weak equivalence by Proposition~\ref{prop:top-vs-simpl-categories}, $2$-out-of-$3$ shows that $\Fun(EG,\eta)$ is a $G$-weak equivalence, i.e.~$\eta$ is a $G$-`h'fp weak equivalence as desired.
\end{proof}

Now let $\mathcal O$ be any operad in $\cat{$\bm G$-Cat}_{\cat{SSet}}$. As $|\blank|$ preserves products, it lifts to a functor $\Alg_{\mathcal O}(\cat{$\bm G$-Cat}_{\cat{SSet}})\to\Alg_{|\mathcal O|}(\cat{$\bm G$-Cat}_{\cat{Top}})$, and likewise $\Sing$ induces $\Alg_{\mathcal P}(\cat{$\bm G$-Cat}_{\cat{Top}})\to\Alg_{\Sing(\mathcal P)}(\cat{$\bm G$-Cat}_{\cat{SSet}})$ for any operad $\mathcal P$ in $\cat{$\bm G$-Cat}_{\cat{Top}}$. If $\mathcal P=|\mathcal O|$, then we can compose this with the restriction along $\mathcal O\to\Sing|\mathcal O|$, yielding an adjunction
\begin{equation}\label{eq:sset-vs-top-alg}
\Alg_{\mathcal O}(\cat{$\bm G$-Cat}_{\cat{SSet}})\rightleftarrows\Alg_{|\mathcal O|}(\cat{$\bm G$-Cat}_{\cat{Top}}).
\end{equation}
Propositions~\ref{prop:top-vs-simpl-categories} and~\ref{prop:top-vs-simpl-homotopy} now immediately imply:

\begin{cor}\label{cor:simplicial-vs-top-alg}
Let $\mathcal O$ be any operad in $\cat{$\bm G$-Cat}_{\cat{SSet}}$. Then $(\ref{eq:sset-vs-top-alg})$ is a homotopy equivalence with respect to the $G$-weak equivalences as well as with respect to the $G$-`h'fp weak equivalences.\qed
\end{cor}

\subsection{Simplicial vs.~ordinary categories}\label{subsec:simpl-vs-ord} It remains to compare algebras of simplicial $G$-categories to those of ordinary $G$-categories. Here the key idea will be to again exploit the identification between simplicial $G$-categories and simplicial objects in $\cat{$\bm G$-Cat}$. We begin by describing the $G$-weak equivalences from this point of view:

\begin{lemma}\label{lemma:nerve-homotopical-internal-simplicial}
A map in $\cat{$\bm G$-Cat}_{\cat{SSet}}$ is a $G$-weak equivalence if and only if it is sent to a realization $G$-weak equivalence under the composition
\begin{equation*}
\cat{$\bm G$-Cat}_{\cat{SSet}}\cong\Fun(\Delta^\op,\cat{$\bm G$-Cat})\xrightarrow{\nerve\circ\blank}\Fun(\Delta^\op,\cat{$\bm G$-SSet}),
\end{equation*}
and analogously for maps in $\Alg_{\mathcal O}(\cat{$\bm G$-Cat}_{\cat{SSet}})$ for any operad $\mathcal O$ in $\cat{$\bm G$-Cat}$.
\begin{proof}
One checks by direct inspection that the diagram
\begin{equation*}
\begin{tikzcd}[column sep=small]
\cat{$\bm G$-Cat}_{\textup{SSet}}\arrow[rr, "\cong"]\arrow[d, "\NERVE"'] && \Fun(\Delta^\op,\cat{$\bm G$-Cat})\arrow[d, "\nerve\circ\blank"]\\
\Fun(\Delta^\op,\cat{$\bm G$-SSet})\arrow[dr, "\diag^*"', bend right=10pt]\arrow[rr, "\textup{twist}^*"]&&\Fun(\Delta^\op,\cat{$\bm G$-SSet})\arrow[dl, "\diag^*", bend left=10pt]\\
&\cat{$\bm G$-SSet}
\end{tikzcd}
\end{equation*}
commutes, where $\textup{twist}$ denotes the functor exchanging the two factors of $\Delta^\op\times\Delta^\op$, and $\diag\colon\Delta^\op\to\Delta^\op\times\Delta^\op$ is the diagonal embedding. The first statement now follows as $\diag^*$ is isomorphic to $\|\blank\|$, and the second statement is a formal consequence of the first one as the forgetful functor $\Alg_{\nerve\mathcal O}(\cat{$\bm G$-SSet})\to\cat{$\bm G$-SSet}$ preserves geometric realization (Proposition~\ref{prop:forget-geometric-realization}).
\end{proof}
\end{lemma}

For suitably nice \emph{model categories} $\mathscr C$, \cite[Theorem~3.6]{rss-realization} shows that the homotopy theory of simplicial objects in $\mathscr C$ with respect to the realization weak equivalences is again equivalent to $\mathscr C$. However, in our situation, we do not have model structures available, so we will argue $\infty$-categorically instead:

\begin{lemma}\label{lemma:hocolim-Delta-op}
Let $\mathscr C$ be a cocomplete quasi-category. Then
\begin{equation*}
\hocolim\colon\Fun(\nerve\Delta^\op,\mathscr C)\rightleftarrows\mathscr C :\!\const
\end{equation*}
is a Bousfield localization.
\begin{proof}
We have to show that for every $X\in\mathscr C$ the counit $\hocolim_{\nerve(\Delta^\op)}\mathop\const X\to X$ is an equivalence. However $\nerve(\Delta^\op)$ is weakly contractible as $\Delta^\op$ has an initial object, so the claim is simply an instance of \cite[Corollary~4.4.4.10]{htt}.
\end{proof}
\end{lemma}

Before we can apply this, however, there is a technical hurdle to overcome---namely, we have to compare simplicial objects in the underlying $1$-categories with simplicial objects in the associated quasi-categories:

\begin{prop}\label{prop:where-to-diagram}
Let $I$ be any small category.
\begin{enumerate}
\item Let $\mathcal O$ be a genuine $G$-$E_\infty$-operad in $\cat{$\bm G$-Cat}$. Then the natural map
\begin{equation*}
\Fun\big(I,\Alg_{\mathcal O}(\cat{$\bm G$-Cat})\big)^\infty_{\textup{levelwise $G$-w.e.}}\to\Fun\big(\nerve I,\Alg_{\mathcal O}(\cat{$\bm G$-Cat})^\infty_{\textup{$G$-w.e.}}\big)
\end{equation*}
is an equivalence.
\item Let $\mathcal P$ be an underlying $E_\infty$-operad in $\cat{$\bm G$-Cat}$. Then the natural map
\begin{equation*}
\Fun\big(I,\Alg_{\mathcal P}(\cat{$\bm G$-Cat})\big)^\infty_{\textup{lev.~$G$-`h'fp w.e.}}\to\Fun\big(\nerve I,\Alg_{\mathcal P}(\cat{$\bm G$-Cat})^\infty_{\textup{$G$-`h'fp w.e.}}\big)
\end{equation*}
is an equivalence.
\end{enumerate}
\begin{proof}
First some terminology: recall that a \emph{relative category} \cite[3.1]{barwick-kan} is a pair of a category $\mathscr C$ together with a wide subcategory $W$, called the \emph{weak equivalences} of $\mathscr C$; as usual, we will also simply refer to $\mathscr C$ as a relative category if $W$ is understood. Any relative category $(\mathscr C,W)$ again admits an $\infty$-categorical localization $\mathscr C_W^\infty$ (or $\mathscr C^\infty$ for short).

We now call $\mathscr C$ \emph{hereditary} if the natural map $\Fun(I,\mathscr C)^\infty_\text{level w.e.}\to \Fun(\nerve I,\mathscr C^\infty)$ is an equivalence for every $I\in\cat{Cat}$. Moreover, a relative equivalence $f$ of relative categories (i.e.~a homotopical functor inducing equivalences of associated quasi-categories) will be called \emph{hereditary} if also $f^I$ is a relative equivalence (with respect to the levelwise weak equivalences) for every $I$; for example, every homotopy equivalence is hereditary.

By $2$-out-of-$3$ all relative equivalences of hereditary relative categories are hereditary, and conversely the hereditary relative categories are closed under hereditary relative equivalences. Moreover, \cite[Theorem~7.9.8 and Remark~7.9.7]{cisinski-book} shows that the underlying relative category of any model category is hereditary. It will therefore suffice to connect $\Alg_{\mathcal O}(\cat{$\bm G$-Cat})_{\textup{$G$-w.e.}}$ and $\Alg_{\mathcal P}(\cat{$\bm G$-Cat})_{\textup{$G$-`h'fp w.e.}}$ by zig-zags of hereditary relative equivalences to suitable model categories.

For this we will unravel parts of the proof of Theorem~\ref{thm:g-equiv-mandell}. We first recall we have a zig-zag of relative equivalences
\begin{equation}\label{eq:zig-zag-alg-O-I-G}
\Alg_{\mathcal O}(\cat{$\bm G$-Cat})\xrightarrow{p_1^*}\Alg_{\mathcal O\times\mathcal I_G}(\cat{$\bm G$-Cat})\xleftarrow{p_2^*}\Alg_{\mathcal I_G}(\cat{$\bm G$-Cat})
\end{equation}
with respect to the \emph{$G$-equivariant equivalences}. Since these are maps between model categories, they are hereditary, i.e.~for every $I\in\cat{Cat}$ also the induced functors
\begin{equation*}
\Alg_{\mathcal O}(\cat{$\bm G$-Cat})^I\to\Alg_{\mathcal O\times\mathcal I_G}(\cat{$\bm G$-Cat})^I\gets\Alg_{\mathcal I_G}(\cat{$\bm G$-Cat})^I
\end{equation*}
are relative equivalences with respect to the levelwise $G$-equivalences of categories, hence also with respect to the levelwise $G$-\emph{weak} equivalences everywhere (as the latter are saturated and each of the above functors creates $G$-weak equivalences). Put differently, the zig-zag $(\ref{eq:zig-zag-alg-O-I-G})$ also consists of hereditary relative equivalences with respect to the $G$-equivariant weak equivalences everywhere.

Similarly, Theorem~\ref{thm:taming-I-algebras} implies that $\cat{$\bm G$-ParSumCat}\hookrightarrow\Alg_{\mathcal I}(\cat{$\bm G$-Cat})$ is a hereditary relative equivalence with respect to the $G$-global equivalences and hence also for the $G$-equivariant weak equivalences. On the other hand, the functor $\Delta\colon\Alg_{\mathcal I}(\cat{$\bm G$-Cat})\to\Alg_{\mathcal I_G}(\cat{$\bm G$-Cat})$ is right Quillen (for the $G$-global model structure on the source) and induces a quasi-localization at the $G$-equivariant equivalences by Lemma~\ref{lemma:Delta-right-Bousfield} and its proof. Precomposing its left adjoint $\nabla$ with a functorial cofibrant replacement then provides a homotopy inverse to $\Delta$ with respect to the $G$-equivariant equivalences on both sides, hence also for the $G$-equivariant weak equivalences. In particular, also $\Delta$ is a hereditary relative equivalence for the $G$-equivariant weak equivalences on both sides.

Moreover, the nerve $\cat{$\bm G$-ParSumCat}\to\cat{$\bm G$-ParSumSSet}$ is again a homotopy equivalence (with respect to the $G$-global and hence also with respect to the $G$-equivariant weak equivalences) by \cite[Theorem~5.8]{sym-mon-global}, hence hereditary. Arguing just like above, we then finally get a zig-zag of hereditary relative equivalences between $\cat{$\bm G$-ParSumSSet}$ and the model category $\Alg_{\mathcal I_G}(\cat{$\bm G$-SSet})$, which completes the proof of the first statement.

The proof of the second statement is similar: namely, as above we get a zig-zag of hereditary relative equivalences between $\Alg_{\mathcal P}(\cat{$\bm G$-Cat})$ and $\Alg_{E\Sigma_*}(\cat{$\bm G$-Cat})$ with respect to the \emph{underlying equivalences} of categories, hence also with respect to the $G$-`h'fp weak equivalences. However, $\Alg_{E\Sigma_*}(\cat{$\bm G$-Cat})$ is even isomorphic as a $1$-category to $\cat{$\bm G$-PermCat}$ (and this isomorphism respects underlying $G$-categories, hence also the weak equivalences in question), which is in turn homotopy equivalent to $\cat{$\bm G$-ParSumSSet}$ with respect to the {$G$-global weak equivalences} \cite[Theorems~5.8 and~6.9]{sym-mon-global}, hence in particular with respect to the $G$-`h'fp weak equivalences on $\cat{$\bm G$-PermCat}$ and the $G$-weak equivalences on $\cat{$\bm G$-ParSumSSet}$. The claim follows as before.
\end{proof}
\end{prop}

\begin{prop}\label{prop:ord-vs-sset-alg}
For every genuine $G$-$E_\infty$-operad $\mathcal O$ the inclusion
\begin{equation*}
\Alg_{\mathcal O}(\cat{$\bm G$-Cat})_{\textup{$G$-w.e.}}\hookrightarrow\Alg_{\mathcal O}(\cat{$\bm G$-Cat}_{\cat{SSet}})_{\textup{$G$-w.e.}}
\end{equation*}
induces an equivalence of associated quasi-categories.
\begin{proof}
We first observe that the composition
\begin{equation*}
\Alg_{\mathcal O}(\cat{$\bm G$-Cat})\hookrightarrow\Alg_{\mathcal O}(\cat{$\bm G$-Cat}_{\cat{SSet}})\cong\Fun\big(\Delta^\op,\Alg_{\mathcal O}(\cat{$\bm G$-Cat})\big)
\end{equation*}
is just the functor sending an $\mathcal O$-algebra to the constant simplicial object. By Proposition~\ref{prop:where-to-diagram} together with Lemma~\ref{lemma:hocolim-Delta-op} it will therefore be enough to show that a map in $\Fun\big(\Delta^\op,\Alg_{\mathcal O}(\cat{$\bm G$-Cat})\big)$ is a $G$-weak equivalence if and only if its image under the composition
\begin{equation*}
\Fun\big(\Delta^\op,\Alg_{\mathcal O}(\cat{$\bm G$-Cat})\big)\to\Fun\big(\nerve\Delta^\op,\Alg_{\mathcal O}(\cat{$\bm G$-Cat})^\infty\big)\xrightarrow{\hocolim}\Alg_{\mathcal O}(\cat{$\bm G$-Cat})^\infty
\end{equation*}
is an equivalence. However, $\nerve\colon\Alg_{\mathcal O}(\cat{$\bm G$-Cat})\to\Alg_{\nerve\mathcal O}(\cat{$\bm G$-SSet})$ descends to an equivalence by Theorem~\ref{thm:g-equiv-mandell}, so a map is inverted by the above if and only if it is inverted by
\begin{align*}
\Fun\big(\Delta^\op,\Alg_{\mathcal O}(\cat{$\bm G$-Cat})\big)&\xrightarrow{\nerve}\Fun\big(\Delta^\op,\Alg_{\mathcal O}(\cat{$\bm G$-SSet})\big)\\
&\to\Fun\big(\nerve\Delta^\op,\Alg_{\mathcal O}(\cat{$\bm G$-SSet})^\infty\big)\\&\xrightarrow{\hocolim}\Alg_{\mathcal O}(\cat{$\bm G$-SSet})^\infty.
\end{align*}
Finally, as $\Alg_{\mathcal O}(\cat{$\bm G$-SSet})$ is a simplicial model category, the composition of the final two arrows is induced by geometric realization. The claim therefore follows from Lemma~\ref{lemma:nerve-homotopical-internal-simplicial}.
\end{proof}
\end{prop}

Now we can prove:

\begin{thm}\label{thm:main-topological}
Let $\mathcal O$ be an underlying $E_\infty$-operad in $\cat{$\bm G$-Cat}$. Then all the maps in the commutative diagram
\begin{equation*}
\begin{tikzcd}[column sep=large]
\Alg_{\mathcal O}(\cat{$\bm G$-Cat})_\textup{$G$-`h'fp}\arrow[r, "{\Fun(EG,\blank)}"]\arrow[d, hook] &[1.5em] \Alg_{\mathcal O^{EG}}(\cat{$\bm G$-Cat})_\textup{$G$-w.e.}\arrow[d,hook]\\
\Alg_{\mathcal O}(\cat{$\bm G$-Cat}_{\cat{SSet}})_\textup{$G$-`h'fp}\arrow[r, "{\Fun(EG,\blank)}"]\arrow[d, "|\blank|"'] & \Alg_{\mathcal O^{EG}}(\cat{$\bm G$-Cat}_{\cat{SSet}})_\textup{$G$-w.e.}\arrow[d,"|\blank|"]\\
\Alg_{\mathcal O}(\cat{$\bm G$-Cat}_{\cat{Top}})_\textup{$G$-`h'fp}\arrow[r, "{\Fun(EG,\blank)}"'] & \Alg_{\mathcal O^{EG}}(\cat{$\bm G$-Cat}_{\cat{Top}})_\textup{$G$-w.e.}
\end{tikzcd}
\end{equation*}
descend to equivalences of quasi-categories.
\begin{proof}
The right hand vertical inclusion in the upper square induces an equivalence of associated quasi-categories by the previous proposition, and so do the vertical functors in the lower square by Corollary~\ref{cor:simplicial-vs-top-alg}. Moreover, the top horizontal arrow induces an equivalence by Theorem~\ref{thm:main}, and together with Proposition~\ref{prop:where-to-diagram} we see that the middle horizontal arrow becomes an equivalence when we define a map $f$ in the target to be a weak equivalence if applying the nerve levelwise turns it into a levelwise weak equivalence of $G$-bisimplicial sets and a map $f'$ in the source if $\Fun(EG,f')$ has the same property. Thus, it also becomes an equivalence when we quasi-localize the target at the $G$-weak equivalences and the source at those maps inverted by $\Fun(EG,\blank)$, i.e.~the $G$-`h'fp weak equivalences. By $2$-out-of-$3$ we then see that also the top left vertical map and the bottom horizontal map descend to equivalences, which completes the proof of the theorem.
\end{proof}
\end{thm}

As an immediate consequence we now get the following version of Theorem~\ref{thm:g-equiv-mandell} for simplicial $G$-categories:

\begin{cor}
In the above situation,
\begin{equation*}
\|\blank\|\circ\NERVE\colon\Alg_{\mathcal O}(\cat{$\bm G$-Cat}_{\cat{SSet}})_{\textup{$G$-w.e.}}\to\Alg_{\nerve\mathcal O}(\cat{$\bm G$-SSet})_{\textup{$G$-w.e.}}
\end{equation*}
induces an equivalence of associated quasi-categories.\qed
\end{cor}

Next, we come to a version of this statement for $\cat{$\bm G$-Cat}_{\cat{Top}}$. Here a slight subtlety arises: namely, $\|\blank\|\circ\NERVE\colon\cat{$\bm G$-Cat}_{\cat{Top}}\to\cat{$\bm G$-Top}$ is not homotopical, while the usual `fat' realizations do not preserve products up to \emph{isomorphism}. However, we can solve this issue by composing $\|\blank\|$ with the product-preserving cofibrant replacement given by the usual geometric realization-singular set adjunction:

\begin{cor}
The functor $\cat{$\bm G$-Cat}_{\cat{Top}}\to\cat{$\bm G$-Top}$ given on objects by $C\mapsto \||\blank|\circ\Sing\circ\NERVE C\|$ and likewise on morphisms induces an equivalence
\begin{equation*}
\Alg_{\mathcal O}(\cat{$\bm G$-Cat}_{\cat{Top}})^\infty_{\textup{$G$-w.e.}}\to \Alg_{|\mathcal O|}(\cat{$\bm G$-Top})^\infty.\pushQED\qed\qedhere\popQED
\end{equation*}
\end{cor}

Finally, we also get a version of Theorem~\ref{thm:guillou-may-k-theory} in the internal context:

\begin{cor}
The Guillou-May construction of equivariant algebraic $K$-theory exhibits the quasi-category of connective genuine $G$-spectra as a quasi-localization of the category of genuine permutative $G$-categories internal to $\cat{Top}$.\qed
\end{cor}

\appendix
\section{Comparison of equivariant \texorpdfstring{$K$}{K}-theory constructions}
Throughout, let $G$ be a finite group. Given a $G$-parsummable category $C$, we can use Theorem~\ref{thm:genuine-sym-vs-parsum} to build a genuine permutative $G$-category $\Psi(C)$ from this, which then via the general equivariant infinite loop space machinery of Guillou and May \cite{guillou-may} gives rise to a \emph{$G$-equivariant $K$-theory spectrum} $\cat{K}_G\Psi(C)$. On the other hand, \cite[Definition~4.1.10]{g-global} produces a \emph{$G$-global $K$-theory spectrum} directly from $C$, which by considering the same object with respect to a coarser notion of weak equivalence in particular gives us another $G$-equivariant spectrum $\mathbb K_G(C)$. The goal of this short appendix is to prove that these two spectra actually agree. In fact, we will prove more generally:

\begin{thm}\label{thm:comparison-k-theory}
Let $\Theta\colon\cat{$\bm G$-ParSumCat}^\infty_{\textup{$G$-equiv}}\to\Alg_{E\Sigma_*^{EG}}(\cat{$\bm G$-Cat})^\infty_\textup{$G$-equiv.}$ be any functor that preserves underlying $G$-categories, e.g.~the equivalence $\Psi$ considered above. Then the diagram
\begin{equation}\label{diag:comp-k-theo}
\begin{tikzcd}[column sep=-10pt]
\cat{$\bm G$-ParSumCat}_\textup{$G$-equiv.}^\infty\arrow[rr, "\Theta", "\simeq"']\arrow[dr, bend right=15pt, "\mathbb{K}_G"'] && \Alg_{E\Sigma_*^{EG}}(\cat{$\bm G$-Cat})_\textup{$G$-equiv.}^\infty\arrow[dl, bend left=15pt, "\cat K_G"]\\
& \cat{$\bm G$-Spectra}_\textup{$G$-w.e.}^\infty
\end{tikzcd}
\end{equation}
commutes up to preferred equivalence.
\end{thm}

Somewhat amusingly, we will never need to know how the above $K$-theory constructions actually look like---instead, we will deduce the theorem formally from the results of this paper together with \cite{g-global,sym-mon-global}.

We begin by making the equivalence $\Psi$ (and more generally any $\Theta$ as above) explicit in a special case, as promised in Remark~\ref{rk:comp-on-sat-promise}. For this we recall from \cite[Definition~4.1.25]{g-global} that any na\"ive permutative $G$-category $C$ gives rise to an (explicit) $G$-parsummable category $\Phi^\textup{sat}(C)$; again, the precise construction will not be relevant.

\begin{prop}\label{prop:comparison-on-saturated}
Let $\Theta$ be as above. Then
\begin{equation}\label{diag:comp-saturated}
\begin{tikzcd}[column sep=-10pt]
& \cat{$\bm G$-PermCat}^\infty_\textup{$G$-equiv.}\arrow[dl, bend right=15pt, "\Phi^\textup{sat}"']\arrow[dr, bend left=15pt, "{\Fun(EG,\blank)}"]\\
\cat{$\bm G$-ParSumCat}_\textup{$G$-equiv.}^\infty\arrow[rr, "\Theta"'] && \Alg_{E\Sigma_*^{EG}}(\cat{$\bm G$-Cat})^\infty_\textup{$G$-equiv.}
\end{tikzcd}
\end{equation}
commutes up to preferred equivalence.
\begin{proof}
Let us call a $G$-category \emph{strongly saturated} if for every $H\subset G$ the canonical map $C^H\to \Fun(EH,C)^H\simeq \Fun(EG,C)^H$ of fixed points into (categorical) homotopy fixed points is an equivalence (the $G$-global version of this was simply called \emph{saturated} in \cite{g-global,sym-mon-global}). If $A$ is any $G$-category, then $\Fun(EG,A)$ is strongly saturated \cite[Lemma~2.8]{merling}, and so is the underlying $G$-category of $\Phi^\text{sat}(B)$ for any permutative $G$-category $B$ by \cite[Theorem~4.1.23]{g-global}. Thus, both composites in (\ref{diag:comp-saturated}) factor through the full subcategory spanned by the strongly saturated genuine permutative $G$-categories. However, a map of strongly saturated $G$-categories is a $G$-equivariant equivalence if and only if it is an underyling equivalence, while the same argument as in Lemma~\ref{lemma:sat-vs-all-cat} shows that the inclusion of strongly saturated algebras into all algebras is an equivalence with respect to the \emph{underlying} equivalences. Altogether we are therefore reduced to proving the claim when we consider the composites as functors into $\Alg_{E\Sigma_*^{EG}}(\cat{$\bm G$-Cat})_\text{underlying}^\infty$.

However, in this case we have an isomorphism
\begin{equation*}
\Alg_{E\Sigma_*^{EG}}(\cat{$\bm G$-Cat})_\text{underlying}\cong\Fun(BG,\Alg_{E\Sigma_*^{EG}}(\cat{Cat})_\text{equiv.})
\end{equation*}
preserving and reflecting weak equivalences, hence an equivalence
\begin{equation*}
\Alg_{E\Sigma_*^{EG}}(\cat{$\bm G$-Cat})_\text{underlying}^\infty\simeq\Fun(BG,\CMon(\cat{Cat}^\infty_\textup{equiv.}))
\end{equation*}
lying over the canonical equivalence $\cat{$\bm G$-Cat}^\infty_\textup{underlying}\simeq\Fun(BG,\cat{Cat}^\infty_\text{equiv.})$ from \cite[Theorem~7.9.8 and Remark~7.9.7]{cisinski-book}. As taking commutative monoid objects in the $\infty$-categorical sense commutes with functor categories, this then gives an equivalence
\begin{equation*}
\Alg_{E\Sigma_*^{EG}}(\cat{$\bm G$-Cat})_\text{underlying}^\infty\simeq\CMon(\Fun(BG,\cat{Cat}^\infty_\textup{equiv.}))
\end{equation*}
over the same equivalence as before, and hence by functoriality of $\CMon$ finally an equivalence
\begin{equation*}
\Alg_{E\Sigma_*^{EG}}(\cat{$\bm G$-Cat})_\text{underlying}^\infty\simeq\CMon(\cat{$\bm G$-Cat}^\infty_\textup{underlying})
\end{equation*}
over $\cat{$\bm G$-Cat}^\infty_\textup{underlying}$. By \cite[Corollary~2.5-(iii)]{ggn} it therefore suffices to construct the equivalence filling (\ref{diag:comp-saturated}) after postcomposition with the forgetful functor to $\cat{$\bm G$-Cat}^\infty_\text{underlying}$, where both paths through the diagram can simply be identified with the forgetful functor (see \cite[proof of Lemma~6.12]{g-global} for the left hand composite).
\end{proof}
\end{prop}

\begin{proof}[Proof of Theorem~\ref{thm:comparison-k-theory}]
Both $K$-theory functors actually invert $G$-equivariant \emph{weak} equivalences, and so does $\Theta$ by assumption; we may therefore prove the theorem after localizing at the $G$-equivariant weak equivalences instead. In this case, both $\Fun(EG,\blank)$ and $\Phi^{\textup{sat}}$ become equivalences with respect to the $G$-`h'fp weak equivalences on $\cat{$\bm G$-PermCat}$ by Theorem~\ref{thm:main} and \cite[Theorem~6.9]{sym-mon-global}, respectively. By the previous proposition it therefore suffices to prove the theorem after precomposing (\ref{diag:comp-k-theo}) with the maps from (\ref{diag:comp-saturated}), i.e.~to construct an equivalence filling
\begin{equation*}
\begin{tikzcd}
\cat{$\bm G$-PermCat}^\infty_\text{$G$-`h'fp}\arrow[r, "{\Fun(EG,\blank)}"]\arrow[d, "\Phi^\textup{sat}"'] &[2em] \Alg_{E\Sigma_*^{EG}}(\cat{$\bm G$-Cat})^\infty_\text{$G$-w.e.}\arrow[d, "\cat{K}_G"]\\
\cat{$\bm G$-ParSumCat}^\infty_\textup{$G$-w.e.}\arrow[r, "\mathbb K_G"'] & \cat{$\bm G$-Spectra}_\textup{$G$-w.e.},
\end{tikzcd}
\end{equation*}
which is done in \cite[Theorem~4.1.40]{g-global}.
\end{proof}

\begin{rk}
A slight variation of the above arguments yields the following uniqueness result for our comparison $\cat{$\bm G$-ParSumCat}_\textup{$G$-w.e.}^\infty\simeq\Alg_{E\Sigma_*^{EG}}(\cat{$\bm G$-Cat})^\infty_\text{$G$-w.e.}$: for any functor
\begin{equation*}
\cat{$\bm G$-ParSumCat}_\textup{$G$-equivalences}^\infty\to \Alg_{E\Sigma_*^{EG}}(\cat{$\bm G$-Cat})^\infty_\text{$G$-equivalences}
\end{equation*}
compatible with the forgetful maps to $\cat{$\bm G$-Cat}_\textup{$G$-equiv.}^\infty$ (note the finer notion of weak equivalence!), the induced functor $\cat{$\bm G$-ParSumCat}_\textup{$G$-w.e.}^\infty\to\Alg_{E\Sigma_*^{EG}}(\cat{$\bm G$-Cat})^\infty_\text{$G$-w.e.}$ is canonically equivalent to the equivalence $\Psi$ we constructed, in particular itself an equivalence.

I do not know whether this holds more generally without the assumption that our comparison comes from a functor of the localizations at the $G$-equivariant equivalences, or equivalently, whether the space of endomorphism of $\Alg_{E\Sigma_*^{EG}}(\cat{$\bm G$-SSet})^\infty$ over $\cat{$\bm G$-SSet}^\infty$ is contractible. Similarly, it is not clear whether $\Alg_{E\Sigma_*^{EG}}(\cat{$\bm G$-Cat})_\textup{$G$-equiv.}^\infty$ has non-trivial endomorphisms over $\cat{$\bm G$-Cat}_\textup{$G$-equiv.}^\infty$.
\end{rk}

\begin{ack}
    I would like to thank Bastiaan Cnossen for helpful discussions and Branko Juran for suggesting a simplification of the example in Warning~\ref{warning:no-elmendorf}. I am moreover grateful to two anonymous referees for helpful feedback, which in particular prompted the inclusion of Subsection~\ref{subsec:explicit-comparison}.
\end{ack}

\begin{funding}
    Much of the work on this paper was done when I was at the Max Planck Institute for Mathematics in Bonn, and I would like to thank them for their hospitality and support. The first version of this article was completed while I was in residence at Institut Mittag-Leffler in Djursholm, Sweden in early 2022 as a participant of the research program `Higher algebraic structures in algebra, topology and geometry,' supported by the Swedish Research Council under grant no. 2016-06596.
\end{funding}

\bibliographystyle{emss}
\bibliography{literature.bib}
\end{document}